\documentclass[12pt]{l4dc2021}

\usepackage[utf8]{inputenc} % allow utf-8 input
\usepackage[T1]{fontenc}    % use 8-bit T1 fonts
\usepackage{hyperref}       % hyperlinks
\usepackage{url}            % simple URL typesetting
\usepackage{booktabs}       % professional-quality tables
\usepackage{amsfonts}       % blackboard math symbols
\usepackage{nicefrac}       % compact symbols for 1/2, etc.
\usepackage{microtype}      % microtypography

\usepackage{graphicx}
\usepackage{subcaption}
\usepackage{amsmath}
\usepackage{amssymb}
\usepackage{mathrsfs}
\usepackage{color,colortbl}
\usepackage{xcolor}
\usepackage{cite}
\newtheorem{prop}{Proposition}
\newtheorem{assumption}{Assumption}

\newtheorem{manualtheoreminner}{Theorem}
\newenvironment{manualtheorem}[1]{%
  \renewcommand\themanualtheoreminner{#1}%
  \manualtheoreminner
}{\endmanualtheoreminner}

\newtheorem{manuallemmainner}{Lemma}
\newenvironment{manuallemma}[1]{%
  \renewcommand\themanuallemmainner{#1}%
  \manuallemmainner
}{\endmanuallemmainner}

\newcommand{\N}{\mathcal{N}}
\usepackage{wrapfig}

\usepackage{tikz}
\usetikzlibrary{shapes,arrows}
\usetikzlibrary{positioning}
\tikzstyle{block} = [draw, fill=white, rectangle, 
    minimum height=3em, minimum width=3em]
\tikzstyle{circ} = [draw, fill=white, circle, minimum size=2.5em]
\tikzstyle{input} = [coordinate]
\tikzstyle{output} = [coordinate]
\tikzstyle{pinstyle} = [pin edge={to-,thin,black}]

\usepackage{algorithm}
\usepackage{algorithmic}
\graphicspath{ {images/} }
\usepackage{times}

\title[Accelerated Learning with Robustness to Adversarial Regressors]{Accelerated Learning with Robustness to Adversarial Regressors}

\author{%
 \Name{Joseph E. Gaudio}
 \Email{jegaudio@mit.edu}\\
 \Name{Anuradha M. Annaswamy} \Email{aanna@mit.edu}\\
 \Name{Jos\'{e} M. Moreu} \Email{jmmoreu@mit.edu}\\
 \addr Massachusetts Institute of Technology
 \AND
 \Name{Michael A. Bolender} \Email{michael.bolender@us.af.mil}\\
 \addr Air Force Research Laboratory
 \AND
 \Name{Travis E. Gibson} \Email{tegibson@bwh.harvard.edu}\\
 \addr Brigham and Women’s Hospital and Harvard Medical School%
}

\begin{document}

\maketitle

\begin{abstract}
    High order momentum-based parameter update algorithms have seen widespread applications in training machine learning models. Recently, connections with variational approaches have led to the derivation of new learning algorithms with accelerated learning guarantees. Such methods however, have only considered the case of static regressors. There is a significant need for parameter update algorithms which can be proven stable in the presence of adversarial time-varying regressors, as is commonplace in control theory. In this paper, we propose a new discrete time algorithm which 1) provides stability and asymptotic convergence guarantees in the presence of adversarial regressors by leveraging insights from \emph{adaptive control theory} and 2) provides non-asymptotic accelerated learning guarantees leveraging insights from convex optimization. In particular, our algorithm reaches an $\epsilon$ sub-optimal point in at most $\tilde{\mathcal{O}}(1/\sqrt{\epsilon})$ iterations when regressors are constant - matching lower bounds due to Nesterov of $\Omega(1/\sqrt{\epsilon})$, up to a $\log(1/\epsilon)$ factor and provides guaranteed bounds for stability when regressors are time-varying. We provide numerical experiments for a variant of Nesterov's provably hard convex optimization problem with time-varying regressors, as well as the problem of recovering an image with a time-varying blur and noise using streaming data.
\end{abstract}

\section{Introduction}
\label{s:Introduction_N}

Iterative gradient-based optimization methods in machine learning commonly employ a combination of time-scheduled learning rates \citep{Shalev_Shwartz_2011,Hazan_2016}, adaptive learning rates \citep{Duchi_2011,Kingma_2017,Wilson_2017}, and/or higher order ``momentum'' based dynamics \citep{Polyak_1964,Nesterov_1983,Wibisono_2016}. Variants of the higher order update proposed by Nesterov \citep{Nesterov_1983} in particular, have received significant attention in the optimization \citep{Nesterov_2004,Beck_2009,Bubeck_2015,Carmon_2018,Nesterov_2018} and neural network communities \citep{Krizhevsky_2012,Sutskever_2013} due to their provable guarantees of accelerated learning for classes of convex functions. Empirical investigations for non-convex neural network training are also a topic of significant interest.

To gain insight into Nesterov's discrete time method \citep{Nesterov_1983}, the authors in \citep{Su_2016} identified the second order ordinary differential equation (ODE) at the limit of zero step size. Still pushing further in the continuous time analysis of these higher order methods, several recent results have leveraged a variational approach showing that a larger class of higher order methods exist where one can obtain an arbitrarily fast convergence rate \citep{Wibisono_2016,Wilson_2016}. An equivalent algorithm in discrete-time to these continuous-time results, which is implementable and has comparable convergence rates, is an active area of research \citep{Betancourt_2018,Wilson_2018,Shi_2019}. It should be noted that, quite often, in many of these papers \citep{Su_2016,Wibisono_2016,Wilson_2016,Betancourt_2018,Wilson_2018,Shi_2019}, the analysis is performed for static features/regressors, with any dynamic components arising only due to a recursive update of the parameters.

There are many machine learning applications and paradigms where the features or inputs are time-varying. Examples include multi-armed bandits \citep{Auer_1995,Auer_2002,Bubeck_2012}, adaptive-filtering \citep{Goodwin_1984,Widrow_1985,Haykin_2014}, and temporal-prediction tasks \citep{Dietterich_2002,Kuznetsov_2015,Hall_2015}, to name a few. In addition, many models can be trained via adversarial learning \citep{Shalev_Shwartz_2011,Ben_David_2009} which results in time-varying inputs during training \citep{Auer_1995,Cesa_Bianchi_2006}. Even if the application does not require time-varying inputs, the presence of large training data has necessitated online or stochastic training methods in several applications, bringing a dynamic component into the problem statement \citep{Goodfellow-et-al-2016,Shalev_Shwartz_2011,Cesa_Bianchi_2004,Bengio_2012,Jain_2018,Gitman_2019}. Online learning is another class of problems that requires an investigation of optimization \citep{Zinkevich_2003,Hazan_2007,Hazan_2008,Hazan_2016,Shalev_Shwartz_2011,Raginsky_2010}. Online learning has had particular success in the development of state of the art gradient methods for training large neural networks \citep{Duchi_2011,Kingma_2017}. 

Time variations in inputs become even more important in real-time applications, with potentially limited compute \citep{Jordan_2015}, and in an area of machine learning which has now come to be referred to as continual/lifelong learning \citep{Ben_David_2009,Chen_2018_LML,Thrun_1995,Thrun_1998,Parisi_2019}. Continual/lifelong learning algorithms must be robust to adversarial features/inputs in addition to shifts in the distribution of the incoming data \citep{Ben_David_2009,Lopez_2017}, i.e.~data is not necessarily independent and identically distributed from a fixed probability distribution \citep{Pentina_2015}. Such algorithms must also be able to incrementally learn for an indefinite amount of time without human intervention \citep{Silver_2013,Fei_2016,Chen_2018_LML}.\footnote{In the online learning setting, when minimizing regret \citep{Shalev_Shwartz_2011}, the learning rates decay over time.} These notions are further important in robotics \citep{Thrun_1995,Thrun_1998} and learning-based control theory \citep{Sastry_1989,Narendra2005,Goodwin_1984,Ioannou1996} due to the requirement of continuously running in such applications.

\begin{table}[t]
    \caption{Comparison of gradient-based methods for a class of time-varying convex functions.}
    \label{t:Comparison_Gradient_Methods}
    \centering
    \begin{tabular}{lccc}
        \toprule
        Algorithm & Equation & Constant Regressor & Time-Varying \\
         &  & \# Iterations & Regressor \\
        \midrule
        Gradient Descent Normalized & \eqref{e:GD_Goodwin_N} & $\mathcal{O}(1/\epsilon)$ & Stable \\
        Gradient Descent Fixed & \eqref{e:Gradient_Method} & $\mathcal{O}(1/\epsilon)$ & Unstable \\
        Nesterov Acceleration Varying & \eqref{e:Nesterov_Two_Convex_TV_beta} & $\mathcal{O}(1/\sqrt{\epsilon})$ & Unstable \\
        Nesterov Acceleration Fixed & \eqref{e:Nesterov_Two_Convex} & $\mathcal{O}(1/\sqrt{\epsilon}\cdot\log(1/\epsilon))$ & Unstable \\
        \rowcolor{lightgray!40}This Paper & Alg \ref{alg:HOT_R} & $\mathcal{O}(1/\sqrt{\epsilon}\cdot\log(1/\epsilon))$ & Stable \\
        \bottomrule
    \end{tabular}
\end{table}

This paper proposes a new discrete time algorithm for parameter updates that accommodates time-varying regressors and is of high order. This algorithm will be shown to achieve two objectives. The first objective is to demonstrate stability of this high-order parameter tuner in the presence of time-varying adversarial regressors. This is in contrast to many other iterative methods that cannot be proved to be stable in this setting. The second objective is to show an accelerated convergence rate when the regressors are constant. Our higher order tuner is based on  a novel discretization of a continuous time higher order learning parameter update, and differs from the variational perspective-based high order tuner proposed in \citep{Wibisono_2016} which leverages time-scheduled hyperparameters. Unlike many of the papers listed above, we do not assume the requirement of an \emph{a priori} bound in the time-varying regressor for stability of our algorithm, and directly deploy the gradient which utilizes the regressors. The non-asymptotic convergence rate, which corresponds to the second objective, will be shown to be a logarithm factor away from provable lower bounds due to Nesterov \citep{Nesterov_2018}, with comparable constant factors. 

Table \ref{t:Comparison_Gradient_Methods} provides an overview of how these two objectives are realized with this algorithm in comparison to other iterative methods. The first objective ensures that our iterative algorithm remains stable and learns indefinitely as streaming data changes, even in an adversarial manner - crucial in learning for dynamical systems applications. The second objective demonstrates that the proposed learning algorithm retains fast convergence in the standard setting of constant regressors. That is, the significant benefit of provable stability of our algorithm in the presence of time-varying regressors does not come at the expense of large degradation in the rate of convergence - our proposed algorithm has a near-optimal convergence rate in the standard static regressor analysis setting as well.

The main contributions of this work are summarized as: (i) A new class of momentum/Nesterov-type iterative optimization algorithms, (ii) Accelerated learning guarantee a logarithm factor away from Nesterov (while remaining stable), (iii) Explicit stability conditions for the new algorithms with adversarial regressors, (iv) Connections to a Lagrangian variational perspective alongside an introduction of an adaptive systems-based normalization in machine learning, and (v) Numerical simulations demonstrating the efficacy of the proposed methods.

\section{Problem setting}
\label{s:Problem_Setting_N}

In this paper, we present the continuous time perspective with time $t$ while discrete time steps are indexed by $k$. When in continuous time, the time dependence of variables may be omitted when it is clear from the context. The classes $\mathcal{L}_p$ and $\ell_p$ for $p\in[1,\infty]$ are described in \citep[Appendix \ref{s:Preliminaries_N}]{Gaudio_arXiv_2020}, alongside definitions of (strong) convexity, smoothness, and Euler discretization techniques. Unless otherwise specified, $\lVert\cdot\rVert$ represents the 2-norm. We denote the discrete time difference of a function $V$ as $\Delta V_k:=V_{k+1}-V_k$. For notational clarity and to focus on multidimensional parameters/regressors we present the single output setting. The results of this paper trivially extend to multiple outputs.

We consider the setting of linear regression with time-varying regressors $\phi\in\mathbb{R}^N$ which are related in a linear combination with an unknown parameter $\theta^*\in\mathbb{R}^N$ to the output $y\in\mathbb{R}$ as $y_k=\theta^{*T}\phi_k$. Given that the parameter $\theta^*$ is unknown, an estimator $\hat{y}_k=\theta^T_k\phi_k$ is formulated, where $\hat{y}\in\mathbb{R}$ is the output estimate and $\theta\in\mathbb{R}^N$ is the parameter estimate. In this setting, the output error is defined as
\begin{equation}\label{e:error1_N_discrete_N}
    e_{y,k}=\hat{y}_k-y_k=\tilde{\theta}^T_k\phi_k,
\end{equation}
where $\tilde{\theta}_k=\theta_k-\theta^*$ is the parameter estimation error. The goal is to design an iterative algorithm to adjust the parameter estimate $\theta$ using streaming regressor-output data pairs $\mathcal{D}ata_k=(\phi_k,y_k)$ such that the prediction error $e_y$ converges to zero with a provably fast non-asymptotic convergence rate when regressors $\phi_k$ are constant, and that stability and asymptotic convergence properties remain in the presence of time-varying regressors. An iterative gradient-based method is proposed to enable computational simplicity and accommodate data in real-time. To formulate the gradient-based methods of this paper, we consider the squared loss function using \eqref{e:error1_N_discrete_N} of the form
\begin{equation}\label{e:Squared_Loss_N}
    L_k(\theta_k)=\frac{1}{2}e_{y,k}^2=\frac{1}{2}\tilde{\theta}_k^T\phi_k\phi_k^T\tilde{\theta}_k,
\end{equation}
where the subscript $k$ in $L_k$ denotes the regressor iteration number. At each iteration $k$, the gradient of the loss function is implementable as $\nabla L_k(\theta_k)=\phi_ke_{y,k}$. The Hessian of \eqref{e:Squared_Loss_N} can be expressed as $\nabla^2L_k(\theta_k)=\phi_k\phi_k^T$, and thus $0\leq\nabla^2L_k(\theta_k)\leq\lVert\phi_k\rVert^2I$. Therefore, the loss function can be seen to be (non-strongly) convex with a time-varying regressor-dependent smoothness parameter.
\begin{remark}
    The stability results of this paper will be shown to hold even for \textbf{adversarial} time-varying regressors $\phi_k$. No bound on $\phi_k$ is required to be known and the prediction error $e_{y,k}$ is not assumed to be bounded a priori. This is in comparison to standard methods in online learning which assume knowledge of a bound on gradients and regressors for proving stability \citep{Shalev_Shwartz_2011,Hazan_2016}. Thus the algorithm proposed in this paper can be employed in the continual learning \citep{Ben_David_2009,Thrun_1995,Thrun_1998} and learning-based control theory \citep{Sastry_1989,Narendra2005,Goodwin_1984,Ioannou1996} settings where such assumptions of a priori boundedness cannot be made.
\end{remark}
The starting point for our proposed algorithm comes from adaptive methods (see for example, \citep[Ch. 3]{Goodwin_1984}) which leads to an iterative normalized gradient descent method
\begin{equation}\label{e:GD_Goodwin_N}
    \theta_{k+1}=\theta_k-\gamma \nabla \bar f_k(\theta_k),\quad 0<\gamma<2,
\end{equation}
where $\bar{f}_k(\cdot)$ corresponds to a normalized loss function defined as
\begin{equation}\label{e:normalize}
\bar f_k(\theta_k)=\frac{L_k(\theta_k)}{\N_k},
\end{equation}
and $\N_k=1+\lVert\phi_k\rVert^2$ is a normalization signal employed to ensure boundedness of signals for any arbitrary regressor $\phi_k$. Motivated by the normalized gradient method in \eqref{e:GD_Goodwin_N}, our goal is to derive a Nesterov-type higher order gradient method to ensure a provably faster convergence rate when regressors are constant while preserving stability of the estimation algorithm in the presence of adversarial time-varying $\phi_k$.

\section{Algorithms based on a higher order tuner}
\label{s:Derivations}

We begin the derivation of our discrete time Nesterov-type higher order tuner algorithm from the continuous time perspective, which provides insights into the underlying stability structure. This perspective further results in a representation of a second order differential equation as two first order differential equations which are used to certify stability in the presence of time-varying regressors by employing Lyapunov function techniques. Motivated by this continuous time representation, we provide a novel discretization to result in a discrete time higher order tuner which can be shown to be stable using the same Lyapunov function. The discrete time algorithm is then shown to be equivalent to the common Nesterov iterative method form when regressors are constant.

\subsection{Continuous time higher order tuner}

We begin the derivation of our algorithm using the variational perspective of Wibisono, Wilson, and Jordan \citep{Wibisono_2016}. In particular, the Bregman Lagrangian in \citep[Eq. 1]{Wibisono_2016} is re-stated with the Euclidean norm employed in the Bregman divergence as $\mathcal{L}(\theta(t),\dot{\theta}(t),t)=\text{e}^{\bar{\alpha}_t+\bar{\gamma}_t}\left(\text{e}^{-2\bar{\alpha}_t}\frac{1}{2}\lVert\dot{\theta}(t)\rVert^2-\text{e}^{\bar{\beta}_t}L_t(\theta(t))\right)$. This Lagrangian weights potential energy (loss) $L_t(\theta(t))$, and kinetic energy $(1/2)\lVert\dot{\theta}(t)\rVert^2$, with an exponential term $\exp(\bar{\alpha}_t+\bar{\gamma}_t)$, which adjusts the damping. The hyperparameters $(\bar{\alpha}_t,\bar{\beta}_t,\bar{\gamma}_t)$ are commonly time-scheduled and result in different algorithms by appropriately weighting each component in the Lagrangian (see \citep{Wibisono_2016} for choices common in optimization for machine learning). It can be easily shown however, that time scheduling the hyperparameters can result in instability when regressors are time-varying. We thus propose the use of a regressor-based normalization $\N_t=1+\lVert\phi(t)\rVert^2$ with constant gains $\gamma,\beta>0$ to parameterize the Lagrangian as
\begin{equation}
    \label{e:Lagrangian_N_N}
    \mathcal{L}(\theta(t),\dot{\theta}(t),t)=\text{e}^{\beta (t-t_0)}\left(\frac{1}{2}\lVert\dot{\theta}(t)\rVert^2-\frac{\gamma\beta}{\N_t} L_t(\theta(t))\right).
\end{equation}
Using a Lagrangian, a functional may be defined as: $J(\theta)=\int_{\mathbb{T}}\mathcal{L}(\theta,\dot{\theta},t)dt$, where $\mathbb{T}$ is an interval of time. To minimize this functional, a necessary condition from the calculus of variations \citep{Goldstein_2002} is that the Lagrangian solves the Euler-Lagrange equation: $\frac{d}{dt}\left(\frac{\partial\mathcal{L}}{\partial\dot{\theta}}(\theta,\dot{\theta},t)\right)=\frac{\partial\mathcal{L}}{\partial\theta}(\theta,\dot{\theta},t)$. Using \eqref{e:Lagrangian_N_N}, the second order differential equation resulting from the application of the Euler-Lagrange equation is: $\ddot{\theta}(t)+\beta\dot{\theta}(t)=-\frac{\gamma\beta}{\N_t}\nabla L_t(\theta(t))$. This differential equation can be seen to have the normalized gradient of the loss function as the forcing term parameterized with $\gamma\beta$, and constant damping parameterized with $\beta$. Crucial to the development of the results of this paper, this second order differential equation may be written as a \emph{higher order tuner} given by
\begin{align}\label{e:Accelerated_GF_N_N}
    \begin{split}
    \dot{\vartheta}(t)&=-\frac{\gamma}{\N_t}\nabla L_t(\theta(t)),\\
    \dot{\theta}(t)&=-\beta(\theta(t)-\vartheta(t)),
    \end{split}
\end{align}
which can be seen to take the form of a normalized gradient flow update followed by a linear time invariant (LTI) filter. This representation of a higher order tuner will be fundamental to prove stability with time-varying regressors using Lyapunov function techniques in Section \ref{s:Stability_N}.

\subsection{Discretization of continuous time higher order tuner}

We propose in this paper a specific discretization of the high-order tuner in \eqref{e:Accelerated_GF_N_N}, of the form
\begin{align}\label{e:higher_order_L_Discretized}
    \begin{split}
        \text{Implicit Euler}:\vartheta_{k+1}&=\vartheta_k-\gamma\nabla \bar{f}_k(\theta_{k+1}),\\
        \text{Explicit Euler}:\hspace{.04cm}\theta_{k+1}&=\bar{\theta}_k-\beta(\bar{\theta}_k-\vartheta_k),\\
        \text{Extra Gradient}:\hspace{.385cm}\bar{\theta}_k&=\theta_k-\gamma\beta\nabla \bar{f}_k(\theta_k),
    \end{split}
\end{align}
where $\bar{f}_k(\cdot)$ is given by \eqref{e:normalize}, and the hyperparameters are $\gamma$ and $\beta$. One can employ any number of techniques for the discretization of an ordinary differential equation, including Runge–Kutta, symplectic, and Euler methods (see for example methods in \citep{Hairer_2006,Betancourt_2018}). The one employed in \eqref{e:higher_order_L_Discretized} can be viewed as a combination of implicit-Euler (for the variable $\vartheta$) and explicit-Euler method (for the variable $\theta$). An important correction is introduced in the explicit-Euler component, which corresponds to the use of an extra gradient. It should also be noted that the extra gradient step only serves to adjust the direction of the update, but does not increase the order of the tuner beyond two.

\subsection{Augmented objective function}

\begin{wrapfigure}{r}{0.50\textwidth}
    \begin{minipage}{0.50\textwidth}
        \vspace{-0.8cm}
        \begin{algorithm}[H]
        \caption{Higher Order Tuner Optimizer}
        \label{alg:HOT_R}
        \begin{algorithmic}[1]
        \STATE {\bfseries Input:} initial conditions $\theta_0$, $\vartheta_0$, gains $\gamma$, $\beta$, $\mu$
        \FOR{$k=0,1,2,\ldots$}
        \STATE \textbf{Receive} regressor $\phi_k$, output $y_k$
        \STATE Let $\N_k=1+\lVert\phi_k\rVert^2$\\%, $\nabla L_k(\theta_k)=\phi_k(\theta_k^T\phi_k-y_k)$,\\
        $\nabla f_k(\theta_k)=\frac{\nabla L_k(\theta_k)}{\N_k}+\mu(\theta_k-\theta_0)$,\\ $\bar{\theta}_k=\theta_k-\gamma\beta\nabla f_k(\theta_k)$
        \STATE $\theta_{k+1}\leftarrow\bar{\theta}_k-\beta(\bar{\theta}_k-\vartheta_k)$
        \STATE Let \\%$\nabla L_k(\theta_{k+1})=\phi_k(\theta_{k+1}^T\phi_k-y_k)$,\\
        $\nabla f_k(\theta_{k+1})=\frac{\nabla L_k(\theta_{k+1})}{\N_k}+\mu(\theta_{k+1}-\theta_0)$
        \STATE $\vartheta_{k+1}\leftarrow\vartheta_k-\gamma\nabla f_k(\theta_{k+1})$
        \ENDFOR
        \end{algorithmic}
        \end{algorithm}
    \end{minipage}
    \vspace{-0.1cm}
\end{wrapfigure}
As discussed in Section \ref{s:Problem_Setting_N}, the squared error loss function in \eqref{e:Squared_Loss_N} as well as its normalized version in \eqref{e:normalize} are non-strongly convex. In order to obtain accelerated learning properties similar to that of Nesterov's in \citep{Nesterov_2018}, we now propose a new algorithm that builds on that in \eqref{e:higher_order_L_Discretized}. For this purpose, we modify the normalized cost function in \eqref{e:normalize} to include L2 regularization as
\begin{equation}\label{e:Strongly_Convex_Objective}
    f_k(\theta_k)=\bar{f}_k(\theta_k)+\frac{\mu}{2}\lVert\theta_k-\theta_0\rVert^2,
\end{equation}
where $\mu>0$ is the regularization constant, $\theta_0$ is the initial condition of the parameter estimate, and the subscript $k$ in $f_k$ denotes the regressor iteration number. Using \eqref{e:Squared_Loss_N}, the Hessian of \eqref{e:Strongly_Convex_Objective} can be expressed as $\nabla^2f_k(\theta_k)=(\phi_k\phi_k^T)/\N_k+\mu I$, and thus it can be seen that $\mu I\leq\nabla^2f_k(\theta_k)\leq (1+\mu)I$. Therefore, the objective function in \eqref{e:Strongly_Convex_Objective} can be seen to be $\mu$-strongly convex and $(1+\mu)$-smooth and has desirable properties of constant smoothness and strong convexity.

By replacing the objective function $\bar{f}_k(\theta_k)$ in \eqref{e:higher_order_L_Discretized} with $f_k(\theta_k)$ in \eqref{e:Strongly_Convex_Objective}, we now obtain Algorithm \ref{alg:HOT_R}, the main higher order tuner optimizer introduced in this paper with hyperparameters $\gamma$, $\beta$, and $\mu$. It should be noted that the gradient expressions for $L_k$ in lines 4 and 6 of Algorithm \ref{alg:HOT_R} are given by $\nabla L_k(\theta_{k})=\phi_k(\theta_{k}^T\phi_k-y_k)$ and $\nabla L_k(\theta_{k+1})=\phi_k(\theta_{k+1}^T\phi_k-y_k)$ respectively. The stability properties of Algorithm \ref{alg:HOT_R}, in the presence of a time-varying regressor $\phi_k$, will be demonstrated in Section \ref{s:Stability_N}. In addition to the stability properties in the presence of $\phi_k$, the additional advantage of Algorithm \ref{alg:HOT_R} is a fast minimization of \eqref{e:Squared_Loss_N} for constant regressors. This accelerated convergence property will be established in Section \ref{s:Acceleration} by minimizing the augmented objective in \eqref{e:Strongly_Convex_Objective}. The relation between Algorithm \ref{alg:HOT_R} and Nesterov's method is summarized in the  following proposition.
\begin{prop}\label{prop:equivalence}
    Algorithm \ref{alg:HOT_R} with a constant regressor $\phi_k\equiv\phi$ (and thus $f_k(\cdot)\equiv f(\cdot)$) may be reduced to the common form of Nesterov's equations \citep[Eq. 2.2.22]{Nesterov_2018} with $\bar{\beta}=1-\beta$ and $\bar{\alpha}=\gamma\beta$ as
\begin{align}\label{e:Nesterov_Two_Convex}
    \begin{split}
        \theta_{k+1}&=\nu_k-\bar{\alpha}\nabla f(\nu_k),\\
        \nu_{k+1}&=\left(1+\bar{\beta}\right)\theta_{k+1} - \bar{\beta}\theta_{k}.
    \end{split}
\end{align}
\end{prop}

\begin{remark}
    It is apparent from Proposition \ref{prop:equivalence} that Algorithm \ref{alg:HOT_R} is Nesterov-type; it includes both averaging outside of the gradient evaluation, and an "extra-gradient" step to enable adjustments inside the gradient evaluation. The presence of both of these ingredients makes our Algorithm \ref{alg:HOT_R} similar to Nesterov-type rather than Heavy-ball type \citep{Polyak_1964} which only includes the first. 
\end{remark}
\begin{remark}
    Similar to the derivation of Algorithm \ref{alg:HOT_R}, which is Nesterov-type, we can derive another algorithm that is Heavy-Ball type \citep{Polyak_1964}. We accomplish this in \citep[Appendix \ref{ss:Discrete_time_setting}]{Gaudio_arXiv_2020} and denote it as Algorithm \ref{alg:HOT_R_HB}. Algorithm \ref{alg:HOT_R_HB} only includes the implicit-explicit mix and not the extra gradient step. As is shown in \citep[Appendix \ref{s:Stability_N_Proofs}]{Gaudio_arXiv_2020}, Algorithm \ref{alg:HOT_R_HB} is also stable in the presence of time-varying regressors. The disadvantage of Algorithm \ref{alg:HOT_R_HB} over Algorithm \ref{alg:HOT_R} is simply due to the well known point of Heavy-ball type methods in comparison to Nesterov-type methods (c.f. \citep{Lessard_2016} for a clear example). This is the reason for our preference of Algorithm \ref{alg:HOT_R} over Algorithm \ref{alg:HOT_R_HB}.
\end{remark}

\section{Stability and asymptotic convergence}
\label{s:Stability_N}

In this section, we state the main results of stability and asymptotic convergence in the presence of time-varying regressors for the continuous time higher order tuner \eqref{e:Accelerated_GF_N_N}, discretized equations \eqref{e:higher_order_L_Discretized}, and the main stability result of Algorithm \ref{alg:HOT_R}. Proofs of all theorems and corollaries in this section are provided in \citep[Appendix \ref{s:Stability_N_Proofs}]{Gaudio_arXiv_2020} alongside more in-depth auxiliary results of stability. For completeness, complementary stability proofs of the normalized gradient method \eqref{e:GD_Goodwin_N} in both continuous and discrete time are additionally provided in \citep[Appendix \ref{s:Stability_N_Proofs}]{Gaudio_arXiv_2020}. We begin with the discussion of stability of the discretized equations in \eqref{e:higher_order_L_Discretized} (Algorithm \ref{alg:HOT_R} with $\mu=0$), in the following theorem.
\begin{theorem}\label{th:HOT_paper_Discrete}
    For the linear regression error model in \eqref{e:error1_N_discrete_N} with loss in \eqref{e:Squared_Loss_N}, with Algorithm \ref{alg:HOT_R} and its hyperparameters chosen as $\mu=0$, $0<\beta<1$, $0<\gamma\leq\frac{\beta(2-\beta)}{16+\beta^2}$, the following
    \begin{equation}\label{e:Lyap_2_discrete}
        V_k=\frac{1}{\gamma}\lVert \vartheta_k-\theta^*\rVert^2+\frac{1}{\gamma}\lVert \theta_k-\vartheta_k\rVert^2,
    \end{equation}
    is a Lyapunov function with increment $\Delta V_k\leq -\frac{L_k(\theta_{k+1})}{\N_k}\leq0$. It can also be shown that $V\in\ell_{\infty}$, and $\sqrt{\frac{L_k(\theta_{k+1})}{\N_k}}\in\ell_2\cap\ell_{\infty}$. If in addition it is assumed that $\phi\in\ell_{\infty}$ then $\lim_{k\rightarrow\infty}L_k(\theta_{k+1})=0$.
\end{theorem}
We now proceed to the main stability theorem of Algorithm \ref{alg:HOT_R} with $\mu\neq0$.
\begin{theorem}\label{th:HOT_paper_Full_Alg}
    For the linear regression error model in \eqref{e:error1_N_discrete_N} with loss in \eqref{e:Squared_Loss_N}, with Algorithm \ref{alg:HOT_R} and its hyperparameters chosen as $0<\mu<1$, $0<\beta<1$, $0<\gamma\leq \frac{\beta(2-\beta)}{16+\beta^2+\mu\left(\frac{57\beta+1}{16\beta}\right)}$, the function $V$ in \eqref{e:Lyap_2_discrete} can be shown to have increment $\Delta V_k\leq-\frac{L_k(\theta_{k+1})}{\N_k}-\mu c_1V_k+\mu c_2$, for constants $0<c_1<1$, $c_2>0$ (given in \citep[Appendix \ref{s:Stability_N_Proofs}]{Gaudio_arXiv_2020}). It can also be shown that $\Delta V_k<0$ outside of the compact set $D=\left\{V\middle|V\leq\frac{c_2}{c_1}\right\}$. Furthermore, $V\in\ell_{\infty}$ and $V_k\leq \exp(-\mu c_1k)\left(V_0-\frac{c_2}{c_1}\right)+\frac{c_2}{c_1}$.
\end{theorem}
The Lyapunov function in \eqref{e:Lyap_2_discrete} which is employed in Theorems \ref{th:HOT_paper_Discrete} and \ref{th:HOT_paper_Full_Alg} was originally motivated by the continuous time higher order tuners in \citep{Morse_1992,Evesque_2003}. The continuous time equivalent of \eqref{e:Lyap_2_discrete} is used in the following theorem to prove stability and asymptotic convergence properties for the continuous time higher order tuner in \eqref{e:Accelerated_GF_N_N}.
\begin{theorem}\label{th:HOT_paper}
    For continuous time equivalents to the linear regression model in \eqref{e:error1_N_discrete_N} with loss in \eqref{e:Squared_Loss_N} (concretely, \eqref{e:error1_N} and \eqref{e:error1_N_Loss} in \citep[Appendix \ref{ss:Continuous_time_setting}]{Gaudio_arXiv_2020}), for the higher order tuner update in \eqref{e:Accelerated_GF_N_N} with $\beta>0$, $0<\gamma\leq\beta/2$, the following
    \begin{equation}\label{e:Lyap_2_continuous}
        V(t)=\frac{1}{\gamma}\lVert\vartheta(t)-\theta^*\rVert^2+\frac{1}{\gamma}\lVert\theta(t)-\vartheta(t)\rVert^2,
    \end{equation}
    is a Lyapunov function with time derivative $\dot{V}(t)\leq-\frac{L_t(\theta(t))}{\N_t}\leq0$. It can be shown that $V\in\mathcal{L}_{\infty}$ and $\sqrt{\frac{L_t(\theta(t))}{\N_t}}\in\mathcal{L}_2\cap\mathcal{L}_{\infty}$. If in addition it assumed that $\phi,\dot{\phi} \in \mathcal{L}_{\infty}$ then $\lim_{t\rightarrow\infty}L_t(\theta(t))=0$.
\end{theorem}
\begin{remark}
    The same function $V$ is employed throughout, as motivated by the continuous time higher order tuner in Theorem \ref{th:HOT_paper}. Note that the proofs of stability in the presence of adversarial time-varying regressors are enabled as the Lyapunov functions in \eqref{e:Lyap_2_discrete} and \eqref{e:Lyap_2_continuous} do not contain the regressor. In both the continuous and discrete time analyses, stability is proven by showing that the provided function $V$ does not increase globally (Theorems \ref{th:HOT_paper_Discrete} and \ref{th:HOT_paper}) or at least does not increase outside a compact set containing the origin (Theorem \ref{th:HOT_paper_Full_Alg}).
\end{remark}

\section{Non-asymptotic accelerated convergence rates with constant regressors}
\label{s:Acceleration}

In this section, we state the main accelerated non-asymptotic convergence rate result for Algorithm \ref{alg:HOT_R} for the case of constant regressors, $\phi_k\equiv\phi$. All proofs in this section are provided in \citep[Appendix \ref{s:NonAsymptoticProofs}]{Gaudio_arXiv_2020} alongside convergence rate proofs for first order gradient descent methods, and Nesterov's method with time-varying gains, as given in overview form in Table \ref{t:Comparison_Gradient_Methods}.

Given the constant smoothness and strong-convexity parameters of the augmented objective function in \eqref{e:Strongly_Convex_Objective}, the representation of Algorithm \ref{alg:HOT_R} as \eqref{e:Nesterov_Two_Convex} may be used to provide a non-asymptotic rate for \eqref{e:Strongly_Convex_Objective} in the following theorem due to Nesterov.
\begin{theorem}[Modified from \citep{Bubeck_2015,Nesterov_2018}]\label{theorem:Nesterov_SC_paper}
    For a $\bar{L}$-smooth and $\mu$-strongly convex function $f$, the iterates $\{\theta_k\}_{k=0}^{\infty}$ generated by \eqref{e:Nesterov_Two_Convex} with $\theta_0=\nu_0$, $\bar{\alpha}=1/\bar{L}$, $\kappa=\bar{L}/\mu$, and $\bar{\beta}=(\sqrt{\kappa}-1)/(\sqrt{\kappa}+1)$ satisfy $f(\theta_k)-f(\theta^*)\leq\frac{\bar{L}+\mu}{2}\lVert\theta_0-\theta^*\rVert^2\exp\left(-\frac{k}{\sqrt{\kappa}}\right)$.
\end{theorem}
Leveraging the accelerated convergence rate for the augmented function $f$ in \eqref{e:Nesterov_Two_Convex} as provided by Theorem \ref{theorem:Nesterov_SC_paper}, we provide the following new lemmas to give accelerated non-asymptotic convergence rates for the normalized and unnormalized versions of the loss function in \eqref{e:Squared_Loss_N}, as desired.
\begin{lemma}\label{l:L_N_rate}
    The iterates $\{\theta_k\}_{k=0}^{\infty}$ generated by \eqref{e:Nesterov_Two_Convex} for the function in \eqref{e:Strongly_Convex_Objective} with $\theta_0=\nu_0$, $\Psi\geq\max\{1,\lVert\theta_0-\theta^*\rVert^2\}$, $\mu=\epsilon/\Psi$, $\bar{L}=1+\mu$, $\bar{\alpha}=1/\bar{L}$, $\kappa=\bar{L}/\mu$, $\bar{\beta}=(\sqrt{\kappa}-1)/(\sqrt{\kappa}+1)$, if
    \begin{equation}\label{e:Desired_Rate_N}
        k\geq\left\lceil\sqrt{1+\frac{\Psi}{\epsilon}}\log\left(2+\frac{\Psi}{\epsilon}\right)\right\rceil,\text{ then $\frac{L(\theta_k)-L(\theta^*)}{\N}\leq\epsilon$}.
    \end{equation}
\end{lemma}
\begin{lemma}\label{l:L_rate}
    The iterates $\{\theta_k\}_{k=0}^{\infty}$ generated by \eqref{e:Nesterov_Two_Convex} for the function in \eqref{e:Strongly_Convex_Objective} with $\theta_0=\nu_0$, $\Psi\geq\max\{1,\N\lVert\theta_0-\theta^*\rVert^2\}$, $\mu=\epsilon/\Psi$, $\bar{L}=1+\mu$, $\bar{\alpha}=1/\bar{L}$, $\kappa=\bar{L}/\mu$, $\bar{\beta}=(\sqrt{\kappa}-1)/(\sqrt{\kappa}+1)$, if
    \begin{equation}\label{e:Desired_Rate}
        k\geq\left\lceil\sqrt{1+\frac{\Psi}{\epsilon}}\log\left(2+\frac{\Psi}{\epsilon}\right)\right\rceil,\text{ then $L(\theta_k)-L(\theta^*)\leq\epsilon$}.
    \end{equation}
\end{lemma}
\begin{remark}
    Lemmas \ref{l:L_N_rate} and \ref{l:L_rate} provide the provable number of iterations required to obtain an $\epsilon$ sub-optimal point of for the normalized and original loss function in \eqref{e:Squared_Loss_N} of $\mathcal{O}(1/\sqrt{\epsilon}\cdot\log(1/\epsilon))$, with all constants included. It can be noted that the constants are comparable to the constants for the gradient and Nesterov iterative methods shown in Table \ref{t:Comparison_Gradient_Methods_Constants} and Figure \ref{f:Comparison_Gradient_Methods_Constants} in \citep[Appendix \ref{s:NonAsymptoticProofs}]{Gaudio_arXiv_2020}.
\end{remark}
\begin{remark}
    It can be noted from both Lemmas \ref{l:L_N_rate} and \ref{l:L_rate} that the L2 regularization parameter $\mu$ in \eqref{e:Strongly_Convex_Objective} is smaller than the $\epsilon$ sub-optimality gap. The L2 regularization parameter is present to ensure strong convexity of the augmented objective function in \eqref{e:Strongly_Convex_Objective}, such that Algorithm \ref{alg:HOT_R} can be reduced to Nesterov's iterative method with constant gains in \eqref{e:Nesterov_Two_Convex}, which results in the convergence rate for the augmented function in Theorem \ref{theorem:Nesterov_SC_paper}, which in turn lends to Lemmas \ref{l:L_N_rate} and \ref{l:L_rate}.
\end{remark}

\section{Numerical experiments}
\label{s:Numerical_N}

In this section, we analyze and compare the performance of our proposed algorithm in two different numerical experiment settings: a variant of Nesterov's provably hard smooth convex optimization problem \citep[p.~69]{Nesterov_2018} and a variant of the image deblurring problem considered by Beck and Teboulle \citep{Beck_2009}. In each setting, we compare hyperparameters chosen in accordance with Theorem \ref{th:HOT_paper_Full_Alg} and hyperparameters chosen optimally as per the standard Nesterov iterative method. All simulations are implemented in Python code available at \underline{\href{https://colab.research.google.com/drive/14V1X4kyEU3G6NMZuCaswSHSLaywUbqEe}{link1}} and \underline{\href{https://colab.research.google.com/drive/1cQRk44jaMNXCtOvEZ-QWIeqzKZBujQQN}{link2}}. Videos demonstrating the real-time image deblurring results are furthermore \underline{\href{https://drive.google.com/drive/folders/1oQFFWuUox-ChbFeQFveGlLR2nim-8y0w?usp=sharing}{available}}.

\subsection{Nesterov's smooth convex function}\label{sec:SHP}

In this section, we consider a modified version of Nesterov's provably hard smooth convex problem \citep[p.~69]{Nesterov_2018} of the form $L_k(\theta)=\lVert \phi_k^T\theta\rVert^2+B^T\theta$. In the experiments, the regressor $\phi_k$ changes, resulting in a change in the smoothness parameter $\bar{L}$. This problem was selected to demonstrate a lower bound of $\mathcal{O}(1/\sqrt{\epsilon})$ for iterative methods with gradient information \citep[p.~69]{Nesterov_2018}. \citep[Appendix \ref{SmoothConvexMinProblem}]{Gaudio_arXiv_2020} provides a detailed description.

In Figure \ref{fig:SHPsimulations}, we compare two different experiments with hyperparameters selected in two ways, where the regressor $\phi_k$ changes at iteration $k=500$. For the first experiment in Figure \ref{fig:SHPconstant}, the hyperparameters for the high order tuner algorithm are chosen according to Theorem \ref{th:HOT_paper_Full_Alg}, and the hyperparameters of the other methods are chosen with the same step size and the momentum parameter as in Proposition \ref{prop:equivalence}. In Figure \ref{fig:SHPstep2}, the hyperparameters are chosen optimally in accordance with the Nesterov iterative method with $\beta=1-\Bar{\beta}$ and $\gamma=\Bar{\alpha}/\beta$, as per Proposition \ref{prop:equivalence}. For $k<500$, the convergence rate of the Higher Order Tuner algorithm can be seen to be comparable to the optimal Nesterov algorithm. After $\phi_k$ changes at iteration $k=500$, all unnormalized algorithms become unstable. Even in the presence of a change in the regressor, the Higher Order Tuner algorithm can be seen to have a fast rate of convergence as compared to the normalized gradient descent method.

\begin{figure}[ht]
\centering
\begin{subfigure}{0.527\textwidth}
\includegraphics[width=1\linewidth]{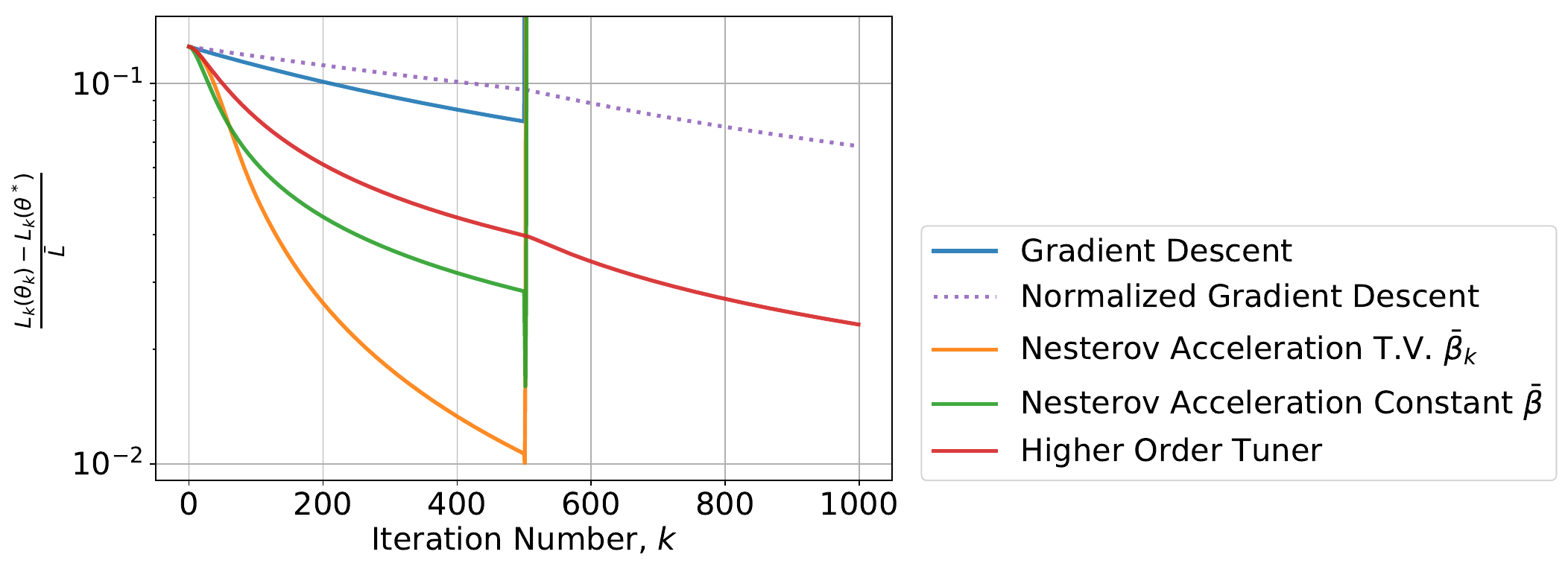}
\caption{}
\label{fig:SHPconstant}
\end{subfigure}
\begin{subfigure}{0.306\textwidth}
\includegraphics[width=1\linewidth]{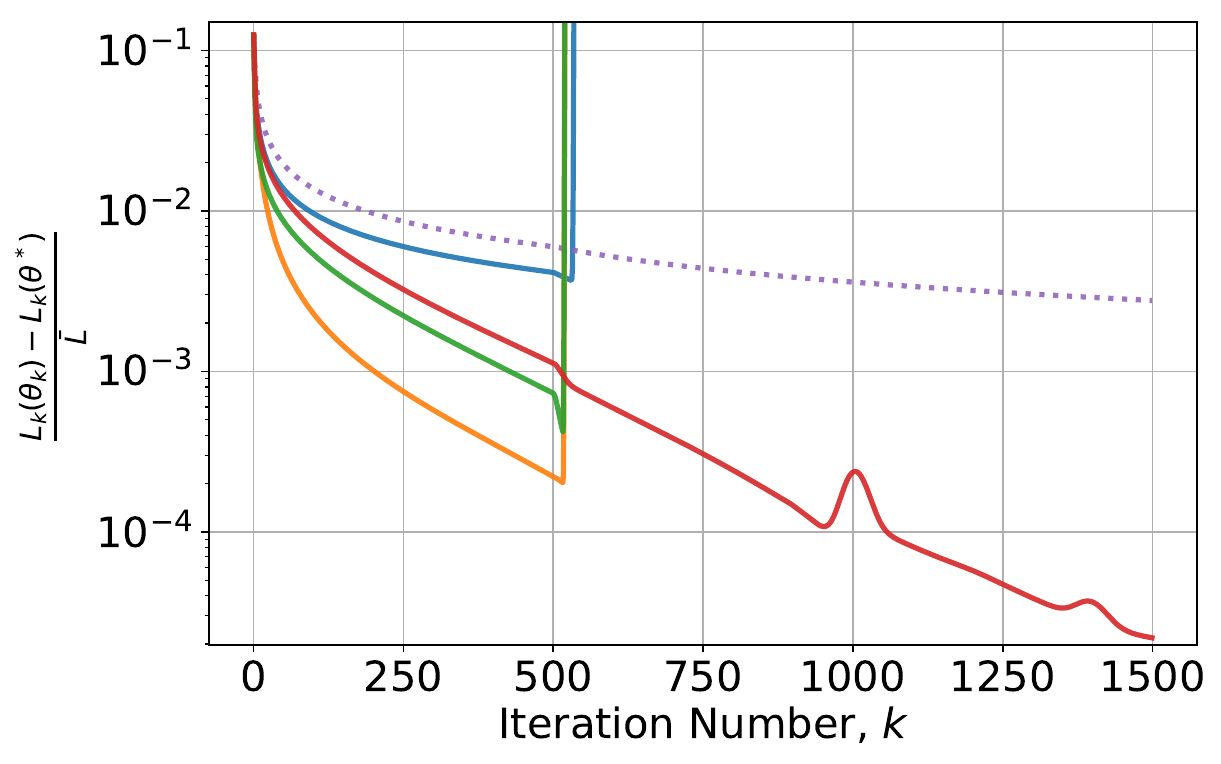}
\caption{}
\label{fig:SHPstep2}
\end{subfigure}
\caption{A variant of Nesterov's smooth convex function \citep[p.~69]{Nesterov_2018}. (a) $\mu=10^{-5}$, $\beta=0.1$ and $\gamma$ as in Theorem \ref{th:HOT_paper_Full_Alg}, $\Bar{\beta}=1-\beta$ and $\Bar{\alpha}=\gamma\beta$. At iteration $k=500$, step change in $\Bar{L}$ from $2$ to $8000$. (b) Hyperparameters chosen satisfying Lemma \ref{l:L_rate}  at iteration $k=0$ with $\epsilon=0.001$, $\beta=1-\Bar{\beta}$, and $\gamma=\Bar{\alpha}/\beta$. At iteration $k=500$, step change in $\Bar{L}$, from $2$ to $8$.}
\label{fig:SHPsimulations}
\end{figure}

\begin{figure}[ht]
        \centering
        \begin{subfigure}[b]{0.9\textwidth}
            \centering
            \includegraphics[width=\textwidth]{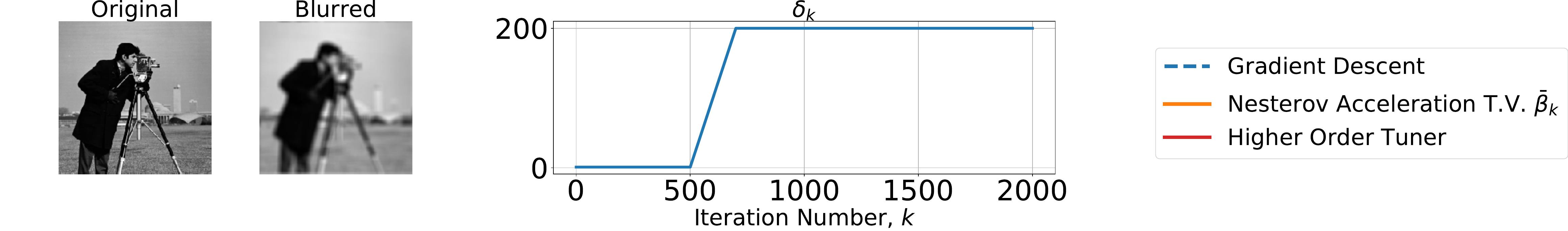}
            \caption{} 
            \label{fig:IDP_info}
        \end{subfigure}
        \vskip\baselineskip
        \begin{subfigure}[b]{0.245\textwidth}  
            \centering 
            \includegraphics[width=\textwidth]{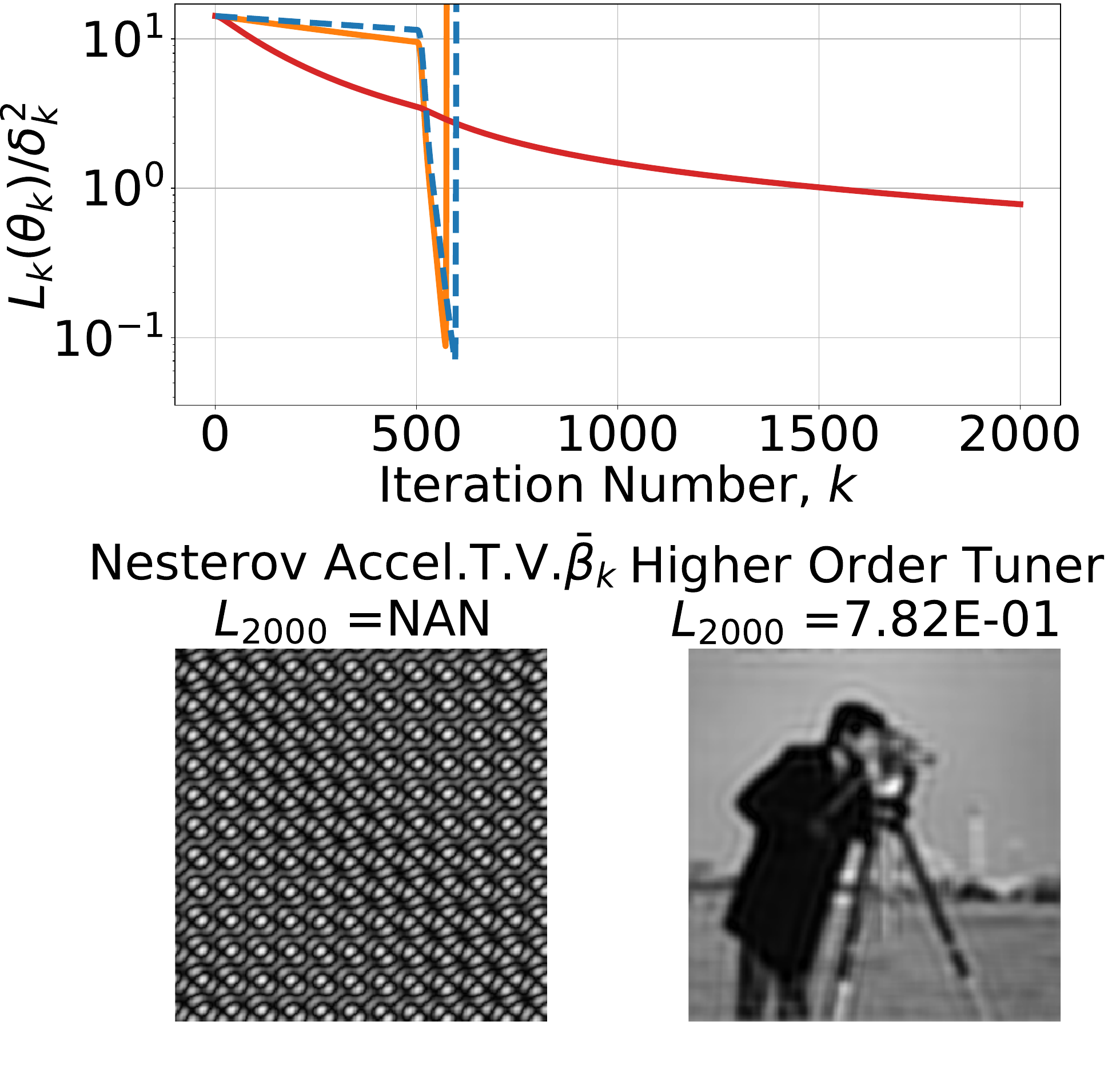}
            \caption{}  
            \label{fig:IDP_stable}
        \end{subfigure}
        \begin{subfigure}[b]{0.245\textwidth}  
            \centering 
            \includegraphics[width=\textwidth]{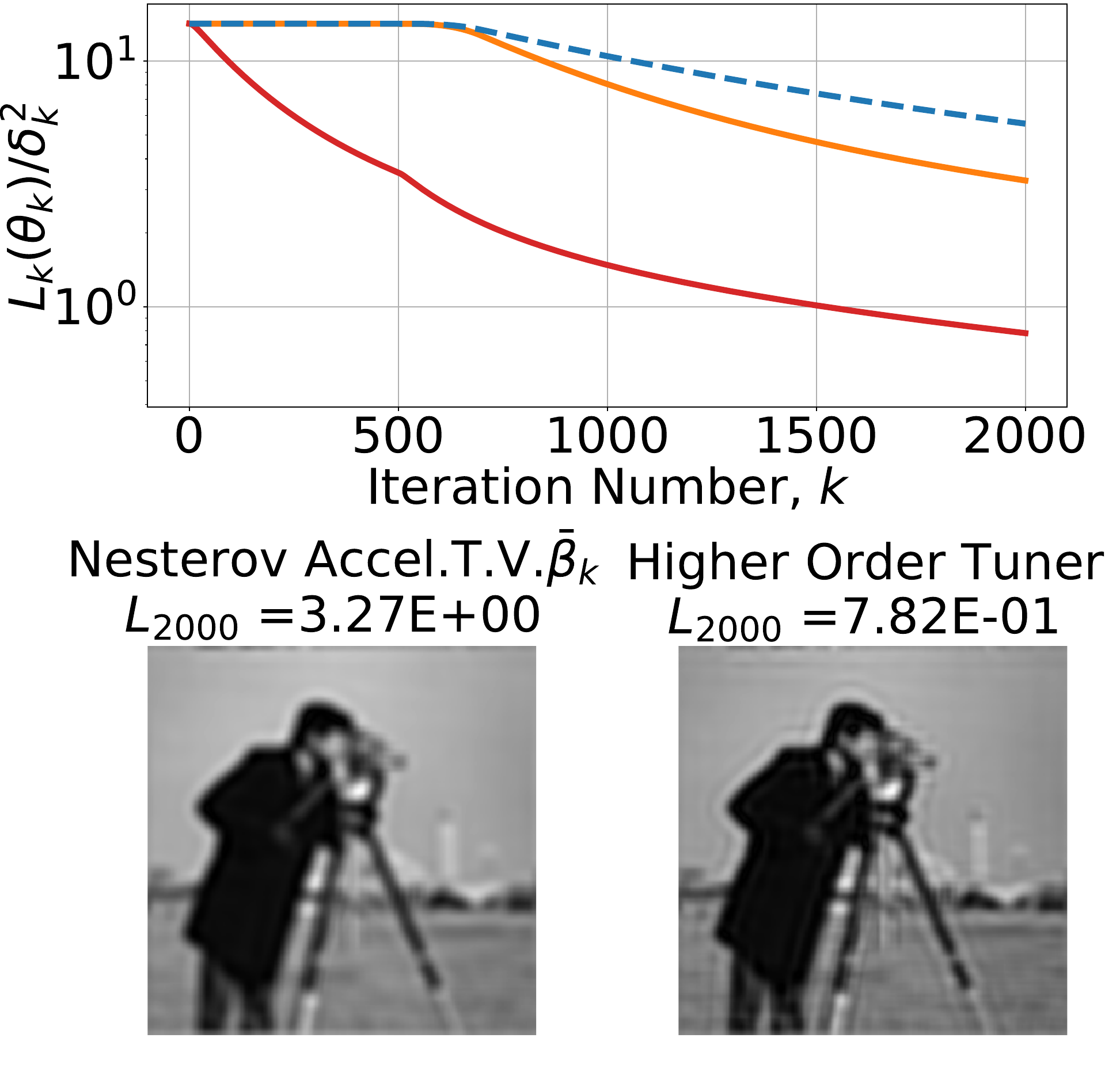}
            \caption{}  
            \label{fig:IDP_stable_SUP}
        \end{subfigure}
        \begin{subfigure}[b]{0.245\textwidth}   
            \centering 
            \includegraphics[width=\textwidth]{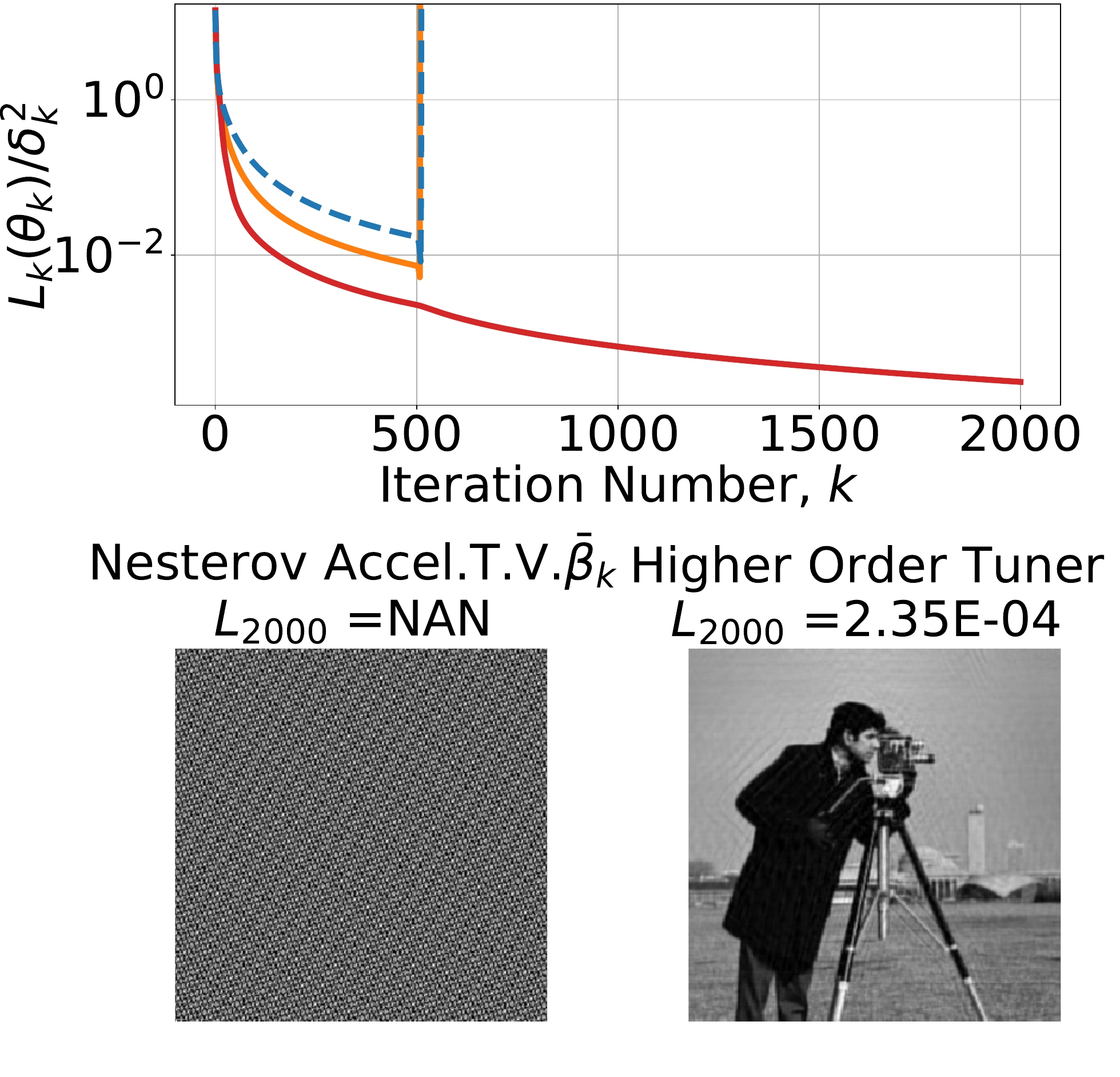}
            \caption{}  
            \label{fig:IDP_aggresive}
        \end{subfigure}
        \begin{subfigure}[b]{0.245\textwidth}   
            \centering 
            \includegraphics[width=\textwidth]{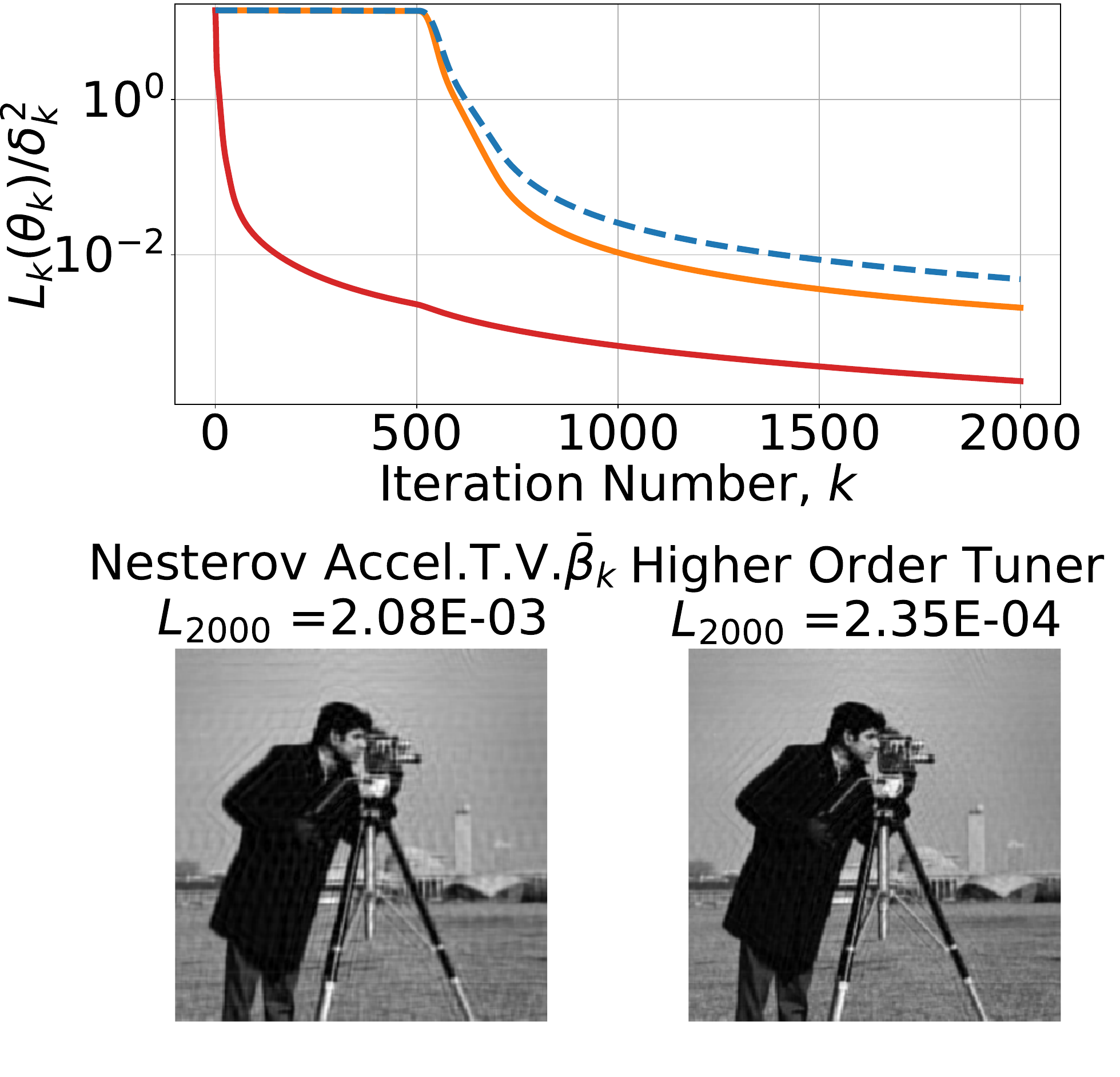}
            \caption{}
            \label{fig:IDP_aggresive_SUP}
        \end{subfigure}
        \caption{(a) Original and blurred images; Ramp increase of $\delta_k$ from $1$ to $200$ in $200$ iterations, starting at $k=500$. (b) and (c) Hyperparameters as in Theorem \ref{th:HOT_paper_Full_Alg}. (b) Loss values and reconstructed images when only $\phi_0$ is known \emph{a priori}. (c) Loss values and reconstructed images when all $\phi_k$ are known \emph{a priori}. (d) and (e) Hyperparameters chosen optimally as per the Nesterov iterative method. (d) Loss values and reconstructed images when only $\phi_0$ is known \emph{a priori}. (e) Loss values and reconstructed images when all $\phi_k$ are known \emph{a priori}. Please, see Figure \ref{fig:Noisy} for the noisy case.}
        \label{fig:IDPsimulations}
\end{figure}

\subsection{Image deblurring problem} \label{sec:IDP}

In this section, we consider a variant of the image deblurring problem of Beck and Teboulle \citep{Beck_2009}, with a time-varying blur. All processing is done in the \emph{frequency domain} in which the unknown true image is denoted as $\theta^*$, the measured blurry version is represented as $y_k=\phi_k^T\theta^*$, where the blur is represented as the regressor $\phi_k$. We consider the fully adversarial setting where $\phi_k$ may be affected due to issues in the communication system, changes in lighting, or any other adversarial effects. For the purposes of the presented results, we consider a scalar multiplicative perturbation $\delta_k$, which results in a frequency domain blur representation as $\phi_k = \delta_k\odot\text{blur\_operator}(P_k)$, where $P_k$ is a known point spread function. We employ the same squared loss function as in \eqref{e:Squared_Loss_N}. A complete description of the problem formulation is provided in \citep[Appendix \ref{ImageDeblurringStatement}]{Gaudio_arXiv_2020}, alongside additional experiments with noisy measurements.

Figure \ref{fig:IDPsimulations} shows the loss values and reconstructed images at iteration $2000$, for the ramp change in the regressor/blur in Figure \ref{fig:IDP_info}. We present numerical results for hyperparameters chosen in four different ways. In Figure \ref{fig:IDP_stable} and Figure \ref{fig:IDP_stable_SUP}, the higher order tuner hyperparameters are chosen according to Theorem \ref{th:HOT_paper_Full_Alg}, and  $\bar{\alpha}$ is chosen as $\bar{\alpha} = \gamma\beta/\mathcal{N}_0$ in Figure \ref{fig:IDP_stable} and $\bar{\alpha} = \gamma\beta/\max{\mathcal{N}_k}$ in Figure \ref{fig:IDP_stable_SUP}. In Figure \ref{fig:IDP_aggresive} and Figure \ref{fig:IDP_aggresive_SUP}, the step size is chosen as $\bar{\alpha}=1/\lVert\phi_0\rVert^2_2$ and $\bar{\alpha}=1/\max{\lVert\phi_k\rVert^2_2}$ respectively, which results in hyperparameter choices for the higher order tuner as $\mu=10^{-20}$, $\beta=0.1$, and $\gamma=1/\beta$, as per Proposition \ref{prop:equivalence}. The higher order tuner remains stable as opposed to the other methods which are unstable if the parameters are not chosen appropriately for all $\phi_k$.

\acks{This work was supported by the Air Force Research Laboratory, Collaborative Research and Development for Innovative Aerospace Leadership (CRDInAL), Thrust 3 - Control Automation and Mechanization grant FA 8650-16-C-2642 and the Boeing Strategic University Initiative; cleared for release, case number 88ABW-2020-1889.}

\bibliography{References}

\begin{thebibliography}{64}
\providecommand{\natexlab}[1]{#1}
\providecommand{\url}[1]{\texttt{#1}}
\expandafter\ifx\csname urlstyle\endcsname\relax
  \providecommand{\doi}[1]{doi: #1}\else
  \providecommand{\doi}{doi: \begingroup \urlstyle{rm}\Url}\fi

\bibitem[Auer et~al.(1995)Auer, Cesa-Bianchi, Freund, and Schapire]{Auer_1995}
Peter Auer, Nicol{\`{o}} Cesa-Bianchi, Yoav Freund, and Robert~E. Schapire.
\newblock Gambling in a rigged casino: The adversarial multi-armed bandit
  problem.
\newblock In \emph{Proceedings of {IEEE} 36th Annual Foundations of Computer
  Science}. {IEEE} Comput. Soc. Press, 1995.
\newblock \doi{10.1109/SFCS.1995.492488}.

\bibitem[Auer et~al.(2002)Auer, Cesa-Bianchi, and Fischer]{Auer_2002}
Peter Auer, Nicol{\`{o}} Cesa-Bianchi, and Paul Fischer.
\newblock Finite-time analysis of the multiarmed bandit problem.
\newblock \emph{Machine Learning}, 47\penalty0 (2/3):\penalty0 235--256, 2002.
\newblock \doi{10.1023/A:1013689704352}.

\bibitem[Beck and Teboulle(2009{\natexlab{a}})]{Beck_2009}
Amir Beck and Marc Teboulle.
\newblock A fast iterative shrinkage-thresholding algorithm for linear inverse
  problems.
\newblock \emph{{SIAM} Journal on Imaging Sciences}, 2\penalty0 (1):\penalty0
  183--202, jan 2009{\natexlab{a}}.
\newblock \doi{10.1137/080716542}.

\bibitem[Beck and Teboulle(2009{\natexlab{b}})]{Beck_2009_TIP}
Amir Beck and Marc Teboulle.
\newblock Fast gradient-based algorithms for constrained total variation image
  denoising and deblurring problems.
\newblock \emph{{IEEE} Transactions on Image Processing}, 18\penalty0
  (11):\penalty0 2419--2434, nov 2009{\natexlab{b}}.
\newblock \doi{10.1109/TIP.2009.2028250}.

\bibitem[Ben-David et~al.(2009)Ben-David, Blitzer, Crammer, Kulesza, Pereira,
  and Vaughan]{Ben_David_2009}
Shai Ben-David, John Blitzer, Koby Crammer, Alex Kulesza, Fernando Pereira, and
  Jennifer~Wortman Vaughan.
\newblock A theory of learning from different domains.
\newblock \emph{Machine Learning}, 79\penalty0 (1-2):\penalty0 151--175, oct
  2009.
\newblock \doi{10.1007/s10994-009-5152-4}.

\bibitem[Bengio(2012)]{Bengio_2012}
Yoshua Bengio.
\newblock Practical recommendations for gradient-based training of deep
  architectures.
\newblock In \emph{Lecture Notes in Computer Science}, pages 437--478. Springer
  Berlin Heidelberg, 2012.
\newblock \doi{10.1007/978-3-642-35289-8_26}.

\bibitem[Betancourt et~al.(2018)Betancourt, Jordan, and
  Wilson]{Betancourt_2018}
Michael Betancourt, Michael~I. Jordan, and Ashia~C. Wilson.
\newblock On symplectic optimization.
\newblock \emph{arXiv preprint arXiv:1802.03653}, 2018.

\bibitem[Bubeck(2015)]{Bubeck_2015}
S{\'{e}}bastien Bubeck.
\newblock Convex optimization: Algorithms and complexity.
\newblock \emph{Foundations and Trends{\textregistered} in Machine Learning},
  8\penalty0 (3-4):\penalty0 231--357, 2015.
\newblock \doi{10.1561/2200000050}.

\bibitem[Bubeck and Cesa-Bianchi(2012)]{Bubeck_2012}
S{\'{e}}bastien Bubeck and Nicol{\`{o}} Cesa-Bianchi.
\newblock \emph{Regret Analysis of Stochastic and Nonstochastic Multi-armed
  Bandit Problems}, volume~5.
\newblock Now Publishers, 2012.
\newblock \doi{10.1561/2200000024}.

\bibitem[Carmon et~al.(2018)Carmon, Duchi, Hinder, and Sidford]{Carmon_2018}
Yair Carmon, John~C. Duchi, Oliver Hinder, and Aaron Sidford.
\newblock Accelerated methods for {NonConvex} optimization.
\newblock \emph{{SIAM} Journal on Optimization}, 28\penalty0 (2):\penalty0
  1751--1772, jan 2018.
\newblock \doi{10.1137/17M1114296}.

\bibitem[Cesa-Bianchi and Lugosi(2006)]{Cesa_Bianchi_2006}
Nicolo Cesa-Bianchi and G{\'a}bor Lugosi.
\newblock \emph{Prediction, Learning, and Games}.
\newblock Cambridge University Press, 2006.

\bibitem[Cesa-Bianchi et~al.(2004)Cesa-Bianchi, Conconi, and
  Gentile]{Cesa_Bianchi_2004}
Nicolo Cesa-Bianchi, Alex Conconi, and Claudio Gentile.
\newblock On the generalization ability of on-line learning algorithms.
\newblock \emph{{IEEE} Transactions on Information Theory}, 50\penalty0
  (9):\penalty0 2050--2057, sep 2004.
\newblock \doi{10.1109/TIT.2004.833339}.

\bibitem[Chen and Liu(2018)]{Chen_2018_LML}
Zhiyuan Chen and Bing Liu.
\newblock \emph{Lifelong Machine Learning}.
\newblock Morgan \& Claypool Publishers, 2018.

\bibitem[Dietterich(2002)]{Dietterich_2002}
Thomas~G. Dietterich.
\newblock Machine learning for sequential data: A review.
\newblock In \emph{Lecture Notes in Computer Science}, pages 15--30. Springer
  Berlin Heidelberg, 2002.
\newblock \doi{10.1007/3-540-70659-3_2}.

\bibitem[Duchi et~al.(2011)Duchi, Hazan, and Singer]{Duchi_2011}
John Duchi, Elad Hazan, and Yoram Singer.
\newblock Adaptive subgradient methods for online learning and stochastic
  optimization.
\newblock \emph{Journal of Machine Learning Research}, 12:\penalty0 2121--2159,
  July 2011.

\bibitem[Evesque et~al.(2003)Evesque, Annaswamy, Niculescu, and
  Dowling]{Evesque_2003}
S.~Evesque, A.~M. Annaswamy, S.~Niculescu, and A.~P. Dowling.
\newblock Adaptive control of a class of time-delay systems.
\newblock \emph{Journal of Dynamic Systems, Measurement, and Control},
  125\penalty0 (2):\penalty0 186, 2003.
\newblock \doi{10.1115/1.1567755}.

\bibitem[Fei et~al.(2016)Fei, Wang, and Liu]{Fei_2016}
Geli Fei, Shuai Wang, and Bing Liu.
\newblock Learning cumulatively to become more knowledgeable.
\newblock In \emph{Proceedings of the 22nd {ACM} {SIGKDD} International
  Conference on Knowledge Discovery and Data Mining - {KDD} 16}. {ACM} Press,
  2016.
\newblock \doi{10.1145/2939672.2939835}.

\bibitem[Gaudio et~al.(2020)Gaudio, Annaswamy, Moreu, Bolender, and
  Gibson]{Gaudio_arXiv_2020}
Joseph~E. Gaudio, Anuradha~M. Annaswamy, Jos\'{e}~M. Moreu, Michael~A.
  Bolender, and Travis~E. Gibson.
\newblock Accelerated learning with robustness to adversarial regressors.
\newblock \emph{arXiv preprint arXiv:2005.01529}, 2020.

\bibitem[Gitman et~al.(2019)Gitman, Lang, Zhang, and Xiao]{Gitman_2019}
Igor Gitman, Hunter Lang, Pengchuan Zhang, and Lin Xiao.
\newblock Understanding the role of momentum in stochastic gradient methods.
\newblock In H.~Wallach, H.~Larochelle, A.~Beygelzimer, F.~d'Alch\'{e} Buc,
  E.~Fox, and R.~Garnett, editors, \emph{Advances in Neural Information
  Processing Systems 32}, pages 9633--9643. Curran Associates, Inc., 2019.

\bibitem[Goldstein et~al.(2002)Goldstein, Poole, and Safko]{Goldstein_2002}
Herbert Goldstein, Charles Poole, and John Safko.
\newblock \emph{Classical Mechanics}.
\newblock Addison Wesley, 2002.

\bibitem[Goodfellow et~al.(2016)Goodfellow, Bengio, and
  Courville]{Goodfellow-et-al-2016}
Ian Goodfellow, Yoshua Bengio, and Aaron Courville.
\newblock \emph{Deep Learning}.
\newblock MIT Press, 2016.

\bibitem[Goodwin and Sin(1984)]{Goodwin_1984}
Graham~C Goodwin and Kwai~Sang Sin.
\newblock \emph{Adaptive Filtering Prediction and Control}.
\newblock Prentice Hall, 1984.

\bibitem[Hairer et~al.(2006)Hairer, Lubich, and Wanner]{Hairer_2006}
Ernst Hairer, Christian Lubich, and Gerhard Wanner.
\newblock \emph{Geometric Numerical Integration: Structure-Preserving
  Algorithms for Ordinary Differential Equations}.
\newblock Springer, 2006.

\bibitem[Hall and Willett(2015)]{Hall_2015}
Eric~C. Hall and Rebecca~M. Willett.
\newblock Online convex optimization in dynamic environments.
\newblock \emph{{IEEE} Journal of Selected Topics in Signal Processing},
  9\penalty0 (4):\penalty0 647--662, jun 2015.
\newblock \doi{10.1109/JSTSP.2015.2404790}.

\bibitem[Hansen et~al.(2006)Hansen, Nagy, and O'leary]{Hansen_2006}
Per~Christian Hansen, James~G. Nagy, and Dianne~P. O'leary.
\newblock \emph{Deblurring Images: Matrices, Spectra, and Filtering}.
\newblock SIAM, 2006.

\bibitem[Haykin(2014)]{Haykin_2014}
Simon Haykin.
\newblock \emph{Adaptive Filter Theory}.
\newblock Pearson, 2014.

\bibitem[Hazan(2016)]{Hazan_2016}
Elad Hazan.
\newblock Introduction to online convex optimization.
\newblock \emph{Foundations and Trends{\textregistered} in Optimization},
  2\penalty0 (3-4):\penalty0 157--325, 2016.
\newblock \doi{10.1561/2400000013}.

\bibitem[Hazan et~al.(2007)Hazan, Agarwal, and Kale]{Hazan_2007}
Elad Hazan, Amit Agarwal, and Satyen Kale.
\newblock Logarithmic regret algorithms for online convex optimization.
\newblock \emph{Machine Learning}, 69\penalty0 (2-3):\penalty0 169--192, aug
  2007.
\newblock \doi{10.1007/s10994-007-5016-8}.

\bibitem[Hazan et~al.(2008)Hazan, Rakhlin, and Bartlett]{Hazan_2008}
Elad Hazan, Alexander Rakhlin, and Peter~L. Bartlett.
\newblock Adaptive online gradient descent.
\newblock In J.~C. Platt, D.~Koller, Y.~Singer, and S.~T. Roweis, editors,
  \emph{Advances in Neural Information Processing Systems 20}, pages 65--72.
  Curran Associates, Inc., 2008.

\bibitem[Ioannou and Sun(1996)]{Ioannou1996}
Petros~A. Ioannou and Jing Sun.
\newblock \emph{Robust Adaptive Control}.
\newblock Prentice-Hall, 1996.

\bibitem[Jain et~al.(2018)Jain, Netrapalli, Kakade, Kidambi, and
  Sidford]{Jain_2018}
Prateek Jain, Praneeth Netrapalli, Sham~M. Kakade, Rahul Kidambi, and Aaron
  Sidford.
\newblock Accelerating stochastic gradient descent for least squares
  regression.
\newblock In Sébastien Bubeck, Vianney Perchet, and Philippe Rigollet,
  editors, \emph{Conference On Learning Theory}, volume~75 of \emph{Proceedings
  of Machine Learning Research}, pages 545--604. PMLR, 06--09 Jul 2018.

\bibitem[Jordan and Mitchell(2015)]{Jordan_2015}
M.~I. Jordan and T.~M. Mitchell.
\newblock Machine learning: Trends, perspectives, and prospects.
\newblock \emph{Science}, 349\penalty0 (6245):\penalty0 255--260, jul 2015.
\newblock ISSN 0036-8075.
\newblock \doi{10.1126/science.aaa8415}.

\bibitem[Kingma and Ba(2017)]{Kingma_2017}
Diederik~P. Kingma and Jimmy~L. Ba.
\newblock Adam: A method for stochastic optimization.
\newblock \emph{arXiv preprint arXiv:1412.6980}, 2017.

\bibitem[Krizhevsky et~al.(2012)Krizhevsky, Sutskever, and
  Hinton]{Krizhevsky_2012}
Alex Krizhevsky, Ilya Sutskever, and Geoffrey~E. Hinton.
\newblock Imagenet classification with deep convolutional neural networks.
\newblock In F.~Pereira, C.~J.~C. Burges, L.~Bottou, and K.~Q. Weinberger,
  editors, \emph{Advances in Neural Information Processing Systems 25}, pages
  1097--1105. Curran Associates, Inc., 2012.

\bibitem[Kuznetsov and Mohri(2015)]{Kuznetsov_2015}
Vitaly Kuznetsov and Mehryar Mohri.
\newblock Learning theory and algorithms for forecasting non-stationary time
  series.
\newblock In C.~Cortes, N.~D. Lawrence, D.~D. Lee, M.~Sugiyama, and R.~Garnett,
  editors, \emph{Advances in Neural Information Processing Systems 28}, pages
  541--549. Curran Associates, Inc., 2015.

\bibitem[Lessard et~al.(2016)Lessard, Recht, and Packard]{Lessard_2016}
Laurent Lessard, Benjamin Recht, and Andrew Packard.
\newblock Analysis and design of optimization algorithms via integral quadratic
  constraints.
\newblock \emph{{SIAM} Journal on Optimization}, 26\penalty0 (1):\penalty0
  57--95, jan 2016.
\newblock \doi{10.1137/15m1009597}.

\bibitem[Lopez-Paz and Ranzato(2017)]{Lopez_2017}
David Lopez-Paz and Marc'Aurelio Ranzato.
\newblock Gradient episodic memory for continual learning.
\newblock In I.~Guyon, U.~V. Luxburg, S.~Bengio, H.~Wallach, R.~Fergus,
  S.~Vishwanathan, and R.~Garnett, editors, \emph{Advances in Neural
  Information Processing Systems 30}, pages 6467--6476. Curran Associates,
  Inc., 2017.

\bibitem[Luenberger(1969)]{Luenberger_1969}
David~G. Luenberger.
\newblock \emph{Optimization by Vector Space Methods}.
\newblock John Wiley \& Sons, 1969.

\bibitem[Morse(1992)]{Morse_1992}
A.~S. Morse.
\newblock High-order parameter tuners for the adaptive control of linear and
  nonlinear systems.
\newblock In \emph{Systems, Models and Feedback: Theory and Applications},
  pages 339--364. Birkhäuser Boston, 1992.
\newblock \doi{10.1007/978-1-4757-2204-8_23}.

\bibitem[Narendra and Annaswamy(2005)]{Narendra2005}
Kumpati~S. Narendra and Anuradha~M. Annaswamy.
\newblock \emph{Stable Adaptive Systems}.
\newblock Dover, 2005.

\bibitem[Nesterov(1983)]{Nesterov_1983}
Yurii Nesterov.
\newblock A method of solving a convex programming problem with convergence
  rate ${O}(1/k^2)$.
\newblock \emph{Soviet Mathematics Doklady}, 27:\penalty0 372--376, 1983.

\bibitem[Nesterov(2004)]{Nesterov_2004}
Yurii Nesterov.
\newblock \emph{Introductory Lectures on Convex Optimization}.
\newblock Springer, 2004.
\newblock \doi{10.1007/978-1-4419-8853-9}.

\bibitem[Nesterov(2018)]{Nesterov_2018}
Yurii Nesterov.
\newblock \emph{Lectures on Convex Optimization}.
\newblock Springer, 2018.
\newblock \doi{10.1007/978-3-319-91578-4}.

\bibitem[Parisi et~al.(2019)Parisi, Kemker, Part, Kanan, and
  Wermter]{Parisi_2019}
German~I. Parisi, Ronald Kemker, Jose~L. Part, Christopher Kanan, and Stefan
  Wermter.
\newblock Continual lifelong learning with neural networks: A review.
\newblock \emph{Neural Networks}, 113:\penalty0 54--71, may 2019.
\newblock \doi{10.1016/j.neunet.2019.01.012}.

\bibitem[Pentina and Lampert(2015)]{Pentina_2015}
Anastasia Pentina and Christoph~H. Lampert.
\newblock Lifelong learning with non-i.i.d. tasks.
\newblock In C.~Cortes, N.~D. Lawrence, D.~D. Lee, M.~Sugiyama, and R.~Garnett,
  editors, \emph{Advances in Neural Information Processing Systems 28}, pages
  1540--1548. Curran Associates, Inc., 2015.

\bibitem[Polyak(1964)]{Polyak_1964}
B.~T. Polyak.
\newblock Some methods of speeding up the convergence of iteration methods.
\newblock \emph{{USSR} Computational Mathematics and Mathematical Physics},
  4\penalty0 (5):\penalty0 1--17, jan 1964.
\newblock \doi{10.1016/0041-5553(64)90137-5}.

\bibitem[Popov(1973)]{Popov_1973}
V.~M. Popov.
\newblock \emph{Hyperstability of Control Systems}.
\newblock Springer-Verlag, 1973.

\bibitem[Raginsky et~al.(2010)Raginsky, Rakhlin, and Yuksel]{Raginsky_2010}
Maxim Raginsky, Alexander Rakhlin, and Serdar Yuksel.
\newblock Online convex programming and regularization in adaptive control.
\newblock In \emph{49th {IEEE} Conference on Decision and Control ({CDC})}.
  {IEEE}, 2010.
\newblock \doi{10.1109/CDC.2010.5717262}.

\bibitem[Recht(2012)]{Recht_HB}
Benjamin Recht.
\newblock Cs726 - lyapunov analysis and the heavy ball method.
\newblock Online, October 2012.

\bibitem[Sastry and Bodson(1989)]{Sastry_1989}
Shankar Sastry and Marc Bodson.
\newblock \emph{Adaptive Control: Stability, Convergence and Robustness}.
\newblock Prentice-Hall, 1989.

\bibitem[Shalev-Shwartz(2011)]{Shalev_Shwartz_2011}
Shai Shalev-Shwartz.
\newblock Online learning and online convex optimization.
\newblock \emph{Foundations and Trends{\textregistered} in Machine Learning},
  4\penalty0 (2):\penalty0 107--194, 2011.
\newblock \doi{10.1561/2200000018}.

\bibitem[Shi et~al.(2019)Shi, Du, Su, and Jordan]{Shi_2019}
Bin Shi, Simon~S. Du, Weijie~J. Su, and Michael~I. Jordan.
\newblock Acceleration via symplectic discretization of high-resolution
  differential equations.
\newblock In H.~Wallach, H.~Larochelle, A.~Beygelzimer, F.~d'Alch\'{e} Buc,
  E.~Fox, and R.~Garnett, editors, \emph{Advances in Neural Information
  Processing Systems 32}, pages 5744--5752. Curran Associates, Inc., 2019.

\bibitem[Silver et~al.(2013)Silver, Yang, and Li]{Silver_2013}
Daniel Silver, Qiang Yang, and Lianghao Li.
\newblock Lifelong machine learning systems: Beyond learning algorithms.
\newblock In \emph{AAAI Spring Symposium Series}, 2013.

\bibitem[Su et~al.(2016)Su, Boyd, and Cand{{\`e}}s]{Su_2016}
Weijie Su, Stephen Boyd, and Emmanuel~J. Cand{{\`e}}s.
\newblock A differential equation for modeling nesterov's accelerated gradient
  method: Theory and insights.
\newblock \emph{Journal of Machine Learning Research}, 17\penalty0
  (153):\penalty0 1--43, 2016.

\bibitem[Sutskever et~al.(2013)Sutskever, Martens, Dahl, and
  Hinton]{Sutskever_2013}
Ilya Sutskever, James Martens, George Dahl, and Geoffrey Hinton.
\newblock On the importance of initialization and momentum in deep learning.
\newblock In Sanjoy Dasgupta and David McAllester, editors, \emph{Proceedings
  of the 30th International Conference on Machine Learning}, volume~28 of
  \emph{Proceedings of Machine Learning Research}, pages 1139--1147. PMLR,
  2013.

\bibitem[Thrun(1998)]{Thrun_1998}
Sebastian Thrun.
\newblock Lifelong learning algorithms.
\newblock In \emph{Learning to Learn}, pages 181--209. Springer {US}, 1998.
\newblock \doi{10.1007/978-1-4615-5529-2_8}.

\bibitem[Thrun and Mitchell(1995)]{Thrun_1995}
Sebastian Thrun and Tom~M. Mitchell.
\newblock Lifelong robot learning.
\newblock \emph{Robotics and Autonomous Systems}, 15\penalty0 (1-2):\penalty0
  25--46, jul 1995.
\newblock \doi{10.1016/0921-8890(95)00004-Y}.

\bibitem[Wang and Tao(2014)]{WangTao}
Ruxin Wang and Dacheng Tao.
\newblock Recent progress in image deblurring.
\newblock \emph{arXiv preprint arXiv:1409.6838}, 2014.

\bibitem[Wibisono et~al.(2016)Wibisono, Wilson, and Jordan]{Wibisono_2016}
Andre Wibisono, Ashia~C. Wilson, and Michael~I. Jordan.
\newblock A variational perspective on accelerated methods in optimization.
\newblock \emph{Proceedings of the National Academy of Sciences}, 113\penalty0
  (47):\penalty0 E7351--E7358, nov 2016.
\newblock \doi{10.1073/pnas.1614734113}.

\bibitem[Widrow and Stearns(1985)]{Widrow_1985}
Bernard Widrow and Samuel~D. Stearns.
\newblock \emph{Adaptive Signal Processing}.
\newblock Prentice-Hall, 1985.

\bibitem[Wilson(2018)]{Wilson_2018}
Ashia Wilson.
\newblock \emph{Lyapunov Arguments in Optimization}.
\newblock PhD thesis, University of California, Berkeley, 2018.

\bibitem[Wilson et~al.(2016)Wilson, Recht, and Jordan]{Wilson_2016}
Ashia~C. Wilson, Benjamin Recht, and Michael~I. Jordan.
\newblock A lyapunov analysis of momentum methods in optimization.
\newblock \emph{arXiv preprint arXiv:1611.02635}, 2016.

\bibitem[Wilson et~al.(2017)Wilson, Roelofs, Stern, Srebro, and
  Recht]{Wilson_2017}
Ashia~C. Wilson, Rebecca Roelofs, Mitchell Stern, Nathan Srebro, and Benjamin
  Recht.
\newblock The marginal value of adaptive gradient methods in machine learning.
\newblock In I.~Guyon, U.~V. Luxburg, S.~Bengio, H.~Wallach, R.~Fergus,
  S.~Vishwanathan, and R.~Garnett, editors, \emph{Advances in Neural
  Information Processing Systems 30}, pages 4148--4158. Curran Associates,
  Inc., 2017.

\bibitem[Zinkevich(2003)]{Zinkevich_2003}
Martin Zinkevich.
\newblock Online convex programming and generalized infinitesimal gradient
  ascent.
\newblock In \emph{Proceedings of the 20th International Conference on Machine
  Learning (ICML-03)}, pages 928--936, 2003.

\end{thebibliography}

\clearpage
\tableofcontents

\clearpage
\appendix
\noindent{\LARGE \textbf{Appendix}}

\paragraph{Organization of the appendix.}

Mathematical preliminaries alongside additional continuous and discrete time equations are provided for completeness in Appendix \ref{s:Preliminaries_N}. Lyapunov stability proofs for the continuous and discrete time algorithms considered in this paper are provided in Appendix \ref{s:Stability_N_Proofs}. Non-asymptotic convergence rate proofs are provided in Appendix \ref{s:NonAsymptoticProofs}. Appendix \ref{s:Experiment_details} provides further implementation details of the numerical experiments in this paper, alongside additional experiments.

\section{Preliminaries}
\label{s:Preliminaries_N}

\subsection{Definitions}
\label{ss:Definitions_N}

The following definitions of convexity and smoothness, modified from \citep{Nesterov_2018} are used throughout.
\begin{definition}
    A continuously differentiable function $f$ is convex if
    \begin{equation}\label{e:convex}
        f(y)\geq f(x)+\nabla f(x)^T(y-x),\quad \forall x,y\in\mathbb{R}^N.
    \end{equation}
\end{definition}
\begin{definition}
    A twice continuously differentiable function $f$ is convex if
    \begin{equation*}
        \nabla^2f(x)\geq0,\quad \forall x\in\mathbb{R}^N.
    \end{equation*}
\end{definition}
\begin{definition}
    A continuously differentiable function $f$ is $\mu$-strongly convex if there exists a $\mu>0$ such that
    \begin{equation}\label{e:strongly_convex}
        f(y)\geq f(x)+\nabla f(x)^T(y-x)+\frac{\mu}{2}\lVert y-x\rVert^2,\quad \forall x,y\in\mathbb{R}^N.
    \end{equation}
\end{definition}
\begin{definition}
    A twice continuously differentiable function $f$ is $\mu$-strongly convex if there exists a $\mu>0$ such that
    \begin{equation*}
        \nabla^2f(x)\geq\mu I,\quad \forall x\in\mathbb{R}^N.
    \end{equation*}
\end{definition}
\begin{definition}
    A continuously differentiable function $f$ is $\bar{L}$-smooth if there exists a $\bar{L}>0$ such that
    \begin{equation}\label{e:smooth_convex}
        f(y)\leq f(x)+\nabla f(x)^T(y-x)+\frac{\bar{L}}{2}\lVert y-x\rVert^2,\quad \forall x,y\in\mathbb{R}^N.
    \end{equation}
\end{definition}
\begin{definition}
    A twice continuously differentiable function $f$ is $\bar{L}$-smooth if there exists a $\bar{L}>0$ such that
    \begin{equation*}
        \nabla^2f(x)\leq\bar{L}I,\quad \forall x\in\mathbb{R}^N.
    \end{equation*}
\end{definition}

The following two Euler-type discretization methods are employed in this paper.
\begin{definition}\label{d:explicit_Euler}
    An explicit Euler discretization of a differential equation $\dot{x}(t)=f(x(t))$, where $f:\mathbb{R}^N\rightarrow\mathbb{R}^N$, and $\Delta t$ is the sample time, takes the form
    \begin{equation*}
        \dot{x}(t)\approx\frac{x_{k+1}-x_k}{\Delta t}=f(x_k).
    \end{equation*}
\end{definition}
\begin{definition}\label{d:implicit_Euler}
    An implicit Euler discretization of a differential equation $\dot{x}(t)=f(x(t))$, where $f:\mathbb{R}^N\rightarrow\mathbb{R}^N$, and $\Delta t$ is the sample time, takes the form
    \begin{equation*}
        \dot{x}(t)\approx\frac{x_{k+1}-x_k}{\Delta t}=f(x_{k+1}).
    \end{equation*}
\end{definition}

The classes $\mathcal{L}_p$ and $\ell_p$ for $p\in[1,\infty]$ are described below.
\begin{definition}[See \citep{Narendra2005}]\label{d:L_p}
    For any fixed $p\in[1,\infty)$, $f:\mathbb{R}_+\rightarrow\mathbb{R}$ is defined to belong to $\mathcal{L}_p$ if $f$ is locally integrable and
    \begin{equation*}
        \lVert f\rVert_{\mathcal{L}_p}\triangleq\left(\lim_{t\rightarrow\infty}\int_0^t\lVert f(\tau)\rVert^pd\tau\right)^{\frac{1}{p}}<\infty.
    \end{equation*}
    When $p=\infty$, $f\in\mathcal{L}_{\infty}$ if,
    \begin{equation*}
        \lVert f\rVert_{\mathcal{L}_{\infty}}\triangleq \sup_{t\geq0}\lVert f(t)\rVert<\infty.
    \end{equation*}
\end{definition}
\begin{definition}[See \citep{Luenberger_1969}]\label{d:l_p}
    For any fixed $p\in[1,\infty)$, a sequence of scalars $\xi=\{\xi_0,\xi_1,\xi_2,\ldots\}$ is defined to belong to $\ell_p$ if
    \begin{equation*}
        \lVert \xi\rVert_{\ell_p}\triangleq\left(\lim_{k\rightarrow\infty}\sum_{i=0}^k\lVert \xi_i\rVert^p\right)^{\frac{1}{p}}<\infty.
    \end{equation*}
    When $p=\infty$, $\xi\in\ell_{\infty}$ if,
    \begin{equation*}
        \lVert \xi\rVert_{\ell_{\infty}}\triangleq \sup_{i\geq0}\lVert \xi_i\rVert<\infty.
    \end{equation*}
\end{definition}

The following lemma was attributed to Barbalat by Popov \citep{Popov_1973} and has found significant use in the fields of adaptive and nonlinear control. The version from \citep{Narendra2005} is stated below with an associated corollary.
\begin{lemma}[See \citep{Narendra2005}]
    If $f:\mathbb{R}_+\rightarrow\mathbb{R}$ is uniformly continuous for $t\geq0$, and if the limit of the integral
    \begin{equation*}
        \lim_{t\rightarrow\infty}\int_0^t|f(\tau)|d\tau
    \end{equation*}
    exists and is finite, then
    \begin{equation*}
        \lim_{t\rightarrow\infty}f(t)=0.
    \end{equation*}
\end{lemma}
\begin{corollary}[See \citep{Narendra2005}]\label{c:Barbalat_Corollary}
    If $f\in\mathcal{L}_2\cap\mathcal{L}_{\infty}$, and $\dot{f}\in\mathcal{L}_{\infty}$, then $\lim_{t\rightarrow\infty}f(t)=0$.
\end{corollary}
A discrete time proposition which corresponds to Corollary \ref{c:Barbalat_Corollary} follows.
\begin{prop}
    If $\xi\in\ell_2\cap\ell_{\infty}$, then $\lim_{k\rightarrow\infty}\xi_k=0$.
\end{prop}

\subsection{Continuous time problem setting}
\label{ss:Continuous_time_setting}

In the interest of completeness for the continuous time stability proofs provided in Appendix \ref{s:Stability_N_Proofs}, the following analogs of the problem setting of Section \ref{s:Problem_Setting_N} in continuous time are provided below. In particular, the output error in direct correspondence with \eqref{e:error1_N_discrete_N} is stated as
\begin{equation}\label{e:error1_N}
    e_y(t)=\hat{y}(t)-y(t)=\tilde{\theta}^T(t)\phi(t),
\end{equation}
where $\tilde{\theta}(t)=\theta(t)-\theta^*$ is the parameter estimation error and $\phi(t)$ is the regressor. The analog of the squared loss function in \eqref{e:Squared_Loss_N} may be formulated in continuous time as
\begin{equation}\label{e:error1_N_Loss}
    L_t(\theta(t))=\frac{1}{2}e_y^2(t)=\frac{1}{2}\tilde{\theta}^T(t)\phi(t)\phi^T(t)\tilde{\theta}(t).
\end{equation}
The continuous time analog of the normalized gradient method in \eqref{e:GD_Goodwin_N} is stated as
\begin{equation}\label{e:Gradient_Flow_N_N}
    \dot{\theta}(t)=-\frac{\gamma}{\N_t}\nabla L_t(\theta(t)),
\end{equation}
where the normalization signal is $\N_t=1+\lVert\phi(t)\rVert^2$.

\subsection{Additional discrete time methods}
\label{ss:Discrete_time_setting}

\begin{wrapfigure}{r}{0.50\textwidth}
    \begin{minipage}{0.50\textwidth}
        \vspace{-0.8cm}
        \begin{algorithm}[H]
        \caption{Higher Order Tuner Optimizer (HB)}
        \label{alg:HOT_R_HB}
        \begin{algorithmic}[1]
        \STATE {\bfseries Input:} initial conditions $\theta_0$, $\vartheta_0$, gains $\gamma$, $\beta$, $\mu$
        \FOR{$k=0,1,2,\ldots$}
        \STATE \textbf{Receive} regressor $\phi_k$, output $y_k$
        \STATE $\theta_{k+1}\leftarrow\theta_k-\beta(\theta_k-\vartheta_k)$
        \STATE Let $\N_k=1+\lVert\phi_k\rVert^2$,\\
        $\nabla L_k(\theta_{k+1})=\phi_k(\theta_{k+1}^T\phi_k-y_k)$,\\
        $\nabla f_k(\theta_{k+1})=\frac{\nabla L_k(\theta_{k+1})}{\N_k}+\mu(\theta_{k+1}-\theta_0)$
        \STATE $\vartheta_{k+1}\leftarrow\vartheta_k-\gamma\nabla f_k(\theta_{k+1})$
        \ENDFOR
        \end{algorithmic}
        \end{algorithm}
    \end{minipage}
\end{wrapfigure}
For completeness, this section provides additional discrete time iterative methods discussed in this paper. The gradient method may be stated as
\begin{equation}\label{e:Gradient_Method}
    \theta_{k+1}=\theta_k-\bar{\alpha}\nabla f(\theta_k).
\end{equation}
Similar to \eqref{e:Nesterov_Two_Convex}, Nesterov's algorithm with a time-varying $\beta_k$ may be expressed as
\begin{align}\label{e:Nesterov_Two_Convex_TV_beta}
    \begin{split}
        \theta_{k+1}&=\nu_k-\bar{\alpha}\nabla f(\nu_k),\\
        \nu_{k+1}&=\left(1+\bar{\beta}_k\right)\theta_{k+1} - \bar{\beta}_k\theta_{k}.
    \end{split}
\end{align}

A provably stable version of the Heavy Ball method of Polyak \citep{Polyak_1964} may be stated using a similar discretization of the continuous higher order tuner as in \eqref{e:higher_order_L_Discretized}, but without the ``Extra Gradient Step'' as
\begin{align}\label{e:higher_order_L_Discretized_HB1}
    \begin{split}
        \text{Implicit Euler}:\vartheta_{k+1}&=\vartheta_k-\gamma\nabla \bar{f}_k(\theta_{k+1}),\\
        \text{Explicit Euler}:\hspace{.04cm}\theta_{k+1}&=\theta_k-\beta(\theta_k-\vartheta_k).
    \end{split}
\end{align}
Similar to Algorithm \ref{alg:HOT_R}, using the same regularized function in \eqref{e:Strongly_Convex_Objective}, Algorithm \ref{alg:HOT_R_HB} may be provided based on the discretization procedure in \eqref{e:higher_order_L_Discretized_HB1} by replacing $\bar{f}_k(\theta_k)$ in \eqref{e:higher_order_L_Discretized_HB1} with $f_k(\theta_k)$ in \eqref{e:Strongly_Convex_Objective}. The following proposition relates Algorithm \ref{alg:HOT_R_HB} to the Heavy Ball method.
\begin{prop}
    Algorithm \ref{alg:HOT_R_HB} with a constant regressor $\phi_k\equiv\phi$ (and thus $f_k(\cdot)\equiv f(\cdot)$) may be reduced to the common form of the Heavy Ball method \citep{Polyak_1964} with $\bar{\beta}=1-\beta$ and $\bar{\alpha}=\gamma\beta$ as
\begin{equation}\label{e:Heavy_Ball}
    \theta_{k+1}=\left(1+\bar{\beta}\right)\theta_k - \bar{\beta}\theta_{k-1}-\bar{\alpha}\nabla f(\theta_k).
\end{equation}
\end{prop}
Stability proofs for \eqref{e:higher_order_L_Discretized_HB1} and Algorithm \ref{alg:HOT_R_HB} can be found in Appendix \ref{s:Stability_N_Proofs}. A non-asymptotic convergence rate proof for the Heavy Ball method with constant regressors as in \eqref{e:Heavy_Ball} can be found in Appendix \ref{s:NonAsymptoticProofs}.

\clearpage
\section{Stability proofs}
\label{s:Stability_N_Proofs}

\subsection{Regressor normalized gradient flow}

\begin{theorem}\label{th:Stability_Error1_GF}
    For the linear regression model in \eqref{e:error1_N} with loss in \eqref{e:error1_N_Loss}, the normalized gradient flow update in \eqref{e:Gradient_Flow_N_N} with $\gamma>0$, results in $\tilde{\theta}\in\mathcal{L}_{\infty}$ and $\frac{e_y}{\sqrt{\N_t}}\in\mathcal{L}_2\cap\mathcal{L}_{\infty}$. If in addition it assumed that $\phi,\dot{\phi} \in \mathcal{L}_{\infty}$ then $\lim_{t\rightarrow\infty}e_y(t)=0$ and $\lim_{t\rightarrow\infty}\dot{\tilde{\theta}}(t)=0$.
\end{theorem}
\begin{proof}
Consider the candidate Lyapunov function stated as
\begin{equation}\label{e:Lyap_error1}
V=\frac{1}{\gamma}\lVert\tilde{\theta}\rVert^2.
\end{equation}
Using \eqref{e:error1_N}, \eqref{e:error1_N_Loss}, and \eqref{e:Gradient_Flow_N_N} with $\gamma>0$, the time derivative of \eqref{e:Lyap_error1} may be bounded as
\begin{align*}
\dot{V}&=\frac{2}{\gamma}\tilde{\theta}^T\left(-\frac{\gamma}{\N_t}\nabla_{\theta}L_t(\theta(t))\right)\\
\dot{V}&=2\tilde{\theta}^T\left(-\frac{1}{\N_t}\phi e_y\right)\\
\dot{V}&=-\frac{2}{\N_t}e_y^2\leq0.
\end{align*}
Thus it can be concluded that $V$ is a Lyapunov function with $\tilde{\theta}\in\mathcal{L}_{\infty}$. Using \eqref{e:error1_N}, $\frac{e_y}{\sqrt{\N_t}}\in\mathcal{L}_{\infty}$. Integrating $\dot{V}$ from $t_0$ to $\infty$: $\int_{t_0}^{\infty}2\lVert\frac{e_y}{\sqrt{\N_t}}\rVert^2dt=-\int_{t_0}^{\infty}\dot{V}dt=V(t_0)-V(\infty)<\infty$, thus $\frac{e_y}{\sqrt{\N_t}}\in\mathcal{L}_2\cap\mathcal{L}_{\infty}$. From \eqref{e:error1_N_Loss} and \eqref{e:Gradient_Flow_N_N}, $\dot{\tilde{\theta}}\in\mathcal{L}_2\cap\mathcal{L}_{\infty}$. If additionally $\phi\in\mathcal{L}_{\infty}$, then $e_y\in\mathcal{L}_2\cap\mathcal{L}_{\infty}$. If additionally, $\dot{\phi}\in\mathcal{L}_{\infty}$ then from the time derivative of \eqref{e:error1_N}, it can be seen that $\dot{e}_y\in\mathcal{L}_{\infty}$ and from the time derivative of \eqref{e:Gradient_Flow_N_N}, $\ddot{\tilde{\theta}}\in\mathcal{L}_{\infty}$ and thus from Barbalat's lemma (Corollary \ref{c:Barbalat_Corollary}), $\lim_{t\rightarrow\infty}e_y(t)=0$ and $\lim_{t\rightarrow\infty}\dot{\tilde{\theta}}(t)=0$.
\end{proof}

\subsection{Regressor normalized gradient descent}

\begin{theorem}\label{th:Stability_Error1_GD}
    For the linear regression error model in \eqref{e:error1_N_discrete_N} with loss in \eqref{e:Squared_Loss_N}, the normalized gradient descent update in \eqref{e:GD_Goodwin_N} with $0<\gamma<2$, results in $\tilde{\theta}\in\ell_{\infty}$ and $\frac{e_{y,k}}{\sqrt{\N_k}}\in\ell_2\cap\ell_{\infty}$. If in addition it is assumed that $\phi\in\ell_{\infty}$ then $\lim_{k\rightarrow\infty}e_{y,k}=0$.
\end{theorem}
\begin{proof}
    Consider the candidate Lyapunov function stated as
    \begin{equation}\label{e:Lyapunov_Gradient_Descent}
        V_k=\frac{1}{\gamma}\lVert\tilde{\theta}_k\rVert^2.
    \end{equation}
    The increment $\Delta V_k:=V_{k+1}-V_k$ may then be expanded using \eqref{e:error1_N_discrete_N}, \eqref{e:Squared_Loss_N}, and \eqref{e:GD_Goodwin_N} as
    \begin{align*}
        \Delta V_k&=\frac{1}{\gamma}\lVert\tilde{\theta}_{k+1}\rVert^2-\frac{1}{\gamma}\lVert\tilde{\theta}_k\rVert^2\\
        \Delta V_k&=\frac{1}{\gamma}\lVert\tilde{\theta}_k-\frac{\gamma}{\N_k}\nabla L_k(\theta_k)\rVert^2-\frac{1}{\gamma}\lVert\tilde{\theta}_k\rVert^2\\
        \Delta V_k&=-2\left(1-\frac{\gamma\phi_k^T\phi_k}{2\N_k}\right)\frac{e_{y,k}^2}{\N_k}.
    \end{align*}
    Using $0<\gamma<2$, it can be seen that
    \begin{equation*}
        \Delta V_k\leq-(2-\gamma)\frac{e_{y,k}^2}{\N_k}\leq0.
    \end{equation*}
    Thus it can be concluded that $V$ is a Lyapunov function with $\tilde{\theta}\in\ell_{\infty}$. Using \eqref{e:error1_N_discrete_N}, $\frac{e_{y,k}}{\sqrt{\N_k}}\in\ell_{\infty}$. Collecting $\Delta V_k$ terms from $t_0$ to $T$: $\sum_{k=t_0}^{T}(2-\gamma)\lVert\frac{e_{y,k}}{\sqrt{\N_k}}\rVert^2\leq V_{t_0}-V_{T+1}<\infty$. Taking $T\rightarrow\infty$, it can be seen that $\frac{e_{y,k}}{\sqrt{\N_k}}\in\ell_2\cap\ell_{\infty}$ and therefore $\lim_{k\rightarrow\infty}\frac{e_{y,k}}{\sqrt{\N_k}}=0$. If additionally $\phi\in\ell_{\infty}$, then $e_{y,k}\in\ell_2\cap\ell_{\infty}$ and therefore $\lim_{k\rightarrow\infty}e_{y,k}=0$.
\end{proof}

\subsection{Continuous time higher order tuner}

\begin{manualtheorem}{\ref{th:HOT_paper} from Main Text (with proof)}
    For the linear regression model in \eqref{e:error1_N} with loss in \eqref{e:error1_N_Loss}, the higher order tuner update in \eqref{e:Accelerated_GF_N_N} with $\beta>0$, $0<\gamma\leq\beta/2$, results in $(\vartheta-\theta^*)\in\mathcal{L}_{\infty}$, $(\theta-\vartheta)\in\mathcal{L}_{\infty}$, and $\frac{e_y}{\sqrt{\N_t}}\in\mathcal{L}_2\cap\mathcal{L}_{\infty}$. If in addition it assumed that $\phi,\dot{\phi} \in \mathcal{L}_{\infty}$ then  $\lim_{t\rightarrow\infty}e_y(t)=0$, $\lim_{t\rightarrow\infty}(\theta(t)-\vartheta(t))=0$, $\lim_{t\rightarrow\infty}\dot{\vartheta}(t)=0$, and $\lim_{t\rightarrow\infty}\dot{\tilde{\theta}}(t)=0$.
\end{manualtheorem}
\begin{proof}
Consider the candidate Lyapunov function inspired by the higher order tuner approach in \citep{Evesque_2003} stated as
\begin{equation}\label{e:V_Accelerated_GF}
    V=\frac{1}{\gamma}\lVert\vartheta-\theta^*\rVert^2+\frac{1}{\gamma}\lVert\theta-\vartheta\rVert^2.
\end{equation}
Using \eqref{e:Accelerated_GF_N_N}, \eqref{e:error1_N}, and \eqref{e:error1_N_Loss} with $\gamma\leq\beta/2$, the time derivative of \eqref{e:V_Accelerated_GF} may be bounded as
\begin{align*}
    \dot{V}&=\frac{2}{\gamma}(\vartheta-\theta^*)^T\left(-\frac{\gamma}{\N_t}\nabla_{\theta}L_t(\theta)\right)+\frac{2}{\gamma}(\theta-\vartheta)^T\left(-\beta(\theta-\vartheta)+\frac{\gamma}{\N_t}\nabla_{\theta}L_t(\theta)\right)\\
    \dot{V}&=\frac{2}{\gamma}(\vartheta-\theta+\tilde{\theta})^T\left(-\frac{\gamma}{\N_t}\phi e_y\right)+\frac{2}{\gamma}(\theta-\vartheta)^T\left(-\beta(\theta-\vartheta)+\frac{\gamma}{\N_t}\phi e_y\right)\\
    \dot{V}&=\frac{1}{\N_t}\left\{-2e_y^2-\N_t\frac{2\beta}{\gamma}\lVert\theta-\vartheta\rVert^2+4(\theta-\vartheta)^T\phi e_y\right\}\\
    \dot{V}&=\frac{1}{\N_t}\left\{-2e_y^2-\frac{2\beta}{\gamma}\lVert\theta-\vartheta\rVert^2-\frac{2\beta}{\gamma}\lVert\theta-\vartheta\rVert^2\lVert\phi\rVert^2+4(\theta-\vartheta)^T\phi e_y\right\}\\
    \dot{V}&\leq\frac{1}{\N_t}\left\{-2e_y^2-\frac{2\beta}{\gamma}\lVert\theta-\vartheta\rVert^2-4\lVert\theta-\vartheta\rVert^2\lVert\phi\rVert^2+4\lVert\theta-\vartheta\rVert\lVert\phi\rVert \lVert e_y\rVert\right\}\\
    \dot{V}&\leq\frac{1}{\N_t}\left\{-\frac{2\beta}{\gamma}\lVert\theta-\vartheta\rVert^2-\lVert e_y\rVert^2-\left[\lVert e_y\rVert-2\lVert\theta-\vartheta\rVert\lVert\phi\rVert\right]^2\right\}\leq0.
\end{align*}
Thus it can be concluded that $V$ is a Lyapunov function with $(\vartheta-\theta^*)\in\mathcal{L}_{\infty}$ and $(\theta-\vartheta)\in\mathcal{L}_{\infty}$. Using \eqref{e:error1_N}, $\frac{e_y}{\sqrt{\N_t}}\in\mathcal{L}_{\infty}$. Integrating $\dot{V}$ from $t_0$ to $\infty$: $\int_{t_0}^{\infty}\lVert \frac{e_y}{\sqrt{\N_t}}\rVert^2dt\leq-\int_{t_0}^{\infty}\dot{V}dt=V(t_0)-V(\infty)<\infty$, thus $\frac{e_y}{\sqrt{\N_t}}\in\mathcal{L}_2\cap\mathcal{L}_{\infty}$. Likewise, $\int_{t_0}^{\infty}\frac{2\beta}{\gamma}\lVert\frac{\theta-\vartheta}{\sqrt{\N_t}}\rVert^2dt\leq-\int_{t_0}^{\infty}\dot{V}dt=V(t_0)-V(\infty)<\infty$, thus $\left(\frac{\theta-\vartheta}{\sqrt{\N_t}}\right)\in\mathcal{L}_2\cap\mathcal{L}_{\infty}$. Furthermore:
\begin{equation*}
    \left\lVert\frac{\theta-\vartheta}{\sqrt{\N_t}}\right\rVert^2_{\mathcal{L}_2}\leq\frac{\gamma V(t_0)}{2\beta},
\end{equation*}
where $\lVert\frac{\theta-\vartheta}{\sqrt{\N_t}}\rVert^2_{\mathcal{L}_2}\rightarrow0$ as $\beta\rightarrow\infty$. From \eqref{e:Accelerated_GF_N_N} and \eqref{e:error1_N_Loss}, $\dot{\vartheta}\in\mathcal{L}_2\cap\mathcal{L}_{\infty}$. If additionally $\phi\in\mathcal{L}_{\infty}$, then $e_y\in\mathcal{L}_2\cap\mathcal{L}_{\infty}$ and $(\theta-\vartheta)\in\mathcal{L}_2\cap\mathcal{L}_{\infty}$, and from \eqref{e:Accelerated_GF_N_N}, $\dot{\tilde{\theta}}\in\mathcal{L}_2\cap\mathcal{L}_{\infty}$, $\ddot{\tilde{\theta}}\in\mathcal{L}_{\infty}$ thus from Barbalat's lemma (Corollary \ref{c:Barbalat_Corollary}), $\lim_{t\rightarrow\infty}\dot{\tilde{\theta}}(t)=0$. If additionally $\dot{\phi}\in\mathcal{L}_{\infty}$, then from the time derivative of \eqref{e:error1_N}, it can be seen that $\dot{e}_y\in\mathcal{L}_{\infty}$ and from the time derivative of \eqref{e:Accelerated_GF_N_N}, $\ddot{\vartheta}\in\mathcal{L}_{\infty}$ and thus from Barbalat's lemma (Corollary \ref{c:Barbalat_Corollary}), $\lim_{t\rightarrow\infty}e_y(t)=0$, $\lim_{t\rightarrow\infty}(\theta(t)-\vartheta(t))=0$, and $\lim_{t\rightarrow\infty}\dot{\vartheta}(t)=0$. Given that $\lim_{t\rightarrow\infty}e_y(t)=0$, using \eqref{e:error1_N}, \eqref{e:error1_N_Loss} $\lim_{t\rightarrow\infty}L_t(\theta(t))=0$
\end{proof}

\subsection{Heavy Ball discrete time higher order tuner}
\begin{theorem}
    For the linear regression error model in \eqref{e:error1_N_discrete_N} with loss in \eqref{e:Squared_Loss_N}, running Algorithm \ref{alg:HOT_R_HB} with $\mu=0$, $0<\beta<2$, $0<\gamma\leq\frac{\beta(2-\beta)}{16}$ results in $(\vartheta_k-\theta^*)\in\ell_{\infty}$, $(\theta-\vartheta)\in\ell_{\infty}$, and $\sqrt{\frac{L_k(\theta_{k+1})}{\N_k}}\in\ell_2\cap\ell_{\infty}$. If in addition it is assumed that $\phi\in\ell_{\infty}$ then $\lim_{k\rightarrow\infty}L_k(\theta_{k+1})=0$.
\end{theorem}
\begin{proof}
    Consider the candidate Lyapunov function stated as
    \begin{equation}
        V_k=\frac{1}{\gamma}\lVert \vartheta_k-\theta^*\rVert^2+\frac{1}{\gamma}\lVert \theta_k-\vartheta_k\rVert^2.
    \end{equation}
    The increment $\Delta V_k:=V_{k+1}-V_k$ may then be expanded using \eqref{e:error1_N_discrete_N}, \eqref{e:Squared_Loss_N}, and Algorithm \ref{alg:HOT_R_HB} as
    \begingroup
    \allowdisplaybreaks
    \begin{align*}
        \Delta V_k&=\frac{1}{\gamma}\lVert \vartheta_{k+1}-\theta^*\rVert^2+\frac{1}{\gamma}\lVert \theta_{k+1}-\vartheta_{k+1}\rVert^2-\frac{1}{\gamma}\lVert \vartheta_k-\theta^*\rVert^2-\frac{1}{\gamma}\lVert \theta_k-\vartheta_k\rVert^2\\
        \Delta V_k&=\frac{1}{\gamma}\lVert (\vartheta_k-\theta^*)-\frac{\gamma}{\N_k}\nabla L_k(\theta_{k+1})\rVert^2-\frac{1}{\gamma}\lVert \vartheta_k-\theta^*\rVert^2\\
        &\quad +\frac{1}{\gamma}\lVert \theta_k-\beta(\theta_k-\vartheta_k)-\vartheta_k+\frac{\gamma}{\N_k}\nabla L_k(\theta_{k+1})\rVert^2-\frac{1}{\gamma}\lVert \theta_k-\vartheta_k\rVert^2\\
        \Delta V_k&=\frac{\gamma}{\N_k^2}\lVert\nabla L_k(\theta_{k+1})\rVert^2-\frac{2}{\N_k}(\vartheta_k-\theta^*)^T\nabla L_k(\theta_{k+1})\\
        &\quad+\frac{1}{\gamma}\lVert \theta_k-\vartheta_k\rVert^2-\frac{1}{\gamma}\lVert \theta_k-\vartheta_k\rVert^2-\frac{\beta(2-\beta)}{\gamma}\lVert \theta_k-\vartheta_k\rVert^2\\
        &\quad +\frac{2}{\N_k}(1-\beta)(\theta_k-\vartheta_k)^T\nabla L_k(\theta_{k+1})+\frac{\gamma}{\N_k^2}\lVert\nabla L_k(\theta_{k+1})\rVert^2\\
        \Delta V_k&=\frac{2\gamma}{\N_k^2}\lVert\nabla L_k(\theta_{k+1})\rVert^2-\frac{2}{\N_k}(\theta_{k+1}-\theta^*)^T\nabla L_k(\theta_{k+1})-\frac{\beta(2-\beta)}{\gamma}\lVert \theta_k-\vartheta_k\rVert^2\\
        &\quad +\frac{2}{\N_k}(1-\beta)(\theta_k-\vartheta_k)^T\nabla L_k(\theta_{k+1})-\frac{2}{\N_k}(\vartheta_k-\theta_{k+1})^T\nabla L_k(\theta_{k+1})\\
        \Delta V_k&=-2\left(1-\frac{\gamma\phi_k^T\phi_k}{\N_k}\right)\frac{\tilde{\theta}_{k+1}^T\nabla L_k(\theta_{k+1})}{\N_k}-\frac{\beta(2-\beta)}{\gamma}\lVert \theta_k-\vartheta_k\rVert^2\\
        &\quad +\frac{4}{\N_k}(1-\beta)(\theta_k-\vartheta_k)^T\nabla L_k(\theta_{k+1})\\
        \Delta V_k&=\frac{1}{\N_k}\left\{-2\left(1-\frac{\gamma\phi_k^T\phi_k}{\N_k}\right)\tilde{\theta}_{k+1}^T\nabla L_k(\theta_{k+1})-\frac{\beta(2-\beta)}{\gamma}\N_k\lVert \theta_k-\vartheta_k\rVert^2\right.\\
        &\quad\left.+4(1-\beta)(\theta_k-\vartheta_k)^T\nabla L_k(\theta_{k+1})\right\}\\
        \Delta V_k&\leq\frac{1}{\N_k}\left\{-\lVert\tilde{\theta}_{k+1}^T\phi_k\rVert^2-4\lVert\phi_k\rVert^2\lVert \theta_k-\vartheta_k\rVert^2+4\lVert\theta_k-\vartheta_k\rVert\lVert\phi_k\rVert\lVert\tilde{\theta}_{k+1}^T\phi_k\rVert\right.\\
        &\quad\left.-\frac{\beta(2-\beta)}{\gamma}\lVert \theta_k-\vartheta_k\rVert^2-12\lVert\phi_k\rVert^2\lVert \theta_k-\vartheta_k\rVert^2-\frac{7}{8}\lVert\tilde{\theta}_{k+1}^T\phi_k\rVert^2\right\}\\
        \Delta V_k&\leq\frac{1}{\N_k}\left\{-\frac{\beta(2-\beta)}{\gamma}\lVert \theta_k-\vartheta_k\rVert^2-12\lVert\phi_k\rVert^2\lVert \theta_k-\vartheta_k\rVert^2-\frac{7}{8}\lVert\tilde{\theta}_{k+1}^T\phi_k\rVert^2\right.\\
        &\quad\left.-\left[\lVert\tilde{\theta}_{k+1}^T\phi_k\rVert-2\lVert\phi_k\rVert\lVert \theta_k-\vartheta_k\rVert\right]^2\right\}\leq0.
    \end{align*}
    \endgroup
    
    Thus it can be concluded that $V$ is a Lyapunov function with $(\theta-\theta^*)\in\ell_{\infty}$ and $(\theta-\vartheta)\in\ell_{\infty}$. Using \eqref{e:error1_N_discrete_N} and $\N_k$ from Algorithm \ref{alg:HOT_R_HB}, $\frac{e_{y,k}}{\N_k}\in\ell_{\infty}$. Collecting $\Delta V_k$ terms from $t_0$ to $T$: $\sum_{k=t_0}^T\frac{7}{4}\lVert\sqrt{\frac{L_k(\theta_{k+1})}{\N_k}}\rVert^2\leq V_{t_0}-V_{T+1}<\infty$. Taking $T\rightarrow\infty$, it can be seen that $\sqrt{\frac{L_k(\theta_{k+1})}{\N_k}}\in\ell_2\cap\ell_{\infty}$ and therefore $\lim_{k\rightarrow\infty}\sqrt{\frac{L_k(\theta_{k+1})}{\N_k}}=0$. If additionally $\phi\in\ell_{\infty}$, then $\sqrt{L_k(\theta_{k+1})}\in\ell_2\cap\ell_{\infty}$ and therefore $\lim_{k\rightarrow\infty}\sqrt{L_k(\theta_{k+1})}=0$ and $\lim_{k\rightarrow\infty}L_k(\theta_{k+1})=0$.
\end{proof}

\subsection{Heavy Ball discrete time higher order tuner with regularization}
\begin{theorem}
    For the linear regression error model in \eqref{e:error1_N_discrete_N} with loss in \eqref{e:Squared_Loss_N}, running Algorithm \ref{alg:HOT_R_HB} with $0<\mu<1$, $0<\beta<2$, $0<\gamma\leq \frac{\beta(2-\beta)}{16+\mu\left(\frac{157}{48}\right)}$ results in $(\vartheta-\theta^*)\in\ell_{\infty}$, $(\theta-\vartheta)\in\ell_{\infty}$ and $V_k\leq \exp(-\mu c_3k)\left(V_0-\frac{c_4}{c_3}\right)+\frac{c_4}{c_3}$, where $V_k=\frac{1}{\gamma}\lVert \vartheta_k-\theta^*\rVert^2+\frac{1}{\gamma}\lVert \theta_k-\vartheta_k\rVert^2$, $c_3=\gamma\frac{1}{8}$, $c_4=\frac{189}{64}\lVert\theta^*-\theta_0\rVert^2$.
\end{theorem}
\begin{proof}
    Consider the candidate Lyapunov function stated as
    \begin{equation}
        V_k=\frac{1}{\gamma}\lVert \vartheta_k-\theta^*\rVert^2+\frac{1}{\gamma}\lVert \theta_k-\vartheta_k\rVert^2.
    \end{equation}
    The increment $\Delta V_k:=V_{k+1}-V_k$ may then be expanded using \eqref{e:error1_N_discrete_N}, \eqref{e:Squared_Loss_N}, and Algorithm \ref{alg:HOT_R_HB} as
    \begingroup
    \allowdisplaybreaks
    \begin{align*}
        \Delta V_k&=\frac{1}{\gamma}\lVert \vartheta_{k+1}-\theta^*\rVert^2+\frac{1}{\gamma}\lVert \theta_{k+1}-\vartheta_{k+1}\rVert^2-\frac{1}{\gamma}\lVert \vartheta_k-\theta^*\rVert^2-\frac{1}{\gamma}\lVert \theta_k-\vartheta_k\rVert^2\\
        \Delta V_k&=\frac{1}{\gamma}\lVert (\vartheta_k-\theta^*)-\frac{\gamma}{\N_k}\nabla L_k(\theta_{k+1})\rVert^2-\frac{1}{\gamma}\lVert \vartheta_k-\theta^*\rVert^2\\
        &\quad +\frac{1}{\gamma}\lVert \theta_k-\beta(\theta_k-\vartheta_k)-\vartheta_k+\frac{\gamma}{\N_k}\nabla L_k(\theta_{k+1})\rVert^2-\frac{1}{\gamma}\lVert \theta_k-\vartheta_k\rVert^2\\
        &\quad -\frac{2}{\gamma}\left[(\vartheta_k-\theta^*)-\frac{\gamma}{\N_k}\nabla L_k(\theta_{k+1})\right]^T\gamma\mu(\theta_{k+1}-\theta_0)+\gamma\mu^2\lVert\theta_{k+1}-\theta_0\rVert^2\\
        &\quad +\frac{2}{\gamma}\left[ \theta_k-\beta(\theta_k-\vartheta_k)-\vartheta_k+\frac{\gamma}{\N_k}\nabla L_k(\theta_{k+1})\right]^T\gamma\mu(\theta_{k+1}-\theta_0)+\gamma\mu^2\lVert\theta_{k+1}-\theta_0\rVert^2\\
        \Delta V_k&=\frac{\gamma}{\N_k^2}\lVert\nabla L_k(\theta_{k+1})\rVert^2-\frac{2}{\N_k}(\vartheta_k-\theta^*)^T\nabla L_k(\theta_{k+1})\\
        &\quad+\frac{1}{\gamma}\lVert \theta_k-\vartheta_k\rVert^2-\frac{1}{\gamma}\lVert \theta_k-\vartheta_k\rVert^2-\frac{\beta(2-\beta)}{\gamma}\lVert \theta_k-\vartheta_k\rVert^2\\
        &\quad +\frac{2}{\N_k}(1-\beta)(\theta_k-\vartheta_k)^T\nabla L_k(\theta_{k+1})+\frac{\gamma}{\N_k^2}\lVert\nabla L_k(\theta_{k+1})\rVert^2\\
        &\quad -2\left[(\vartheta_k-\theta^*)-\frac{\gamma}{\N_k}\nabla L_k(\theta_{k+1})\right]^T\mu(\theta_{k+1}-\theta_0)\\
        &\quad +2\left[ \theta_k-\beta(\theta_k-\vartheta_k)-\vartheta_k+\frac{\gamma}{\N_k}\nabla L_k(\theta_{k+1})\right]^T\mu(\theta_{k+1}-\theta_0)\\
        &\quad +2\gamma\mu^2\lVert\theta_{k+1}-\theta_0\rVert^2\\
        \Delta V_k&=\frac{2\gamma}{\N_k^2}\lVert\nabla L_k(\theta_{k+1})\rVert^2-\frac{2}{\N_k}(\theta_{k+1}-\theta^*)^T\nabla L_k(\theta_{k+1})-\frac{\beta(2-\beta)}{\gamma}\lVert \theta_k-\vartheta_k\rVert^2\\
        &\quad +\frac{2}{\N_k}(1-\beta)(\theta_k-\vartheta_k)^T\nabla L_k(\theta_{k+1})-\frac{2}{\N_k}(\vartheta_k-\theta_{k+1})^T\nabla L_k(\theta_{k+1})\\
        &\quad -2\left[(\vartheta_k-\theta^*)-\frac{\gamma}{\N_k}\nabla L_k(\theta_{k+1})\right]^T\mu(\theta_{k+1}-\theta_0)\\
        &\quad +2\left[ (1-\beta)(\theta_k-\vartheta_k)+\frac{\gamma}{\N_k}\nabla L_k(\theta_{k+1})\right]^T\mu(\theta_{k+1}-\theta_0)\\
        &\quad +2\gamma\mu^2\lVert(1-\beta)(\theta_k-\vartheta_k)+\vartheta_k-\theta_0\rVert^2\\
        \Delta V_k&=-2\left(1-\frac{\gamma\phi_k^T\phi_k}{\N_k}\right)\frac{\tilde{\theta}_{k+1}^T\nabla L_k(\theta_{k+1})}{\N_k}-\frac{\beta(2-\beta)}{\gamma}\lVert \theta_k-\vartheta_k\rVert^2\\
        &\quad +\frac{4}{\N_k}(1-\beta)(\theta_k-\vartheta_k)^T\nabla L_k(\theta_{k+1})\\
        &\quad +4\left[\frac{\gamma}{\N_k}\nabla L_k(\theta_{k+1})\right]^T\mu(\theta_{k+1}-\theta_0)\\
        &\quad +2\left[ (1-\beta)(\theta_k-\vartheta_k)-(\vartheta_k-\theta^*)\right]^T\mu(\theta_{k+1}-\theta_0)\\
        &\quad +2\gamma\mu^2\lVert(1-\beta)(\theta_k-\vartheta_k)+(\vartheta_k-\theta_0)\rVert^2\\
        \Delta V_k&=\frac{1}{\N_k}\left\{-2\left(1-\frac{\gamma\phi_k^T\phi_k}{\N_k}\right)\tilde{\theta}_{k+1}^T\nabla L_k(\theta_{k+1})-\frac{\beta(2-\beta)}{\gamma}\N_k\lVert \theta_k-\vartheta_k\rVert^2\right.\\
        &\quad \left.+4(1-\beta)(\theta_k-\vartheta_k)^T\nabla L_k(\theta_{k+1})\right\}\\
        &\quad +4\left[\frac{\gamma}{\N_k}\nabla L_k(\theta_{k+1})\right]^T\mu\left[\tilde{\theta}_{k+1}+\theta^*-\theta_0\right]\\
        &\quad +2\left[ (1-\beta)(\theta_k-\vartheta_k)-(\vartheta_k-\theta^*)\right]^T\mu\left[(1-\beta)(\theta_k-\vartheta_k)+(\vartheta_k-\theta^*)+(\theta^*-\theta_0)\right]\\
        &\quad +2\gamma\mu^2\lVert(1-\beta)(\theta_k-\vartheta_k)+(\vartheta_k-\theta^*)+(\theta^*-\theta_0)\rVert^2\\
        \Delta V_k&=\frac{1}{\N_k}\left\{-2\left(1-\frac{\gamma\phi_k^T\phi_k}{\N_k}\right)\tilde{\theta}_{k+1}^T\nabla L_k(\theta_{k+1})-\frac{\beta(2-\beta)}{\gamma}\N_k\lVert \theta_k-\vartheta_k\rVert^2\right.\\
        &\quad \left.+4(1-\beta)(\theta_k-\vartheta_k)^T\nabla L_k(\theta_{k+1})\right\}\\
        &\quad +4\mu\frac{\gamma}{\N_k}\tilde{\theta}_{k+1}^T\nabla L_k(\theta_{k+1})+4\mu\frac{\gamma}{\N_k}\left[(1-\beta)(\theta_k-\vartheta_k)+(\vartheta_k-\theta^*)\right]^T\phi_k\phi_k^T(\theta^*-\theta_0)\\
        &\quad +2\mu(1-\beta)^2\lVert\theta_k-\vartheta_k\rVert^2-2\mu\lVert\vartheta_k-\theta^*\rVert^2\\
        &\quad +2\mu(1-\beta)(\theta_k-\vartheta_k)^T(\theta^*-\theta_0)-2\mu(\vartheta_k-\theta^*)^T(\theta^*-\theta_0)\\
        &\quad +2\mu^2\gamma\lVert(1-\beta)(\theta_k-\vartheta_k)\rVert^2+2\mu^2\gamma\lVert\vartheta_k-\theta^*\rVert^2+2\mu^2\gamma\lVert\theta^*-\theta_0\rVert^2\\
        &\quad +4\mu^2\gamma(1-\beta)(\theta_k-\vartheta_k)^T(\vartheta_k-\theta^*)+4\mu^2\gamma(1-\beta)(\theta_k-\vartheta_k)^T(\theta^*-\theta_0)\\
        &\quad +4\mu^2\gamma(\vartheta_k-\theta^*)^T(\theta^*-\theta_0)\\
        \Delta V_k&=\frac{1}{\N_k}\left\{-2\left(1-\frac{\gamma\phi_k^T\phi_k}{\N_k}\right)\tilde{\theta}_{k+1}^T\nabla L_k(\theta_{k+1})-\frac{\beta(2-\beta)}{\gamma}\N_k\lVert \theta_k-\vartheta_k\rVert^2\right.\\
        &\quad \left.+4(1-\beta)(\theta_k-\vartheta_k)^T\nabla L_k(\theta_{k+1})\right\}\\
        &\quad +4\mu\frac{\gamma}{\N_k}\tilde{\theta}_{k+1}^T\nabla L_k(\theta_{k+1})\\
        &\quad -2\mu\lVert\vartheta_k-\theta^*\rVert^2+2\mu^2\gamma\lVert\vartheta_k-\theta^*\rVert^2\\
        &\quad +2\mu(1-\beta)^2\lVert\theta_k-\vartheta_k\rVert^2+2\mu^2\gamma\lVert(1-\beta)(\theta_k-\vartheta_k)\rVert^2\\
        &\quad +4\mu^2\gamma(1-\beta)(\theta_k-\vartheta_k)^T(\vartheta_k-\theta^*)\\
        &\quad -2\mu(\vartheta_k-\theta^*)^T(\theta^*-\theta_0)+4\mu^2\gamma(\vartheta_k-\theta^*)^T(\theta^*-\theta_0)\\
        &\quad +4\mu\frac{\gamma}{\N_k}(\vartheta_k-\theta^*)^T\phi_k\phi_k^T(\theta^*-\theta_0)\\
        &\quad +2\mu(1-\beta)(\theta_k-\vartheta_k)^T(\theta^*-\theta_0)+4\mu^2\gamma(1-\beta)(\theta_k-\vartheta_k)^T(\theta^*-\theta_0)\\
        &\quad +4\mu\frac{\gamma}{\N_k}(1-\beta)(\theta_k-\vartheta_k)^T\phi_k\phi_k^T(\theta^*-\theta_0)\\
        &\quad +2\mu^2\gamma\lVert\theta^*-\theta_0\rVert^2\\
        \Delta V_k&=\frac{1}{\N_k}\left\{-2\left(1-\frac{\gamma\phi_k^T\phi_k}{\N_k}\right)\tilde{\theta}_{k+1}^T\nabla L_k(\theta_{k+1})-\frac{\beta(2-\beta)}{\gamma}\N_k\lVert \theta_k-\vartheta_k\rVert^2\right.\\
        &\quad \left.+4(1-\beta)(\theta_k-\vartheta_k)^T\nabla L_k(\theta_{k+1})\right\}\\
        &\quad +4\mu\frac{\gamma}{\N_k}\tilde{\theta}_{k+1}^T\nabla L_k(\theta_{k+1})\\
        &\quad -\mu\left(2-2\mu\gamma\right)\lVert\vartheta_k-\theta^*\rVert^2\\
        &\quad +\mu\left(2(1-\beta)^2+2\mu\gamma(1-\beta)^2\right)\lVert\theta_k-\vartheta_k\rVert^2\\
        &\quad +\mu\left(4\mu\gamma(1-\beta)\right)(\theta_k-\vartheta_k)^T(\vartheta_k-\theta^*)\\
        &\quad +\mu\left(-2+4\mu\gamma\right)(\vartheta_k-\theta^*)^T(\theta^*-\theta_0)\\
        &\quad +\mu\left(4\frac{\gamma}{\N_k}\right)(\vartheta_k-\theta^*)^T\phi_k\phi_k^T(\theta^*-\theta_0)\\
        &\quad +\mu\left(2(1-\beta)+4\mu\gamma(1-\beta)\right)(\theta_k-\vartheta_k)^T(\theta^*-\theta_0)\\
        &\quad +\mu\left(4\frac{\gamma}{\N_k}(1-\beta)\right)(\theta_k-\vartheta_k)^T\phi_k\phi_k^T(\theta^*-\theta_0)\\
        &\quad +\mu\left(2\mu\gamma\right)\lVert\theta^*-\theta_0\rVert^2\\
        \Delta V_k&\leq\frac{1}{\N_k}\left\{-\lVert\tilde{\theta}_{k+1}^T\phi_k\rVert^2-4\lVert\phi_k\rVert^2\lVert \theta_k-\vartheta_k\rVert^2+4\lVert\theta_k-\vartheta_k\rVert\lVert\phi_k\rVert\lVert\tilde{\theta}_{k+1}^T\phi_k\rVert\right.\\
        &\quad\left.-16\lVert \theta_k-\vartheta_k\rVert^2-12\lVert\phi_k\rVert^2\lVert \theta_k-\vartheta_k\rVert^2-\frac{7}{8}\lVert\tilde{\theta}_{k+1}^T\phi_k\rVert^2\right\}-\mu\left(\frac{157}{48}\right)\lVert\theta_k-\vartheta_k\rVert^2\\
        &\quad +\frac{1}{4}\frac{1}{\N_k}\lVert\tilde{\theta}_{k+1}^T\phi_k\rVert^2\\
        &\quad -\mu\left(\frac{30}{16}\right)\lVert\vartheta_k-\theta^*\rVert^2\\
        &\quad +\mu\left(\frac{34}{16}\right)\lVert\theta_k-\vartheta_k\rVert^2\\
        &\quad +\mu\left(\frac{1}{4}\right)\lVert\theta_k-\vartheta_k\rVert\lVert\vartheta_k-\theta^*\rVert\\
        &\quad +\mu\left(\frac{9}{4}\right)\lVert\vartheta_k-\theta^*\rVert\lVert\theta^*-\theta_0\rVert\\
        &\quad +\mu\left(\frac{10}{4}\right)\lVert\theta_k-\vartheta_k\rVert\lVert\theta^*-\theta_0\rVert\\
        &\quad +\mu\left(\frac{1}{8}\right)\lVert\theta^*-\theta_0\rVert^2\\
        \Delta V_k&\leq\frac{1}{\N_k}\left\{-\lVert\tilde{\theta}_{k+1}^T\phi_k\rVert^2-4\lVert\phi_k\rVert^2\lVert \theta_k-\vartheta_k\rVert^2+4\lVert\theta_k-\vartheta_k\rVert\lVert\phi_k\rVert\lVert\tilde{\theta}_{k+1}^T\phi_k\rVert\right.\\
        &\quad\left.-16\lVert \theta_k-\vartheta_k\rVert^2-12\lVert\phi_k\rVert^2\lVert \theta_k-\vartheta_k\rVert^2-\frac{7}{8}\lVert\tilde{\theta}_{k+1}^T\phi_k\rVert^2\right\}-\mu\left(\frac{157}{48}\right)\lVert\theta_k-\vartheta_k\rVert^2\\
        &\quad +\frac{1}{4}\frac{1}{\N_k}\lVert\tilde{\theta}_{k+1}^T\phi_k\rVert^2\\
        &\quad -\mu\left(\frac{30}{16}\pm\frac{12}{16}\pm\frac{16}{16}\pm\frac{2}{16}\right)\lVert\vartheta_k-\theta^*\rVert^2\\
        &\quad +\mu\left(\frac{34}{16}\pm\frac{1}{48}\pm1\pm\frac{2}{16}\right)\lVert\theta_k-\vartheta_k\rVert^2\\
        &\quad +\mu\left(\frac{1}{4}\right)\lVert\theta_k-\vartheta_k\rVert\lVert\vartheta_k-\theta^*\rVert\\
        &\quad +\mu\left(\frac{9}{4}\right)\lVert\vartheta_k-\theta^*\rVert\lVert\theta^*-\theta_0\rVert\\
        &\quad +\mu\left(\frac{10}{4}\right)\lVert\theta_k-\vartheta_k\rVert\lVert\theta^*-\theta_0\rVert\\
        &\quad +\mu\left(\frac{1}{8}\pm\frac{81}{64}\pm\frac{100}{64}\right)\lVert\theta^*-\theta_0\rVert^2\\
        \Delta V_k&\leq\frac{1}{\N_k}\left\{-\left[\lVert\tilde{\theta}_{k+1}^T\phi_k\rVert-2\lVert\phi_k\rVert\lVert \theta_k-\vartheta_k\rVert\right]^2\right.\\
        &\quad\left.-16\lVert \theta_k-\vartheta_k\rVert^2-12\lVert\phi_k\rVert^2\lVert \theta_k-\vartheta_k\rVert^2-\frac{7}{8}\lVert\tilde{\theta}_{k+1}^T\phi_k\rVert^2\right\}-\mu\left(\frac{157}{48}\right)\lVert\theta_k-\vartheta_k\rVert^2\\
        &\quad +\frac{1}{4}\frac{1}{\N_k}\lVert\tilde{\theta}_{k+1}^T\phi_k\rVert^2+\mu\left(\frac{157}{48}\right)\lVert\theta_k-\vartheta_k\rVert^2\\
        &\quad -\mu\gamma\frac{1}{8}V_k+\mu\frac{189}{64}\lVert\theta^*-\theta_0\rVert^2\\
        &\quad -\mu\left[\frac{\sqrt{3}}{2}\lVert\vartheta_k-\theta^*\rVert-\frac{1}{4\sqrt{3}}\lVert\theta_k-\vartheta_k\rVert\right]^2\\
        &\quad -\mu\left[\lVert\vartheta_k-\theta^*\rVert-\frac{9}{8}\lVert\theta^*-\theta_0\rVert\right]^2\\
        &\quad -\mu\left[\lVert\theta_k-\vartheta_k\rVert-\frac{10}{8}\lVert\theta^*-\theta_0\rVert\right]^2\\
        \Delta V_k&\leq-\frac{L_k(\theta_{k+1})}{\N_k}-\mu\underbrace{\gamma\frac{1}{8}}_{c_3}V_k+\mu\underbrace{\frac{189}{64}\lVert\theta^*-\theta_0\rVert^2}_{c_4}.
    \end{align*}
    \endgroup
    From the bound on $\Delta V_k$, it can be noted that $\Delta V_k<0$ in $D^c$, where the compact set $D$ is defined as
    \begin{equation*}
        D=\left\{V\middle|V\leq\frac{c_4}{c_3}\right\}.
    \end{equation*}
    Therefore $V\in\ell_{\infty}$, $(\vartheta-\theta^*)\in\ell_{\infty}$, and $(\theta-\vartheta)\in\ell_{\infty}$. Furthermore, from the bound on $\Delta V_k$,
    \begin{align*}
        V_{k+1}&\leq(1-\mu c_3)V_k+\mu c_4\\
        V_{k+1}&\leq(1-\mu c_3)\left(V_k-\frac{c_4}{c_3}\right)+(1-\mu c_3)\frac{c_4}{c_3}+\mu c_4\\
        V_{k+1}&\leq(1-\mu c_3)\left(V_k-\frac{c_4}{c_3}\right)+\frac{c_4}{c_3}\\
        V_{k+1}-\frac{c_4}{c_3}&\leq(1-\mu c_3)\left(V_k-\frac{c_4}{c_3}\right)
    \end{align*}
    Collecting terms,
    \begin{align*}
        V_k-\frac{c_4}{c_3}&\leq(1-\mu c_3)^k\left(V_0-\frac{c_4}{c_3}\right)\\
        V_k-\frac{c_4}{c_3}&\leq\exp(-\mu c_3k)\left(V_0-\frac{c_4}{c_3}\right)\\
        V_k&\leq\exp(-\mu c_3k)\left(V_0-\frac{c_4}{c_3}\right)+\frac{c_4}{c_3}.
    \end{align*}
\end{proof}

\subsection{Nesterov discrete time higher order tuner}

\begin{manualtheorem}{\ref{th:HOT_paper_Discrete} from Main Text (with proof)}
    For the linear regression error model in \eqref{e:error1_N_discrete_N} with loss in \eqref{e:Squared_Loss_N}, running Algorithm \ref{alg:HOT_R} with $\mu=0$, $0<\beta<1$, $0<\gamma\leq\frac{\beta(2-\beta)}{16+\beta^2}$ results in $(\vartheta_k-\theta^*)\in\ell_{\infty}$, $(\theta-\vartheta)\in\ell_{\infty}$, and $\sqrt{\frac{L_k(\theta_{k+1})}{\N_k}}\in\ell_2\cap\ell_{\infty}$. If in addition it is assumed that $\phi\in\ell_{\infty}$ then $\lim_{k\rightarrow\infty}L_k(\theta_{k+1})=0$.
\end{manualtheorem}
\begin{proof}
    Consider the candidate Lyapunov function stated as
    \begin{equation}
        V_k=\frac{1}{\gamma}\lVert \vartheta_k-\theta^*\rVert^2+\frac{1}{\gamma}\lVert \theta_k-\vartheta_k\rVert^2.
    \end{equation}
    The increment $\Delta V_k:=V_{k+1}-V_k$ may then be expanded using \eqref{e:error1_N_discrete_N}, \eqref{e:Squared_Loss_N}, and Algorithm \ref{alg:HOT_R} as
    \begingroup
    \allowdisplaybreaks
    \begin{align*}
        \Delta V_k&=\frac{1}{\gamma}\lVert \vartheta_{k+1}-\theta^*\rVert^2+\frac{1}{\gamma}\lVert \theta_{k+1}-\vartheta_{k+1}\rVert^2-\frac{1}{\gamma}\lVert \vartheta_k-\theta^*\rVert^2-\frac{1}{\gamma}\lVert \theta_k-\vartheta_k\rVert^2\\
        \Delta V_k&=\frac{1}{\gamma}\lVert (\vartheta_k-\theta^*)-\frac{\gamma}{\N_k}\nabla L_k(\theta_{k+1})\rVert^2-\frac{1}{\gamma}\lVert \vartheta_k-\theta^*\rVert^2\\
        &\quad +\frac{1}{\gamma}\lVert \bar{\theta}_k-\beta(\bar{\theta}_k-\vartheta_k)-\vartheta_k+\frac{\gamma}{\N_k}\nabla L_k(\theta_{k+1})\rVert^2-\frac{1}{\gamma}\lVert \theta_k-\vartheta_k\rVert^2\\
        \Delta V_k&=\frac{\gamma}{\N_k^2}\lVert\nabla L_k(\theta_{k+1})\rVert^2-\frac{2}{\N_k}(\vartheta_k-\theta^*)^T\nabla L_k(\theta_{k+1})\\
        &\quad+\frac{1}{\gamma}\lVert \bar{\theta}_k-\vartheta_k\rVert^2-\frac{1}{\gamma}\lVert \theta_k-\vartheta_k\rVert^2-\frac{\beta(2-\beta)}{\gamma}\lVert \bar{\theta}_k-\vartheta_k\rVert^2\\
        &\quad +\frac{2}{\N_k}(1-\beta)(\bar{\theta}_k-\vartheta_k)^T\nabla L_k(\theta_{k+1})+\frac{\gamma}{\N_k^2}\lVert\nabla L_k(\theta_{k+1})\rVert^2\\
        \Delta V_k&=\frac{2\gamma}{\N_k^2}\lVert\nabla L_k(\theta_{k+1})\rVert^2-\frac{2}{\N_k}(\theta_{k+1}-\theta^*)^T\nabla L_k(\theta_{k+1})\\
        &\quad+\frac{1}{\gamma}\lVert \bar{\theta}_k-\vartheta_k\rVert^2-\frac{1}{\gamma}\lVert \theta_k-\vartheta_k\rVert^2-\frac{\beta(2-\beta)}{\gamma}\lVert \bar{\theta}_k-\vartheta_k\rVert^2\\
        &\quad +\frac{2}{\N_k}(1-\beta)(\bar{\theta}_k-\vartheta_k)^T\nabla L_k(\theta_{k+1})-\frac{2}{\N_k}(\vartheta_k-\theta_{k+1})^T\nabla L_k(\theta_{k+1})\\
        \Delta V_k&=-2\left(1-\frac{\gamma\phi_k^T\phi_k}{\N_k}\right)\frac{\tilde{\theta}_{k+1}^T\nabla L_k(\theta_{k+1})}{\N_k}\\
        &\quad+\frac{1}{\gamma}\lVert \bar{\theta}_k-\vartheta_k\rVert^2-\frac{1}{\gamma}\lVert \theta_k-\vartheta_k\rVert^2-\frac{\beta(2-\beta)}{\gamma}\lVert \bar{\theta}_k-\vartheta_k\rVert^2\\
        &\quad +\frac{4}{\N_k}(1-\beta)(\bar{\theta}_k-\vartheta_k)^T\nabla L_k(\theta_{k+1})\\
        \Delta V_k&=-2\left(1-\frac{\gamma\phi_k^T\phi_k}{\N_k}\right)\frac{\tilde{\theta}_{k+1}^T\nabla L_k(\theta_{k+1})}{\N_k}\\
        &\quad+\frac{\gamma\beta^2}{\N_k^2}\lVert\nabla L_k(\theta_k)\rVert^2-\frac{2\beta}{\N_k}(\theta_k-\vartheta_k)^T\nabla L_k(\theta_k)\\
        &\quad -\frac{\beta(2-\beta)}{\gamma}\lVert \bar{\theta}_k-\vartheta_k\rVert^2+\frac{4}{\N_k}(1-\beta)(\bar{\theta}_k-\vartheta_k)^T\nabla L_k(\theta_{k+1})\\
        \Delta V_k&=\frac{1}{\N_k}\left\{-2\left(1-\frac{\gamma\phi_k^T\phi_k}{\N_k}\right)\tilde{\theta}_{k+1}^T\nabla L_k(\theta_{k+1})+\frac{\gamma\beta^2}{\N_k}\lVert\nabla L_k(\theta_k)\rVert^2-2\beta(\theta_k-\vartheta_k)^T\nabla L_k(\theta_k)\right.\\
        &\quad\left.-\frac{\beta(2-\beta)\N_k}{\gamma}\lVert \bar{\theta}_k-\vartheta_k\rVert^2+4(1-\beta)(\bar{\theta}_k-\vartheta_k)^T\nabla L_k(\theta_{k+1})\right\}\\
        \Delta V_k&=\frac{1}{\N_k}\left\{-2\left(1-\frac{\gamma\phi_k^T\phi_k}{\N_k}\right)\tilde{\theta}_{k+1}^T\nabla L_k(\theta_{k+1})-\frac{\gamma\beta^2}{\N_k}\lVert\nabla L_k(\theta_k)\rVert^2-2\beta(\bar{\theta}_k-\vartheta_k)^T\nabla L_k(\theta_k)\right.\\
        &\quad\left.-\frac{\beta(2-\beta)\N_k}{\gamma}\lVert \bar{\theta}_k-\vartheta_k\rVert^2+4(1-\beta)(\bar{\theta}_k-\vartheta_k)^T\nabla L_k(\theta_{k+1})\right\}\\
        \Delta V_k&=\frac{1}{\N_k}\left\{-2\left(1-\frac{\gamma\phi_k^T\phi_k}{\N_k}\right)\tilde{\theta}_{k+1}^T\nabla L_k(\theta_{k+1})-\frac{\gamma\beta^2}{\N_k}\lVert\nabla L_k(\theta_k)\rVert^2-2\beta(\bar{\theta}_k-\vartheta_k)^T\phi_k\phi^T_k\tilde{\theta}_k\right.\\
        &\quad\left.-\frac{\beta(2-\beta)\N_k}{\gamma}\lVert \bar{\theta}_k-\vartheta_k\rVert^2+4(1-\beta)(\bar{\theta}_k-\vartheta_k)^T\nabla L_k(\theta_{k+1})\right\}\\
        \Delta V_k&=\frac{1}{\N_k}\left\{-2\left(1-\frac{\gamma\phi_k^T\phi_k}{\N_k}\right)\tilde{\theta}_{k+1}^T\nabla L_k(\theta_{k+1})-\frac{\gamma\beta^2}{\N_k}\lVert\nabla L_k(\theta_k)\rVert^2\right.\\
        &\quad\left.-\frac{\beta(2-\beta)\N_k}{\gamma}\lVert \bar{\theta}_k-\vartheta_k\rVert^2+4(1-\beta)(\bar{\theta}_k-\vartheta_k)^T\nabla L_k(\theta_{k+1})\right.\\
        &\quad\left.-2\beta(\bar{\theta}_k-\vartheta_k)^T\phi_k\phi^T_k\left[\theta_k-\theta^*+(1-\beta)\bar{\theta}_k+\beta\vartheta_k-(1-\beta)\bar{\theta}_k-\beta\vartheta_k\right]\right\}\\
        \Delta V_k&=\frac{1}{\N_k}\left\{-2\left(1-\frac{\gamma\phi_k^T\phi_k}{\N_k}\right)\tilde{\theta}_{k+1}^T\nabla L_k(\theta_{k+1})-\frac{\gamma\beta^2}{\N_k}\lVert\nabla L_k(\theta_k)\rVert^2\right.\\
        &\quad\left.-2\beta(\bar{\theta}_k-\vartheta_k)^T\nabla L_k(\theta_{k+1})-\frac{\beta(2-\beta)\N_k}{\gamma}\lVert \bar{\theta}_k-\vartheta_k\rVert^2\right.\\
        &\quad\left.+4(1-\beta)(\bar{\theta}_k-\vartheta_k)^T\nabla L_k(\theta_{k+1})-2\beta(\bar{\theta}_k-\vartheta_k)^T\phi_k\phi^T_k\left[\theta_k-\bar{\theta}_k+\beta(\bar{\theta}_k-\vartheta_k)\right]\right\}\\
        \Delta V_k&=\frac{1}{\N_k}\left\{-2\left(1-\frac{\gamma\phi_k^T\phi_k}{\N_k}\right)\tilde{\theta}_{k+1}^T\nabla L_k(\theta_{k+1})-\frac{\gamma\beta^2}{\N_k}\lVert\nabla L_k(\theta_k)\rVert^2\right.\\
        &\quad\left.-2\gamma\beta^2(\bar{\theta}_k-\vartheta_k)^T\frac{\phi_k\phi^T_k}{\N_k}\nabla L_k(\theta_k)-\frac{\beta(2-\beta)\N_k}{\gamma}\lVert \bar{\theta}_k-\vartheta_k\rVert^2\right.\\
        &\quad\left.+4(1-\frac{3}{2}\beta)(\bar{\theta}_k-\vartheta_k)^T\nabla L_k(\theta_{k+1})-2\beta^2(\bar{\theta}_k-\vartheta_k)^T\phi_k\phi^T_k(\bar{\theta}_k-\vartheta_k)\right\}\\
        \Delta V_k&\leq\frac{1}{\N_k}\left\{-2\left(1-\frac{\gamma\phi_k^T\phi_k}{\N_k}\right)\tilde{\theta}_{k+1}^T\nabla L_k(\theta_{k+1})-16\lVert\phi_k\rVert^2\lVert \bar{\theta}_k-\vartheta_k\rVert^2\right.\\
        &\quad\left.+4(1-\frac{3}{2}\beta)(\bar{\theta}_k-\vartheta_k)^T\nabla L_k(\theta_{k+1})\right.\\
        &\quad\left.-\frac{\gamma\beta^2}{\N_k}\lVert\nabla L_k(\theta_k)\rVert^2-\beta^2\lVert\phi_k\rVert^2\lVert \bar{\theta}_k-\vartheta_k\rVert^2-2\beta^2(\bar{\theta}_k-\vartheta_k)^T\gamma\frac{\phi_k\phi^T_k}{\N_k}\nabla L_k(\theta_k)\right.\\
        &\quad\left.-(16+\beta^2)\lVert \bar{\theta}_k-\vartheta_k\rVert^2-2\beta^2(\bar{\theta}_k-\vartheta_k)^T\phi_k\phi^T_k(\bar{\theta}_k-\vartheta_k)\right\}\\
        \Delta V_k&\leq\frac{1}{\N_k}\left\{-\frac{30}{16}\lVert\tilde{\theta}_{k+1}^T\phi_k\rVert^2-16\lVert\phi_k\rVert^2\lVert \bar{\theta}_k-\vartheta_k\rVert^2+8\lVert\bar{\theta}_k-\vartheta_k\rVert\lVert\phi_k\rVert\lVert\tilde{\theta}_{k+1}^T\phi_k\rVert\right.\\
        &\quad\left.-\frac{\gamma\beta^2}{\N_k}\lVert\nabla L_k(\theta_k)\rVert^2-\beta^2\lVert\phi_k\rVert^2\lVert \bar{\theta}_k-\vartheta_k\rVert^2+2\beta^2\lVert\bar{\theta}_k-\vartheta_k\rVert\frac{\lVert\sqrt{\gamma}\phi_k\rVert^2}{\N_k}\lVert\nabla L_k(\theta_k)\rVert\right.\\
        &\quad\left.-(16+\beta^2)\lVert \bar{\theta}_k-\vartheta_k\rVert^2-2\beta^2(\bar{\theta}_k-\vartheta_k)^T\phi_k\phi^T_k(\bar{\theta}_k-\vartheta_k)\right\}\\
        \Delta V_k&\leq\frac{1}{\N_k}\left\{-\lVert\tilde{\theta}_{k+1}^T\phi_k\rVert^2-16\lVert\phi_k\rVert^2\lVert \bar{\theta}_k-\vartheta_k\rVert^2+8\lVert\bar{\theta}_k-\vartheta_k\rVert\lVert\phi_k\rVert\lVert\tilde{\theta}_{k+1}^T\phi_k\rVert\right.\\
        &\quad\left.-\frac{\gamma\beta^2}{\N_k}\lVert\nabla L_k(\theta_k)\rVert^2-\beta^2\lVert\phi_k\rVert^2\lVert \bar{\theta}_k-\vartheta_k\rVert^2+\frac{2\sqrt{\gamma}\beta^2}{\sqrt{\N_k}}\lVert\bar{\theta}_k-\vartheta_k\rVert\lVert\phi_k\rVert\lVert\nabla L_k(\theta_k)\rVert\right.\\
        &\quad\left.-\frac{7}{8}\lVert\tilde{\theta}_{k+1}^T\phi_k\rVert^2-(16+\beta^2)\lVert \bar{\theta}_k-\vartheta_k\rVert^2-2\beta^2(\bar{\theta}_k-\vartheta_k)^T\phi_k\phi^T_k(\bar{\theta}_k-\vartheta_k)\right\}\\
        \Delta V_k&\leq\frac{1}{\N_k}\left\{-\left[\lVert\tilde{\theta}_{k+1}^T\phi_k\rVert-4\lVert\phi_k\rVert\lVert \bar{\theta}_k-\vartheta_k\rVert\right]^2\right.\\
        &\quad\left.-\left[\frac{\sqrt{\gamma}\beta}{\sqrt{\N_k}}\lVert\nabla L_k(\theta_k)\rVert-\beta\lVert\phi_k\rVert\lVert \bar{\theta}_k-\vartheta_k\rVert\right]^2\right.\\
        &\quad\left.-\frac{7}{8}\lVert\tilde{\theta}_{k+1}^T\phi_k\rVert^2-(16+\beta^2)\lVert \bar{\theta}_k-\vartheta_k\rVert^2-2\beta^2\lVert(\bar{\theta}_k-\vartheta_k)^T\phi_k\rVert^2\right\}\leq0.
    \end{align*}
    \endgroup
    
    Thus it can be concluded that $V$ is a Lyapunov function with $(\theta-\theta^*)\in\ell_{\infty}$ and $(\theta-\vartheta)\in\ell_{\infty}$. Using \eqref{e:error1_N_discrete_N} and $\N_k$ from Algorithm \ref{alg:HOT_R}, $\frac{e_{y,k}}{\N_k}\in\ell_{\infty}$. Collecting $\Delta V_k$ terms from $t_0$ to $T$: $\sum_{k=t_0}^T\frac{7}{4}\lVert\sqrt{\frac{L_k(\theta_{k+1})}{\N_k}}\rVert^2\leq V_{t_0}-V_{T+1}<\infty$. Taking $T\rightarrow\infty$, it can be seen that $\sqrt{\frac{L_k(\theta_{k+1})}{\N_k}}\in\ell_2\cap\ell_{\infty}$ and therefore $\lim_{k\rightarrow\infty}\sqrt{\frac{L_k(\theta_{k+1})}{\N_k}}=0$. If additionally $\phi\in\ell_{\infty}$, then $\sqrt{L_k(\theta_{k+1})}\in\ell_2\cap\ell_{\infty}$ and therefore $\lim_{k\rightarrow\infty}\sqrt{L_k(\theta_{k+1})}=0$ and $\lim_{k\rightarrow\infty}L_k(\theta_{k+1})=0$.
\end{proof}

\subsection{Nesterov discrete time higher order tuner algorithm with regularization}

\begin{manualtheorem}{\ref{th:HOT_paper_Full_Alg} from Main Text (with proof)}
    For the linear regression error model in \eqref{e:error1_N_discrete_N} with loss in \eqref{e:Squared_Loss_N}, running Algorithm \ref{alg:HOT_R} with $0<\mu<1$, $0<\beta<1$, $0<\gamma\leq \frac{\beta(2-\beta)}{16+\beta^2+\mu\left(\frac{57\beta+1}{16\beta}\right)}$ results in $(\vartheta-\theta^*)\in\ell_{\infty}$, $(\theta-\vartheta)\in\ell_{\infty}$ and $V_k\leq \exp(-\mu c_1k)\left(V_0-\frac{c_2}{c_1}\right)+\frac{c_2}{c_1}$, where $V_k=\frac{1}{\gamma}\lVert \vartheta_k-\theta^*\rVert^2+\frac{1}{\gamma}\lVert \theta_k-\vartheta_k\rVert^2$, $c_1=\gamma\beta\frac{10}{16}$, $c_2=\left(\frac{3570\beta+896}{224\beta}\right)\lVert\theta^*-\theta_0\rVert^2$.
\end{manualtheorem}
\begin{proof}
    Consider the candidate Lyapunov function stated as
    \begin{equation}
        V_k=\frac{1}{\gamma}\lVert \vartheta_k-\theta^*\rVert^2+\frac{1}{\gamma}\lVert \theta_k-\vartheta_k\rVert^2.
    \end{equation}
    The increment $\Delta V_k:=V_{k+1}-V_k$ may then be expanded using \eqref{e:error1_N_discrete_N}, \eqref{e:Squared_Loss_N}, and Algorithm \ref{alg:HOT_R} as
    \begingroup
    \allowdisplaybreaks
    \begin{align*}
        \Delta V_k&=\frac{1}{\gamma}\lVert \vartheta_{k+1}-\theta^*\rVert^2+\frac{1}{\gamma}\lVert \theta_{k+1}-\vartheta_{k+1}\rVert^2-\frac{1}{\gamma}\lVert \vartheta_k-\theta^*\rVert^2-\frac{1}{\gamma}\lVert \theta_k-\vartheta_k\rVert^2\\
        \Delta V_k&=\frac{1}{\gamma}\lVert (\vartheta_k-\theta^*)-\frac{\gamma}{\N_k}\nabla L_k(\theta_{k+1})\rVert^2-\frac{1}{\gamma}\lVert \vartheta_k-\theta^*\rVert^2\\
        &\quad +\frac{1}{\gamma}\lVert \bar{\theta}_k-\beta(\bar{\theta}_k-\vartheta_k)-\vartheta_k+\frac{\gamma}{\N_k}\nabla L_k(\theta_{k+1})\rVert^2-\frac{1}{\gamma}\lVert \theta_k-\vartheta_k\rVert^2\\
        &\quad -\frac{2}{\gamma}\left[(\vartheta_k-\theta^*)-\frac{\gamma}{\N_k}\nabla L_k(\theta_{k+1})\right]^T\gamma\mu(\theta_{k+1}-\theta_0)+\gamma\mu^2\lVert\theta_{k+1}-\theta_0\rVert^2\\
        &\quad +\frac{2}{\gamma}\left[ \bar{\theta}_k-\beta(\bar{\theta}_k-\vartheta_k)-\vartheta_k+\frac{\gamma}{\N_k}\nabla L_k(\theta_{k+1})\right]^T\gamma\mu(\theta_{k+1}-\theta_0)+\gamma\mu^2\lVert\theta_{k+1}-\theta_0\rVert^2\\
        \Delta V_k&=\frac{\gamma}{\N_k^2}\lVert\nabla L_k(\theta_{k+1})\rVert^2-\frac{2}{\N_k}(\vartheta_k-\theta^*)^T\nabla L_k(\theta_{k+1})\\
        &\quad+\frac{1}{\gamma}\lVert \bar{\theta}_k-\vartheta_k\rVert^2-\frac{1}{\gamma}\lVert \theta_k-\vartheta_k\rVert^2-\frac{\beta(2-\beta)}{\gamma}\lVert \bar{\theta}_k-\vartheta_k\rVert^2\\
        &\quad +\frac{2}{\N_k}(1-\beta)(\bar{\theta}_k-\vartheta_k)^T\nabla L_k(\theta_{k+1})+\frac{\gamma}{\N_k^2}\lVert\nabla L_k(\theta_{k+1})\rVert^2\\
        &\quad -2\left[(\vartheta_k-\theta^*)-\frac{\gamma}{\N_k}\nabla L_k(\theta_{k+1})\right]^T\mu(\theta_{k+1}-\theta_0)\\
        &\quad +2\left[ \bar{\theta}_k-\beta(\bar{\theta}_k-\vartheta_k)-\vartheta_k+\frac{\gamma}{\N_k}\nabla L_k(\theta_{k+1})\right]^T\mu(\theta_{k+1}-\theta_0)\\
        &\quad +2\gamma\mu^2\lVert\theta_{k+1}-\theta_0\rVert^2\\
        \Delta V_k&=\frac{2\gamma}{\N_k^2}\lVert\nabla L_k(\theta_{k+1})\rVert^2-\frac{2}{\N_k}(\theta_{k+1}-\theta^*)^T\nabla L_k(\theta_{k+1})\\
        &\quad+\frac{1}{\gamma}\lVert \bar{\theta}_k-\vartheta_k\rVert^2-\frac{1}{\gamma}\lVert \theta_k-\vartheta_k\rVert^2-\frac{\beta(2-\beta)}{\gamma}\lVert \bar{\theta}_k-\vartheta_k\rVert^2\\
        &\quad +\frac{2}{\N_k}(1-\beta)(\bar{\theta}_k-\vartheta_k)^T\nabla L_k(\theta_{k+1})-\frac{2}{\N_k}(\vartheta_k-\theta_{k+1})^T\nabla L_k(\theta_{k+1})\\
        &\quad -2\left[(\vartheta_k-\theta^*)-\frac{\gamma}{\N_k}\nabla L_k(\theta_{k+1})\right]^T\mu(\theta_{k+1}-\theta_0)\\
        &\quad +2\left[ (1-\beta)(\bar{\theta}_k-\vartheta_k)+\frac{\gamma}{\N_k}\nabla L_k(\theta_{k+1})\right]^T\mu(\theta_{k+1}-\theta_0)\\
        &\quad +2\gamma\mu^2\lVert(1-\beta)(\bar{\theta}_k-\vartheta_k)+\vartheta_k-\theta_0\rVert^2\\
        \Delta V_k&=-2\left(1-\frac{\gamma\phi_k^T\phi_k}{\N_k}\right)\frac{\tilde{\theta}_{k+1}^T\nabla L_k(\theta_{k+1})}{\N_k}\\
        &\quad+\frac{1}{\gamma}\lVert \bar{\theta}_k-\vartheta_k\rVert^2-\frac{1}{\gamma}\lVert \theta_k-\vartheta_k\rVert^2-\frac{\beta(2-\beta)}{\gamma}\lVert \bar{\theta}_k-\vartheta_k\rVert^2\\
        &\quad +\frac{4}{\N_k}(1-\beta)(\bar{\theta}_k-\vartheta_k)^T\nabla L_k(\theta_{k+1})\\
        &\quad +4\left[\frac{\gamma}{\N_k}\nabla L_k(\theta_{k+1})\right]^T\mu(\theta_{k+1}-\theta_0)\\
        &\quad +2\left[ (1-\beta)(\bar{\theta}_k-\vartheta_k)-(\vartheta_k-\theta^*)\right]^T\mu(\theta_{k+1}-\theta_0)\\
        &\quad +2\gamma\mu^2\lVert(1-\beta)(\bar{\theta}_k-\vartheta_k)+(\vartheta_k-\theta_0)\rVert^2\\
        \Delta V_k&=-2\left(1-\frac{\gamma\phi_k^T\phi_k}{\N_k}\right)\frac{\tilde{\theta}_{k+1}^T\nabla L_k(\theta_{k+1})}{\N_k}\\
        &\quad+\frac{\gamma\beta^2}{\N_k^2}\lVert\nabla L_k(\theta_k)\rVert^2-\frac{2\beta}{\N_k}(\theta_k-\vartheta_k)^T\nabla L_k(\theta_k)\\
        &\quad -\frac{\beta(2-\beta)}{\gamma}\lVert \bar{\theta}_k-\vartheta_k\rVert^2+\frac{4}{\N_k}(1-\beta)(\bar{\theta}_k-\vartheta_k)^T\nabla L_k(\theta_{k+1})\\
        &\quad +4\left[\frac{\gamma}{\N_k}\nabla L_k(\theta_{k+1})\right]^T\mu\left[\tilde{\theta}_{k+1}+(\theta^*-\theta_0)\right]\\
        &\quad +2\left[ (1-\beta)(\bar{\theta}_k-\vartheta_k)-(\vartheta_k-\theta^*)\right]^T\mu\left[(1-\beta)(\bar{\theta}_k-\vartheta_k)+(\vartheta_k-\theta^*)+(\theta^*-\theta_0)\right]\\
        &\quad +2\gamma\mu^2\lVert(1-\beta)(\bar{\theta}_k-\vartheta_k)+(\vartheta_k-\theta^*)+(\theta^*-\theta_0)\rVert^2\\
        &\quad -\frac{2}{\gamma}\left[\theta_k-\vartheta_k-\frac{\gamma\beta}{\N_k}\nabla L_k(\theta_k)\right]^T\gamma\beta\mu(\theta_k-\theta_0)+\gamma\beta^2\mu^2\lVert\theta_k-\theta_0\rVert^2\\
        \Delta V_k&=\frac{1}{\N_k}\left\{-2\left(1-\frac{\gamma\phi_k^T\phi_k}{\N_k}\right)\tilde{\theta}_{k+1}^T\nabla L_k(\theta_{k+1})+\frac{\gamma\beta^2}{\N_k}\lVert\nabla L_k(\theta_k)\rVert^2-2\beta(\theta_k-\vartheta_k)^T\nabla L_k(\theta_k)\right.\\
        &\quad\left.-\frac{\beta(2-\beta)\N_k}{\gamma}\lVert \bar{\theta}_k-\vartheta_k\rVert^2+4(1-\beta)(\bar{\theta}_k-\vartheta_k)^T\nabla L_k(\theta_{k+1})\right\}\\
        &\quad +4\mu\frac{\gamma}{\N_k}\tilde{\theta}_{k+1}^T\nabla L_k(\theta_{k+1})+4\mu\frac{\gamma}{\N_k}\tilde{\theta}_{k+1}^T\phi_k\phi_k^T(\theta^*-\theta_0)\\
        &\quad +2\mu(1-\beta)^2\lVert\bar{\theta}_k-\vartheta_k\rVert^2-2\mu\lVert\vartheta_k-\theta^*\rVert^2\\
        &\quad +2\mu (1-\beta)(\bar{\theta}_k-\vartheta_k)^T(\theta^*-\theta_0)-2\mu(\vartheta_k-\theta^*)^T(\theta^*-\theta_0)\\
        &\quad +2\gamma\mu^2\lVert(1-\beta)(\bar{\theta}_k-\vartheta_k)\rVert^2+2\gamma\mu^2\lVert\vartheta_k-\theta^*\rVert^2+2\gamma\mu^2\lVert\theta^*-\theta_0\rVert^2\\
        &\quad +4\gamma\mu^2(1-\beta)(\bar{\theta}_k-\vartheta_k)^T(\vartheta_k-\theta^*)+4\gamma\mu^2(1-\beta)(\bar{\theta}_k-\vartheta_k)^T(\theta^*-\theta_0)\\
        &\quad +4\gamma\mu^2(\vartheta_k-\theta^*)^T(\theta^*-\theta_0)\\
        &\quad +2\left[\frac{\gamma\beta}{\N_k}\nabla L_k(\theta_k)\right]^T\beta\mu(\theta_k-\theta_0)+\gamma\beta^2\mu^2\lVert(\theta_k-\vartheta_k)+(\vartheta_k-\theta^*)+(\theta^*-\theta_0)\rVert^2\\
        &\quad -2\left[\theta_k-\vartheta_k\right]^T\beta\mu(\theta_k-\theta_0)\\
        \Delta V_k&=\frac{1}{\N_k}\left\{-2\left(1-\frac{\gamma\phi_k^T\phi_k}{\N_k}\right)\tilde{\theta}_{k+1}^T\nabla L_k(\theta_{k+1})-\frac{\gamma\beta^2}{\N_k}\lVert\nabla L_k(\theta_k)\rVert^2-2\beta(\bar{\theta}_k-\vartheta_k)^T\nabla L_k(\theta_k)\right.\\
        &\quad\left.-\frac{\beta(2-\beta)\N_k}{\gamma}\lVert \bar{\theta}_k-\vartheta_k\rVert^2+4(1-\beta)(\bar{\theta}_k-\vartheta_k)^T\nabla L_k(\theta_{k+1})\right\}\\
        &\quad +4\mu\frac{\gamma}{\N_k}\tilde{\theta}_{k+1}^T\nabla L_k(\theta_{k+1})+4\mu\frac{\gamma}{\N_k}\left[(1-\beta)(\bar{\theta}_k-\vartheta_k)+(\vartheta_k-\theta^*)\right]^T\phi_k\phi_k^T(\theta^*-\theta_0)\\
        &\quad +2\mu(1-\beta)^2\lVert\bar{\theta}_k-\vartheta_k\rVert^2-2\mu\lVert\vartheta_k-\theta^*\rVert^2\\
        &\quad +2\mu (1-\beta)(\bar{\theta}_k-\vartheta_k)^T(\theta^*-\theta_0)-2\mu(\vartheta_k-\theta^*)^T(\theta^*-\theta_0)\\
        &\quad +2\gamma\mu^2(1-\beta)^2\lVert(\bar{\theta}_k-\vartheta_k)\rVert^2+2\gamma\mu^2\lVert\vartheta_k-\theta^*\rVert^2+2\gamma\mu^2\lVert\theta^*-\theta_0\rVert^2\\
        &\quad +4\gamma\mu^2(1-\beta)(\bar{\theta}_k-\vartheta_k)^T(\vartheta_k-\theta^*)+4\gamma\mu^2(1-\beta)(\bar{\theta}_k-\vartheta_k)^T(\theta^*-\theta_0)\\
        &\quad +4\gamma\mu^2(\vartheta_k-\theta^*)^T(\theta^*-\theta_0)\\
        &\quad +2\left[\frac{\gamma\beta}{\N_k}\nabla L_k(\theta_k)\right]^T\beta\mu(\theta_k-\theta_0)\\
        &\quad +\gamma\beta^2\mu^2\lVert\theta_k-\vartheta_k\rVert^2+\gamma\beta^2\mu^2\lVert\vartheta_k-\theta^*\rVert^2+\gamma\beta^2\mu^2\lVert\theta^*-\theta_0\rVert^2\\
        &\quad +2\gamma\beta^2\mu^2(\theta_k-\vartheta_k)^T(\vartheta_k-\theta^*)+2\gamma\beta^2\mu^2(\theta_k-\vartheta_k)^T(\theta^*-\theta_0)\\
        &\quad +2\gamma\beta^2\mu^2(\vartheta_k-\theta^*)^T(\theta^*-\theta_0)\\
        &\quad -2\beta\mu(\theta_k-\vartheta_k)^T\left[(\theta_k-\vartheta_k)+(\vartheta_k-\theta^*)+(\theta^*-\theta_0)\right]-\beta\mu\lVert\vartheta_k-\theta^*\rVert^2\\
        &\quad +\beta\mu\lVert\vartheta_k-\theta^*\rVert^2\\
        &\quad  -2\frac{\gamma\beta^2\mu}{\N_k}(\theta_k-\theta_0)^T\nabla L_k(\theta_k)\\
        \Delta V_k&=\frac{1}{\N_k}\left\{-2\left(1-\frac{\gamma\phi_k^T\phi_k}{\N_k}\right)\tilde{\theta}_{k+1}^T\nabla L_k(\theta_{k+1})-\frac{\gamma\beta^2}{\N_k}\lVert\nabla L_k(\theta_k)\rVert^2-2\beta(\bar{\theta}_k-\vartheta_k)^T\phi_k\phi^T_k\tilde{\theta}_k\right.\\
        &\quad\left.-\frac{\beta(2-\beta)\N_k}{\gamma}\lVert \bar{\theta}_k-\vartheta_k\rVert^2+4(1-\beta)(\bar{\theta}_k-\vartheta_k)^T\nabla L_k(\theta_{k+1})\right\}\\
        &\quad +4\mu\frac{\gamma}{\N_k}\tilde{\theta}_{k+1}^T\nabla L_k(\theta_{k+1})\\
        &\quad -2\mu\lVert\vartheta_k-\theta^*\rVert^2-\beta\mu\lVert\vartheta_k-\theta^*\rVert^2+\beta\mu\lVert\vartheta_k-\theta^*\rVert^2+2\gamma\mu^2\lVert\vartheta_k-\theta^*\rVert^2\\
        &\quad +\gamma\beta^2\mu^2\lVert\vartheta_k-\theta^*\rVert^2\\
        &\quad -2\beta\mu\lVert\theta_k-\vartheta_k\rVert^2+\gamma\beta^2\mu^2\lVert\theta_k-\vartheta_k\rVert^2\\
        &\quad +2\mu(1-\beta)^2\lVert\bar{\theta}_k-\vartheta_k\rVert^2+2\gamma\mu^2(1-\beta)^2\lVert\bar{\theta}_k-\vartheta_k\rVert^2\\
        &\quad +2\gamma\beta^2\mu^2(\theta_k-\vartheta_k)^T(\vartheta_k-\theta^*)-2\beta\mu(\theta_k-\vartheta_k)^T(\vartheta_k-\theta^*)\\
        &\quad +4\gamma\mu^2(1-\beta)(\bar{\theta}_k-\vartheta_k)^T(\vartheta_k-\theta^*)\\
        &\quad +4\mu\frac{\gamma}{\N_k}(1-\beta)(\bar{\theta}_k-\vartheta_k)^T\phi_k\phi_k^T(\theta^*-\theta_0)+2\mu (1-\beta)(\bar{\theta}_k-\vartheta_k)^T(\theta^*-\theta_0)\\
        &\quad +4\gamma\mu^2(1-\beta)(\bar{\theta}_k-\vartheta_k)^T(\theta^*-\theta_0)\\
        &\quad +4\mu\frac{\gamma}{\N_k}(\vartheta_k-\theta^*)^T\phi_k\phi_k^T(\theta^*-\theta_0)-2\mu(\vartheta_k-\theta^*)^T(\theta^*-\theta_0)\\
        &\quad +4\gamma\mu^2(\vartheta_k-\theta^*)^T(\theta^*-\theta_0)+2\gamma\beta^2\mu^2(\vartheta_k-\theta^*)^T(\theta^*-\theta_0)\\
        &\quad +2\gamma\beta^2\mu^2(\theta_k-\vartheta_k)^T(\theta^*-\theta_0)-2\beta\mu(\theta_k-\vartheta_k)^T(\theta^*-\theta_0)\\
        &\quad +2\gamma\mu^2\lVert\theta^*-\theta_0\rVert^2+\gamma\beta^2\mu^2\lVert\theta^*-\theta_0\rVert^2\\
        \Delta V_k&=\frac{1}{\N_k}\left\{-2\left(1-\frac{\gamma\phi_k^T\phi_k}{\N_k}\right)\tilde{\theta}_{k+1}^T\nabla L_k(\theta_{k+1})-\frac{\gamma\beta^2}{\N_k}\lVert\nabla L_k(\theta_k)\rVert^2\right.\\
        &\quad\left.-\frac{\beta(2-\beta)\N_k}{\gamma}\lVert \bar{\theta}_k-\vartheta_k\rVert^2+4(1-\beta)(\bar{\theta}_k-\vartheta_k)^T\nabla L_k(\theta_{k+1})\right.\\
        &\quad\left.-2\beta(\bar{\theta}_k-\vartheta_k)^T\phi_k\phi^T_k\left[\theta_k-\theta^*+(1-\beta)\bar{\theta}_k+\beta\vartheta_k-(1-\beta)\bar{\theta}_k-\beta\vartheta_k\right]\right\}\\
        &\quad +4\mu\frac{\gamma}{\N_k}\tilde{\theta}_{k+1}^T\nabla L_k(\theta_{k+1})\\
        &\quad -\mu\left(2+\beta-\beta-2\gamma\mu-\gamma\beta^2\mu\right)\lVert\vartheta_k-\theta^*\rVert^2\\
        &\quad -\mu\left(2\beta-\gamma\beta^2\mu\right)\lVert\theta_k-\vartheta_k\rVert^2\\
        &\quad +\mu\left(2(1-\beta)^2+2\gamma\mu(1-\beta)^2\right)\lVert\bar{\theta}_k-\vartheta_k\rVert^2\\
        &\quad +\mu\left(2\gamma\beta^2\mu-2\beta\right)(\theta_k-\vartheta_k)^T(\vartheta_k-\theta^*)\\
        &\quad +\mu\left(4\gamma\mu(1-\beta)\right)(\bar{\theta}_k-\vartheta_k)^T(\vartheta_k-\theta^*)\\
        &\quad +\mu\left(4\frac{\gamma}{\N_k}(1-\beta)\right)(\bar{\theta}_k-\vartheta_k)^T\phi_k\phi_k^T(\theta^*-\theta_0)\\
        &\quad +\mu\left(2 (1-\beta)+4\gamma\mu(1-\beta)\right)(\bar{\theta}_k-\vartheta_k)^T(\theta^*-\theta_0)\\
        &\quad +\mu\left(4\frac{\gamma}{\N_k}\right)(\vartheta_k-\theta^*)^T\phi_k\phi_k^T(\theta^*-\theta_0)-\mu\left(2\right)(\vartheta_k-\theta^*)^T(\theta^*-\theta_0)\\
        &\quad +\mu\left(4\gamma\mu+2\gamma\beta^2\mu\right)(\vartheta_k-\theta^*)^T(\theta^*-\theta_0)\\
        &\quad +\mu\left(2\gamma\beta^2\mu-2\beta\right)(\theta_k-\vartheta_k)^T(\theta^*-\theta_0)\\
        &\quad +\mu\left(2\gamma\mu+\gamma\beta^2\mu\right)\lVert\theta^*-\theta_0\rVert^2\\
        \Delta V_k&=\frac{1}{\N_k}\left\{-2\left(1-\frac{\gamma\phi_k^T\phi_k}{\N_k}\right)\tilde{\theta}_{k+1}^T\nabla L_k(\theta_{k+1})-\frac{\gamma\beta^2}{\N_k}\lVert\nabla L_k(\theta_k)\rVert^2\right.\\
        &\quad\left.-2\beta(\bar{\theta}_k-\vartheta_k)^T\nabla L_k(\theta_{k+1})-\frac{\beta(2-\beta)\N_k}{\gamma}\lVert \bar{\theta}_k-\vartheta_k\rVert^2\right.\\
        &\quad\left.+4(1-\beta)(\bar{\theta}_k-\vartheta_k)^T\nabla L_k(\theta_{k+1})-2\beta(\bar{\theta}_k-\vartheta_k)^T\phi_k\phi^T_k\left[\theta_k-\bar{\theta}_k+\beta(\bar{\theta}_k-\vartheta_k)\right]\right\}\\
        &\quad +4\mu\frac{\gamma}{\N_k}\tilde{\theta}_{k+1}^T\nabla L_k(\theta_{k+1})\\
        &\quad -\mu\left(2+\beta-\beta-2\gamma\mu-\gamma\beta^2\mu\right)\lVert\vartheta_k-\theta^*\rVert^2\\
        &\quad -\mu\left(2\beta-\gamma\beta^2\mu\right)\lVert\theta_k-\vartheta_k\rVert^2\\
        &\quad +\mu\left(2(1-\beta)^2+2\gamma\mu(1-\beta)^2\right)\lVert\bar{\theta}_k-\vartheta_k\rVert^2\\
        &\quad +\mu\left(2\gamma\beta^2\mu-2\beta\right)(\theta_k-\vartheta_k)^T(\vartheta_k-\theta^*)\\
        &\quad +\mu\left(4\gamma\mu(1-\beta)\right)(\bar{\theta}_k-\vartheta_k)^T(\vartheta_k-\theta^*)\\
        &\quad +\mu\left(4\frac{\gamma}{\N_k}(1-\beta)\right)(\bar{\theta}_k-\vartheta_k)^T\phi_k\phi_k^T(\theta^*-\theta_0)\\
        &\quad +\mu\left(2 (1-\beta)+4\gamma\mu(1-\beta)\right)(\bar{\theta}_k-\vartheta_k)^T(\theta^*-\theta_0)\\
        &\quad +\mu\left(4\frac{\gamma}{\N_k}\right)(\vartheta_k-\theta^*)^T\phi_k\phi_k^T(\theta^*-\theta_0)-\mu\left(2\right)(\vartheta_k-\theta^*)^T(\theta^*-\theta_0)\\
        &\quad +\mu\left(4\gamma\mu+2\gamma\beta^2\mu\right)(\vartheta_k-\theta^*)^T(\theta^*-\theta_0)\\
        &\quad +\mu\left(2\gamma\beta^2\mu-2\beta\right)(\theta_k-\vartheta_k)^T(\theta^*-\theta_0)\\
        &\quad +\mu\left(2\gamma\mu+\gamma\beta^2\mu\right)\lVert\theta^*-\theta_0\rVert^2\\
        \Delta V_k&=\frac{1}{\N_k}\left\{-2\left(1-\frac{\gamma\phi_k^T\phi_k}{\N_k}\right)\tilde{\theta}_{k+1}^T\nabla L_k(\theta_{k+1})-\frac{\gamma\beta^2}{\N_k}\lVert\nabla L_k(\theta_k)\rVert^2\right.\\
        &\quad\left.-2\gamma\beta^2(\bar{\theta}_k-\vartheta_k)^T\frac{\phi_k\phi^T_k}{\N_k}\nabla L_k(\theta_k)-\frac{\beta(2-\beta)\N_k}{\gamma}\lVert \bar{\theta}_k-\vartheta_k\rVert^2\right.\\
        &\quad\left.+4(1-\frac{3}{2}\beta)(\bar{\theta}_k-\vartheta_k)^T\nabla L_k(\theta_{k+1})-2\beta^2(\bar{\theta}_k-\vartheta_k)^T\phi_k\phi^T_k(\bar{\theta}_k-\vartheta_k)\right\}\\
        &\quad +4\mu\frac{\gamma}{\N_k}\tilde{\theta}_{k+1}^T\nabla L_k(\theta_{k+1})\\
        &\quad -\mu\left(2+\beta-\beta-2\gamma\mu-\gamma\beta^2\mu\right)\lVert\vartheta_k-\theta^*\rVert^2\\
        &\quad -\mu\left(2\beta-\gamma\beta^2\mu\right)\lVert\theta_k-\vartheta_k\rVert^2\\
        &\quad +\mu\left(2(1-\beta)^2+2\gamma\mu(1-\beta)^2\right)\lVert\bar{\theta}_k-\vartheta_k\rVert^2\\
        &\quad +\mu\left(2\gamma\beta^2\mu-2\beta\right)(\theta_k-\vartheta_k)^T(\vartheta_k-\theta^*)\\
        &\quad +\mu\left(4\gamma\mu(1-\beta)\right)(\bar{\theta}_k-\vartheta_k)^T(\vartheta_k-\theta^*)\\
        &\quad +\mu\left(4\frac{\gamma}{\N_k}(1-\beta)\right)(\bar{\theta}_k-\vartheta_k)^T\phi_k\phi_k^T(\theta^*-\theta_0)\\
        &\quad +\mu\left(2 (1-\beta)+4\gamma\mu(1-\beta)\right)(\bar{\theta}_k-\vartheta_k)^T(\theta^*-\theta_0)\\
        &\quad +\mu\left(4\frac{\gamma}{\N_k}\right)(\vartheta_k-\theta^*)^T\phi_k\phi_k^T(\theta^*-\theta_0)-\mu\left(2\right)(\vartheta_k-\theta^*)^T(\theta^*-\theta_0)\\
        &\quad +\mu\left(4\gamma\mu+2\gamma\beta^2\mu\right)(\vartheta_k-\theta^*)^T(\theta^*-\theta_0)\\
        &\quad +\mu\left(2\gamma\beta^2\mu-2\beta\right)(\theta_k-\vartheta_k)^T(\theta^*-\theta_0)\\
        &\quad +\mu\left(2\gamma\mu+\gamma\beta^2\mu\right)\lVert\theta^*-\theta_0\rVert^2\\
        &\quad -2\frac{\gamma\beta^2\mu}{\N_k}(\bar{\theta}_k-\vartheta_k)^T\phi_k\phi_k^T\left[(\theta_k-\vartheta_k)+(\vartheta_k-\theta^*)+(\theta^*-\theta_0)\right]\\
        \Delta V_k&=\frac{1}{\N_k}\left\{-2\left(1-\frac{\gamma\phi_k^T\phi_k}{\N_k}\right)\tilde{\theta}_{k+1}^T\nabla L_k(\theta_{k+1})-\frac{\gamma\beta^2}{\N_k}\lVert\nabla L_k(\theta_k)\rVert^2\right.\\
        &\quad\left.-2\gamma\beta^2(\bar{\theta}_k-\vartheta_k)^T\frac{\phi_k\phi^T_k}{\N_k}\nabla L_k(\theta_k)-\frac{\beta(2-\beta)\N_k}{\gamma}\lVert \bar{\theta}_k-\vartheta_k\rVert^2\right.\\
        &\quad\left.+4(1-\frac{3}{2}\beta)(\bar{\theta}_k-\vartheta_k)^T\nabla L_k(\theta_{k+1})-2\beta^2(\bar{\theta}_k-\vartheta_k)^T\phi_k\phi^T_k(\bar{\theta}_k-\vartheta_k)\right\}\\
        &\quad +4\mu\frac{\gamma}{\N_k}\tilde{\theta}_{k+1}^T\nabla L_k(\theta_{k+1})\\
        &\quad -\mu\left(2+\beta-\beta-2\gamma\mu-\gamma\beta^2\mu\right)\lVert\vartheta_k-\theta^*\rVert^2\\
        &\quad -\mu\left(2\beta-\gamma\beta^2\mu\right)\lVert\theta_k-\vartheta_k\rVert^2\\
        &\quad +\mu\left(2(1-\beta)^2+2\gamma\mu(1-\beta)^2\right)\lVert\bar{\theta}_k-\vartheta_k\rVert^2\\
        &\quad +\mu\left(2\gamma\beta^2\mu-2\beta\right)(\theta_k-\vartheta_k)^T(\vartheta_k-\theta^*)\\
        &\quad +\mu\left(4\gamma\mu(1-\beta)\right)(\bar{\theta}_k-\vartheta_k)^T(\vartheta_k-\theta^*)-\mu\left(2\frac{\gamma\beta^2}{\N_k}\right)(\bar{\theta}_k-\vartheta_k)^T\phi_k\phi_k^T(\vartheta_k-\theta^*)\\
        &\quad -\mu\left(2\frac{\gamma\beta^2}{\N_k}\right)(\bar{\theta}_k-\vartheta_k)^T\phi_k\phi_k^T(\theta_k-\vartheta_k)\\
        &\quad +\mu\left(4\frac{\gamma}{\N_k}(1-\beta)\right)(\bar{\theta}_k-\vartheta_k)^T\phi_k\phi_k^T(\theta^*-\theta_0)\\
        &\quad +\mu\left(2 (1-\beta)+4\gamma\mu(1-\beta)\right)(\bar{\theta}_k-\vartheta_k)^T(\theta^*-\theta_0)\\
        &\quad +\mu\left(4\frac{\gamma}{\N_k}\right)(\vartheta_k-\theta^*)^T\phi_k\phi_k^T(\theta^*-\theta_0)-\mu\left(2\right)(\vartheta_k-\theta^*)^T(\theta^*-\theta_0)\\
        &\quad +\mu\left(4\gamma\mu+2\gamma\beta^2\mu\right)(\vartheta_k-\theta^*)^T(\theta^*-\theta_0)\\
        &\quad +\mu\left(2\gamma\beta^2\mu-2\beta\right)(\theta_k-\vartheta_k)^T(\theta^*-\theta_0)\\
        &\quad -\mu\left(2\frac{\gamma\beta^2}{\N_k}\right)(\bar{\theta}_k-\vartheta_k)^T\phi_k\phi_k^T(\theta^*-\theta_0)\\
        &\quad +\mu\left(2\gamma\mu+\gamma\beta^2\mu\right)\lVert\theta^*-\theta_0\rVert^2\\
        \Delta V_k&\leq\frac{1}{\N_k}\left\{-2\left(1-\frac{\gamma\phi_k^T\phi_k}{\N_k}\right)\tilde{\theta}_{k+1}^T\nabla L_k(\theta_{k+1})-16\lVert\phi_k\rVert^2\lVert \bar{\theta}_k-\vartheta_k\rVert^2\right.\\
        &\quad\left.+4(1-\frac{3}{2}\beta)(\bar{\theta}_k-\vartheta_k)^T\nabla L_k(\theta_{k+1})\right.\\
        &\quad\left.-\frac{\gamma\beta^2}{\N_k}\lVert\nabla L_k(\theta_k)\rVert^2-\beta^2\lVert\phi_k\rVert^2\lVert \bar{\theta}_k-\vartheta_k\rVert^2-2\beta^2(\bar{\theta}_k-\vartheta_k)^T\gamma\frac{\phi_k\phi^T_k}{\N_k}\nabla L_k(\theta_k)\right.\\
        &\quad\left.-(16+\beta^2)\lVert \bar{\theta}_k-\vartheta_k\rVert^2-2\beta^2(\bar{\theta}_k-\vartheta_k)^T\phi_k\phi^T_k(\bar{\theta}_k-\vartheta_k)\right\}\\
        &\quad +4\mu\frac{\gamma}{\N_k}\lVert\tilde{\theta}_{k+1}^T\phi_k\rVert^2\\
        &\quad -\mu\left(2+\beta-\beta-2\gamma\mu-\gamma\beta^2\mu\right)\lVert\vartheta_k-\theta^*\rVert^2\\
        &\quad -\mu\left(2\beta-\gamma\beta^2\mu\right)\lVert\theta_k-\vartheta_k\rVert^2\\
        &\quad +\mu\left(2(1-\beta)^2+2\gamma\mu(1-\beta)^2\right)\lVert\bar{\theta}_k-\vartheta_k\rVert^2\\
        &\quad +\mu\left|2\gamma\beta^2\mu-2\beta\right|\lVert\theta_k-\vartheta_k\rVert\lVert\vartheta_k-\theta^*\rVert\\
        &\quad +\mu\left(\left|4\gamma\mu(1-\beta)\right|+2\gamma\beta^2\right)\lVert\bar{\theta}_k-\vartheta_k\rVert\lVert\vartheta_k-\theta^*\rVert\\
        &\quad +\mu\left(2\gamma\beta^2\right)\lVert\bar{\theta}_k-\vartheta_k\rVert\lVert\theta_k-\vartheta_k\rVert\\
        &\quad +\mu\left(4\gamma\left|(1-\beta)\right|+2\left|(1-\beta)\right|+4\gamma\mu\left|(1-\beta)\right|\right)\lVert\bar{\theta}_k-\vartheta_k\rVert\lVert\theta^*-\theta_0\rVert\\
        &\quad +\mu\left(4\gamma+2\right)\lVert\vartheta_k-\theta^*\rVert\lVert\theta^*-\theta_0\rVert\\
        &\quad +\mu\left(4\gamma\mu+2\gamma\beta^2\mu\right)\lVert\vartheta_k-\theta^*\rVert\lVert\theta^*-\theta_0\rVert\\
        &\quad +\mu\left|2\gamma\beta^2\mu-2\beta\right|\lVert\theta_k-\vartheta_k\rVert\lVert\theta^*-\theta_0\rVert\\
        &\quad +\mu\left(2\gamma\beta^2\right)\lVert\bar{\theta}_k-\vartheta_k\rVert\lVert\theta^*-\theta_0\rVert\\
        &\quad +\mu\left(2\gamma\mu+\gamma\beta^2\mu\right)\lVert\theta^*-\theta_0\rVert^2\\
        &\quad -\mu\left(\frac{17}{8}+\frac{7}{8}+\frac{9}{16}+\frac{1}{16\beta}\right)\lVert\bar{\theta}_k-\vartheta_k\rVert^2\\
        \Delta V_k&\leq\frac{1}{\N_k}\left\{-\frac{30}{16}\lVert\tilde{\theta}_{k+1}^T\phi_k\rVert^2-16\lVert\phi_k\rVert^2\lVert \bar{\theta}_k-\vartheta_k\rVert^2+8\lVert\bar{\theta}_k-\vartheta_k\rVert\lVert\phi_k\rVert\lVert\tilde{\theta}_{k+1}^T\phi_k\rVert\right.\\
        &\quad\left.-\frac{\gamma\beta^2}{\N_k}\lVert\nabla L_k(\theta_k)\rVert^2-\beta^2\lVert\phi_k\rVert^2\lVert \bar{\theta}_k-\vartheta_k\rVert^2+2\beta^2\lVert\bar{\theta}_k-\vartheta_k\rVert\frac{\lVert\sqrt{\gamma}\phi_k\rVert^2}{\N_k}\lVert\nabla L_k(\theta_k)\rVert\right.\\
        &\quad\left.-(16+\beta^2)\lVert \bar{\theta}_k-\vartheta_k\rVert^2-2\beta^2(\bar{\theta}_k-\vartheta_k)^T\phi_k\phi^T_k(\bar{\theta}_k-\vartheta_k)\right\}\\
        &\quad +\frac{1}{4}\frac{1}{\N_k}\lVert\tilde{\theta}_{k+1}^T\phi_k\rVert^2\\
        &\quad -\mu\left(\frac{13}{16}+\beta\right)\lVert\vartheta_k-\theta^*\rVert^2\\
        &\quad -\mu\left(\beta\left(\frac{31}{16}\right)\right)\lVert\theta_k-\vartheta_k\rVert^2\\
        &\quad +\mu\left(\frac{17}{8}\right)\lVert\bar{\theta}_k-\vartheta_k\rVert^2\\
        &\quad +\mu\left(2\beta\right)\lVert\theta_k-\vartheta_k\rVert\lVert\vartheta_k-\theta^*\rVert\\
        &\quad +\mu\left(\frac{3}{8}\right)\lVert\bar{\theta}_k-\vartheta_k\rVert\lVert\vartheta_k-\theta^*\rVert\\
        &\quad +\mu\left(\frac{1}{8}\right)\lVert\bar{\theta}_k-\vartheta_k\rVert\lVert\theta_k-\vartheta_k\rVert\\
        &\quad +\mu\left(\frac{21}{8}\right)\lVert\vartheta_k-\theta^*\rVert\lVert\theta^*-\theta_0\rVert\\
        &\quad +\mu\left(2\right)\lVert\theta_k-\vartheta_k\rVert\lVert\theta^*-\theta_0\rVert\\
        &\quad +\mu\left(\frac{21}{8}\right)\lVert\bar{\theta}_k-\vartheta_k\rVert\lVert\theta^*-\theta_0\rVert\\
        &\quad +\mu\left(\frac{3}{16}\right)\lVert\theta^*-\theta_0\rVert^2\\
        &\quad -\mu\left(\frac{17}{8}+\frac{7}{8}+\frac{9}{16}+\frac{1}{16\beta}\right)\lVert\bar{\theta}_k-\vartheta_k\rVert^2\\
        \Delta V_k&\leq\frac{1}{\N_k}\left\{-\lVert\tilde{\theta}_{k+1}^T\phi_k\rVert^2-16\lVert\phi_k\rVert^2\lVert \bar{\theta}_k-\vartheta_k\rVert^2+8\lVert\bar{\theta}_k-\vartheta_k\rVert\lVert\phi_k\rVert\lVert\tilde{\theta}_{k+1}^T\phi_k\rVert\right.\\
        &\quad\left.-\frac{\gamma\beta^2}{\N_k}\lVert\nabla L_k(\theta_k)\rVert^2-\beta^2\lVert\phi_k\rVert^2\lVert \bar{\theta}_k-\vartheta_k\rVert^2+\frac{2\sqrt{\gamma}\beta^2}{\sqrt{\N_k}}\lVert\bar{\theta}_k-\vartheta_k\rVert\lVert\phi_k\rVert\lVert\nabla L_k(\theta_k)\rVert\right.\\
        &\quad\left.-\frac{7}{8}\lVert\tilde{\theta}_{k+1}^T\phi_k\rVert^2-(16+\beta^2)\lVert \bar{\theta}_k-\vartheta_k\rVert^2-2\beta^2(\bar{\theta}_k-\vartheta_k)^T\phi_k\phi^T_k(\bar{\theta}_k-\vartheta_k)\right\}\\
        &\quad +\frac{1}{4}\frac{1}{\N_k}\lVert\tilde{\theta}_{k+1}^T\phi_k\rVert^2\\
        &\quad -\mu\left(\frac{10}{16}\beta+\frac{2}{16}+\frac{1}{16}+\beta\right)\lVert\vartheta_k-\theta^*\rVert^2+\mu\left(\frac{21}{8}\right)\lVert\vartheta_k-\theta^*\rVert\lVert\theta^*-\theta_0\rVert\\
        &\quad \pm\mu\left(\frac{441}{32}\right)\lVert\theta^*-\theta_0\rVert^2\\
        &\quad -\mu\left(\beta\left(\frac{10}{16}+\frac{4}{16}+\frac{1}{16}+1\right)\right)\lVert\theta_k-\vartheta_k\rVert^2+\mu\left(2\right)\lVert\theta_k-\vartheta_k\rVert\lVert\theta^*-\theta_0\rVert\\
        &\quad \pm\mu\left(\frac{4}{\beta}\right)\lVert\theta^*-\theta_0\rVert^2\\
        &\quad -\mu\left(\frac{17}{8}+\frac{7}{8}+\frac{9}{16}+\frac{1}{16\beta}-\frac{17}{8}\right)\lVert\bar{\theta}_k-\vartheta_k\rVert^2+\mu\left(\frac{21}{8}\right)\lVert\bar{\theta}_k-\vartheta_k\rVert\lVert\theta^*-\theta_0\rVert\\
        &\quad \pm\mu\left(\frac{441}{224}\right)\lVert\theta^*-\theta_0\rVert^2\\
        &\quad +\mu\left(2\beta\right)\lVert\theta_k-\vartheta_k\rVert\lVert\vartheta_k-\theta^*\rVert\\
        &\quad +\mu\left(\frac{3}{8}\right)\lVert\bar{\theta}_k-\vartheta_k\rVert\lVert\vartheta_k-\theta^*\rVert\\
        &\quad +\mu\left(\frac{1}{8}\right)\lVert\bar{\theta}_k-\vartheta_k\rVert\lVert\theta_k-\vartheta_k\rVert\\
        &\quad +\mu\left(\frac{3}{16}\right)\lVert\theta^*-\theta_0\rVert^2\\
        \Delta V_k&\leq\frac{1}{\N_k}\left\{-\left[\lVert\tilde{\theta}_{k+1}^T\phi_k\rVert-4\lVert\phi_k\rVert\lVert \bar{\theta}_k-\vartheta_k\rVert\right]^2\right.\\
        &\quad\left.-\left[\frac{\sqrt{\gamma}\beta}{\sqrt{\N_k}}\lVert\nabla L_k(\theta_k)\rVert-\beta\lVert\phi_k\rVert\lVert \bar{\theta}_k-\vartheta_k\rVert\right]^2\right.\\
        &\quad\left.-\frac{7}{8}\lVert\tilde{\theta}_{k+1}^T\phi_k\rVert^2-(16+\beta^2)\lVert \bar{\theta}_k-\vartheta_k\rVert^2-2\beta^2\lVert(\bar{\theta}_k-\vartheta_k)^T\phi_k\rVert^2\right\}\\
        &\quad +\frac{1}{4}\frac{1}{\N_k}\lVert\tilde{\theta}_{k+1}^T\phi_k\rVert^2\\
        &\quad -\mu\gamma\beta\frac{10}{16}V_k+\mu\left(\frac{3570\beta+896}{224\beta}\right)\lVert\theta^*-\theta_0\rVert^2\\
        &\quad -\mu\left[\frac{\sqrt{2}}{4}\lVert\vartheta_k-\theta^*\rVert-\frac{21}{4\sqrt{2}}\lVert\theta^*-\theta_0\rVert\right]^2\\
        &\quad -\mu\left[\frac{\sqrt{\beta}}{2}\lVert\theta_k-\vartheta_k\rVert-\frac{2}{\sqrt{\beta}}\lVert\theta^*-\theta_0\rVert\right]^2\\
        &\quad -\mu\left[\frac{\sqrt{14}}{4}\lVert\bar{\theta}_k-\vartheta_k\rVert-\frac{21}{4\sqrt{14}}\lVert\theta^*-\theta_0\rVert\right]^2\\
        &\quad -\mu\left[\sqrt{\beta}\lVert\vartheta_k-\theta^*\rVert-\sqrt{\beta}\lVert\theta_k-\vartheta_k\rVert\right]^2\\
        &\quad -\mu\left[\frac{1}{4}\lVert\vartheta_k-\theta^*\rVert-\frac{3}{4}\lVert\bar{\theta}_k-\vartheta_k\rVert\right]^2\\
        &\quad -\mu\left[\frac{\sqrt{\beta}}{4}\lVert\theta_k-\vartheta_k\rVert-\frac{1}{4\sqrt{\beta}}\lVert\bar{\theta}_k-\vartheta_k\rVert\right]^2\\
        \Delta V_k&\leq-\frac{L_k(\theta_{k+1})}{\N_k}-\mu\underbrace{\gamma\beta\frac{10}{16}}_{c_1}V_k+\mu\underbrace{\left(\frac{3570\beta+896}{224\beta}\right)\lVert\theta^*-\theta_0\rVert^2}_{c_2}.
    \end{align*}
    \endgroup
    From the bound on $\Delta V_k$, it can be noted that $\Delta V_k<0$ in $D^c$, where the compact set $D$ is defined as
    \begin{equation*}
        D=\left\{V\middle|V\leq\frac{c_2}{c_1}\right\}.
    \end{equation*}
    Therefore $V\in\ell_{\infty}$, $(\vartheta-\theta^*)\in\ell_{\infty}$, and $(\theta-\vartheta)\in\ell_{\infty}$. Furthermore, from the bound on $\Delta V_k$,
    \begin{align*}
        V_{k+1}&\leq(1-\mu c_1)V_k+\mu c_2\\
        V_{k+1}&\leq(1-\mu c_1)\left(V_k-\frac{c_2}{c_1}\right)+(1-\mu c_1)\frac{c_2}{c_1}+\mu c_2\\
        V_{k+1}&\leq(1-\mu c_1)\left(V_k-\frac{c_2}{c_1}\right)+\frac{c_2}{c_1}\\
        V_{k+1}-\frac{c_2}{c_1}&\leq(1-\mu c_1)\left(V_k-\frac{c_2}{c_1}\right)
    \end{align*}
    Collecting terms,
    \begin{align*}
        V_k-\frac{c_2}{c_1}&\leq(1-\mu c_1)^k\left(V_0-\frac{c_2}{c_1}\right)\\
        V_k-\frac{c_2}{c_1}&\leq\exp(-\mu c_1k)\left(V_0-\frac{c_2}{c_1}\right)\\
        V_k&\leq\exp(-\mu c_1k)\left(V_0-\frac{c_2}{c_1}\right)+\frac{c_2}{c_1}.
    \end{align*}
\end{proof}

\clearpage
\section{Non-asymptotic convergence rate proofs}
\label{s:NonAsymptoticProofs}

The theorems provided in section are widely known in the iterative optimization literature (theorem statements and proofs are modified from provided references). The full statement and proof of each theorem is provided for completeness of the comparison in Table \ref{t:Comparison_Gradient_Methods}, and to set the stage for the core lemmas of this paper in Appendix \ref{ss:Regularization_New} and comparisons in Appendix \ref{ss:Comparison_with_Constants}.

\subsection{Gradient descent for smooth convex functions}
\begin{theorem}[Modified from \citep{Nesterov_2018}]\label{theorem:GD_k}
    For a $\bar{L}$-smooth convex function $f$, the iterates $\{\theta_k\}_{k=0}^{\infty}$ generated by \eqref{e:Gradient_Method} with $\bar{\alpha}=1/\bar{L}$ satisfy
    \begin{equation}\label{e:GD_SmoothConvex_convergence}
        f(\theta_k)-f(\theta^*)\leq\frac{2\bar{L}\lVert\theta_0-\theta^*\rVert^2}{k+4},
    \end{equation}
    and therefore if
    \begin{equation}
        k\geq\left\lceil\frac{2\bar{L}\lVert\theta_0-\theta^*\rVert^2}{\epsilon}-4\right\rceil,
    \end{equation}
    then $f(\theta_k)-f(\theta^*)\leq\epsilon$.
\end{theorem}
\begin{proof}[Modified from \citep{Nesterov_2018}]
    From $\bar{L}$-smoothness \eqref{e:smooth_convex},
    \begin{equation*}
        f(\theta_{k+1})\leq f(\theta_k)+\nabla f(\theta_k)^T(\theta_{k+1}-\theta_k)+\frac{\bar{L}}{2}\lVert \theta_{k+1}-\theta_k\rVert^2.
    \end{equation*}
    Applying the iterative method in \eqref{e:Gradient_Method} with $\bar{\alpha}=1/\bar{L}$,
    \begin{equation}\label{e:GD_Primal_Progress_first}
        f(\theta_{k+1})\leq f(\theta_k)-\frac{1}{2\bar{L}}\lVert\nabla f(\theta_k)\rVert^2,
    \end{equation}
    and thus the primal progress may be bounded as,
    \begin{equation}\label{e:GD_Primal_Progress}
        (f(\theta_k)-f^*)-(f(\theta_{k+1})-f^*)=f(\theta_k)-f(\theta_{k+1})\geq\frac{1}{2\bar{L}}\lVert\nabla f(\theta_k)\rVert^2.
    \end{equation}
    Using \eqref{e:GD_Primal_Progress_first} and from convexity \eqref{e:convex}, $f(\theta)\leq f(\theta^*)+\nabla f(\theta)^T(\theta-\theta^*)$,
    \begin{equation*}
        f(\theta_{k+1})-f(\theta^*)\leq\frac{\bar{L}}{2}\left(\hspace{-0.075cm}\frac{2}{\bar{L}}\nabla f(\theta_k)^T(\theta_k-\theta^*)-\frac{1}{\bar{L}^2}\lVert\nabla f(\theta_k)\rVert^2-\lVert\theta_k-\theta^*\rVert^2+\lVert\theta_k-\theta^*\rVert^2\hspace{-0.075cm}\right).
    \end{equation*}
    Grouping terms with the iterative method in \eqref{e:Gradient_Method} with $\bar{\alpha}=1/\bar{L}$,
    \begin{equation*}
        0\leq f(\theta_{k+1})-f(\theta^*)\leq\frac{\bar{L}}{2}\left(\lVert\theta_k-\theta^*\rVert^2-\lVert\theta_{k+1}-\theta^*\rVert^2\right),
    \end{equation*}
    from which it can be seen that $\lVert\theta_{k+1}-\theta^*\rVert^2\leq\lVert\theta_k-\theta^*\rVert^2$. Thus from convexity \eqref{e:convex}, the dual bound may be expressed as,
    \begin{equation}\label{e:GD_Dual_Bound}
        f(\theta_k)-f(\theta^*)\leq\nabla f(\theta_k)^T(\theta_k-\theta^*)\leq\lVert\nabla f(\theta_k)\rVert\lVert\theta_k-\theta^*\rVert\leq\lVert\nabla f(\theta_k)\rVert\lVert\theta_0-\theta^*\rVert.
    \end{equation}
    Combining the primal progress in \eqref{e:GD_Primal_Progress} and the dual bound in \eqref{e:GD_Dual_Bound},
    \begin{equation*}
        (f(\theta_k)-f(\theta^*))-(f(\theta_{k+1})-f(\theta^*))\geq\frac{(f(\theta_k)-f(\theta^*))^2}{2\bar{L}\lVert\theta_0-\theta^*\rVert^2}.
    \end{equation*}
    Thus,
    \begin{align*}
        \frac{1}{f(\theta_{k+1})-f(\theta^*)}-\frac{1}{f(\theta_k)-f(\theta^*)}&=\frac{(f(\theta_k)-f(\theta^*))-(f(\theta_{k+1})-f(\theta^*))}{(f(\theta_k)-f(\theta^*))(f(\theta_{k+1})-f(\theta^*))}\\
        &\geq\frac{(f(\theta_k)-f(\theta^*))}{2\bar{L}\lVert\theta_0-\theta^*\rVert^2(f(\theta_{k+1})-f(\theta^*))}\\
        &\geq\frac{1}{2\bar{L}\lVert\theta_0-\theta^*\rVert^2}.
    \end{align*}
    Collecting terms and using $f(\theta_0)-f(\theta^*)\leq\bar{L}\lVert\theta_0-\theta^*\rVert^2/2$ (from $\bar{L}$-smoothness \eqref{e:smooth_convex}),
    \begin{equation*}
        \frac{1}{f(\theta_k)-f(\theta^*)}\geq\frac{1}{f(\theta_0)-f(\theta^*)}+\frac{k}{2\bar{L}\lVert\theta_0-\theta^*\rVert^2}\geq \frac{k+4}{2\bar{L}\lVert\theta_0-\theta^*\rVert^2}.
    \end{equation*}
    Bounding the inverse of the right hand side equality by $\epsilon$ completes the proof.
\end{proof}

\subsection{Gradient descent for smooth strongly convex functions}

\begin{theorem}[Modified from \citep{Bubeck_2015}]\label{theorem:GD_Exponential}
    For a $\bar{L}$-smooth and $\mu$-strongly convex function $f$, the iterates $\{\theta_k\}_{k=0}^{\infty}$ generated by \eqref{e:Gradient_Method} with $\bar{\alpha}=1/\bar{L}$ satisfy
    \begin{equation}
        f(\theta_k)-f(\theta^*)\leq(f(\theta_0)-f(\theta^*))\exp\left(-\frac{k}{\kappa}\right),
    \end{equation}
    where $\kappa=\bar{L}/\mu$, and therefore if
    \begin{equation}
        k\geq\left\lceil\kappa\log\left(\frac{f(\theta_0)-f(\theta^*)}{\epsilon}\right)\right\rceil,
    \end{equation}
    then $f(\theta_k)-f(\theta^*)\leq\epsilon$.
\end{theorem}
\begin{proof}[Modified from \citep{Bubeck_2015}]
    From $\bar{L}$-smoothness \eqref{e:smooth_convex},
    \begin{equation*}
        f(\theta_{k+1})\leq f(\theta_k)+\nabla f(\theta_k)^T(\theta_{k+1}-\theta_k)+\frac{\bar{L}}{2}\lVert \theta_{k+1}-\theta_k\rVert^2.
    \end{equation*}
    Applying the iterative method in \eqref{e:Gradient_Method} with $\bar{\alpha}=1/\bar{L}$,
    \begin{equation*}
        f(\theta_{k+1})\leq f(\theta_k)-\frac{1}{2\bar{L}}\lVert\nabla f(\theta_k)\rVert^2,
    \end{equation*}
    and thus the primal progress may be bounded as,
    \begin{equation}\label{e:GD_Primal_Progress_SC}
        (f(\theta_k)-f(\theta^*))-(f(\theta_{k+1})-f(\theta^*))=f(\theta_k)-f(\theta_{k+1})\geq\frac{1}{2\bar{L}}\lVert\nabla f(\theta_k)\rVert^2.
    \end{equation}
    From $\mu$-strong convexity \eqref{e:strongly_convex},
    \begin{equation*}
        f(y)\geq f(x)+\nabla f(x)^T(y-x)+\frac{\mu}{2}\lVert y-x\rVert^2:=\underbar{L}(y).
    \end{equation*}
    Setting the gradient of the lower bounding quadratic $\underbar{L}(y)$ with respect to $y$ equal to zero,
    \begin{equation*}
        0=\nabla f(x)+\mu(y-x)\quad\Rightarrow\quad y=x-\frac{1}{\mu}\nabla f(x).
    \end{equation*}
    Choosing this value of $y$,
    \begin{equation*}
        f(y)\geq f(x)-\frac{1}{2\mu}\lVert\nabla f(x)\rVert^2.
    \end{equation*}
    Setting $x=\theta_k$ and $y=\theta^*$ results in the dual bound,
    \begin{equation}\label{e:GD_Dual_Bound_SC}
        f(\theta_k)-f(\theta^*)\leq\frac{1}{2\mu}\lVert\nabla f(\theta_k)\rVert^2.
    \end{equation}
    Combining the primal progress in \eqref{e:GD_Primal_Progress_SC} and the dual bound in \eqref{e:GD_Dual_Bound_SC},
    \begin{equation*}
        (f(\theta_k)-f(\theta^*))-(f(\theta_{k+1})-f(\theta^*))\geq\frac{\mu}{\bar{L}}(f(\theta_k)-f(\theta^*)).
    \end{equation*}
    Thus,
    \begin{equation*}
        f(\theta_{k+1})-f(\theta^*)\leq\left(1-\frac{\mu}{\bar{L}}\right)(f(\theta_k)-f(\theta^*)).
    \end{equation*}
    Collecting terms,
    \begin{align*}
        f(\theta_k)-f(\theta^*)&\leq\left(1-\frac{\mu}{\bar{L}}\right)^k(f(\theta_0)-f(\theta^*))\\
        &\leq(f(\theta_0)-f(\theta^*))\exp\left(-\frac{k}{\kappa}\right).
    \end{align*}
    Bounding the right hand side by $\epsilon$ completes the proof.
\end{proof}

\subsection{Heavy Ball method for symmetric positive definite quadratic functions}

\begin{theorem}[Modified from \citep{Lessard_2016,Recht_HB}]\label{theorem:Heavy_BALL}
    For a $\bar{L}$-smooth and $\mu$-strongly convex quadratic function $f(\theta)=\frac{1}{2}\theta^TA\theta-b^T\theta+c$, where $A\in\mathbb{R}^{N\times N}$ is symmetric positive definite, $b\in\mathbb{R}^N$, $c\in\mathbb{R}$, the iterates $\{\theta_k\}_{k=0}^{\infty}$ generated by \eqref{e:Heavy_Ball} with $\kappa=\bar{L}/\mu$, $\bar{\alpha}=4/\left(\sqrt{\bar{L}}+\sqrt{\mu}\right)^2$, and $\bar{\beta}=\left(\max\left\{\lvert1-\sqrt{\bar{\alpha}\bar{L}}\rvert,\lvert1-\sqrt{\bar{\alpha}\mu}\rvert\right\}\right)^2$ satisfy
    \begin{equation*}
        \left\lVert\begin{bmatrix}
            \tilde{\theta}_{k+1}\\
            \tilde{\theta}_{k}
        \end{bmatrix}\right\rVert\leq\left(\frac{\sqrt{\kappa}-1}{\sqrt{\kappa}+1}+\varepsilon_k\right)^k\left\lVert\begin{bmatrix}
            \tilde{\theta}_1\\
            \tilde{\theta}_0
        \end{bmatrix}\right\rVert,
    \end{equation*}
    where $\varepsilon_k\geq0$ is a sequence such that $\lVert T^k\rVert\leq(\rho(T)+\varepsilon_k)^k$, $\lim_{k\rightarrow\infty}\varepsilon_k=0$, $\rho(t)=\max|eig(T)|$, and where
    \begin{equation*}
        T=\begin{bmatrix}
            \left(1+\bar{\beta}\right)I-\bar{\alpha}A & -\bar{\beta}I\\
            I & 0
        \end{bmatrix}.
    \end{equation*}
\end{theorem}
\begin{proof}[Modified from \citep{Lessard_2016,Recht_HB}]
    Consider the following extended vector using \eqref{e:Heavy_Ball}
    \begin{equation*}
        \begin{bmatrix}
            \tilde{\theta}_{k+1}\\
            \tilde{\theta}_k
        \end{bmatrix}
        =
        \begin{bmatrix}
            \left(1+\bar{\beta}\right)\theta_k - \bar{\beta}\theta_{k-1}-\bar{\alpha}\nabla f(\theta_k)-\theta^*\\
            \tilde{\theta}_k
        \end{bmatrix}.
    \end{equation*}
    For $\bar{\theta}_k\in[\theta_k,\theta^*]$ if $\theta_k\leq\theta^*$ or $\bar{\theta}_k\in[\theta^*,\theta_k]$ if $\theta^*<\theta_k$, by the Mean Value Theorem,
    \begin{align*}
        \begin{bmatrix}
            \tilde{\theta}_{k+1}\\
            \tilde{\theta}_k
        \end{bmatrix}
        &=
        \begin{bmatrix}
            \left(1+\bar{\beta}\right)\theta_k - \bar{\beta}\theta_{k-1}-\bar{\alpha}\nabla^2 f(\bar{\theta}_k)(\theta_k-\theta^*)-\theta^*\\
            \tilde{\theta}_k
        \end{bmatrix}\\
        &=\begin{bmatrix}
            \left(1+\bar{\beta}\right)I-\bar{\alpha}\nabla^2 f(\bar{\theta}_k) & -\bar{\beta}I\\
            I & 0
        \end{bmatrix}
        \begin{bmatrix}
            \tilde{\theta}_k\\
            \tilde{\theta}_{k-1}
        \end{bmatrix}\\
        &=\underbrace{\begin{bmatrix}
            \left(1+\bar{\beta}\right)I-\bar{\alpha}A & -\bar{\beta}I\\
            I & 0
        \end{bmatrix}}_{T}
        \begin{bmatrix}
            \tilde{\theta}_k\\
            \tilde{\theta}_{k-1}
        \end{bmatrix}.
    \end{align*}
    Therefore, for $\rho(T)\leq\frac{\sqrt{\kappa}-1}{\sqrt{\kappa}+1}$, collecting terms,
    \begin{equation*}
        \left\lVert\begin{bmatrix}
            \tilde{\theta}_{k+1}\\
            \tilde{\theta}_{k}
        \end{bmatrix}\right\rVert\leq\left(\frac{\sqrt{\kappa}-1}{\sqrt{\kappa}+1}+\varepsilon_k\right)^k\left\lVert\begin{bmatrix}
            \tilde{\theta}_1\\
            \tilde{\theta}_0
        \end{bmatrix}\right\rVert,
    \end{equation*}
    where $\varepsilon_k\geq0$ is a sequence such that $\lVert T^k\rVert\leq(\rho(T)+\varepsilon_k)^k$ and $\lim_{k\rightarrow\infty}\varepsilon_k=0$ and $\rho(t)=\max|eig(T)|$ (see \citep{Lessard_2016,Recht_HB} for a discussion of $\varepsilon_k$).
    
    We now proceed to show that indeed $\rho(T)\leq\frac{\sqrt{\kappa}-1}{\sqrt{\kappa}+1}$. Given that $A$ is a real symmetric positive definite matrix, it has an eigendecomposition $A=Q\Lambda Q^T$, where $Q$ is an orthogonal matrix and $\Lambda=\text{diag}(\lambda_1,\ldots,\lambda_N)$, where $\lambda_i$ are the eigenvalues of $A$.
    Therefore,
    \begin{equation*}
        T=
        \begin{bmatrix}
            Q & 0\\
            0 & Q
        \end{bmatrix}
        \begin{bmatrix}
            \left(1+\bar{\beta}\right)I-\bar{\alpha}\Lambda & -\bar{\beta}I\\
            I & 0
        \end{bmatrix}
        \begin{bmatrix}
            Q & 0\\
            0 & Q
        \end{bmatrix}^T.
    \end{equation*}
    By a similarity transformation, the eigenvalues of $T$ are the same as the eigenvalues of
    \begin{equation*}
        T_i=
        \begin{bmatrix}
            1+\bar{\beta}-\bar{\alpha}\lambda_i& -\bar{\beta}\\
            1 & 0
        \end{bmatrix}, \quad i\in\{1,2,\ldots,N\}.
    \end{equation*}
    Thus for each $i\in\{1,2,\ldots,N\}$, the eigenvalues of each $T_i$ may be calculated from the roots of $s^2-\left(1+\bar{\beta}-\bar{\alpha}\lambda_i\right)s+\bar{\beta}=0$. If $\left(1+\bar{\beta}-\bar{\alpha}\lambda_i\right)^2\leq 4\bar{\beta}$, then the magnitude of the roots may be bounded from above by $\sqrt{\bar{\beta}}$. The condition $\left(1+\bar{\beta}-\bar{\alpha}\lambda_i\right)^2\leq 4\bar{\beta}$ is satisfied if
    \begin{equation*}
        \bar{\beta}\in\left[(1-\sqrt{\bar{\alpha}\lambda_i})^2,(1+\sqrt{\bar{\alpha}\lambda_i})^2\right],
    \end{equation*}
    which holds for the chosen $\bar{\beta}=\left(\max\left\{\lvert1-\sqrt{\bar{\alpha}\bar{L}}\rvert,\lvert1-\sqrt{\bar{\alpha}\mu}\rvert\right\}\right)^2$. Therefore, $\rho(T)\leq\sqrt{\bar{\beta}}$.
    
    Using $\bar{\alpha}=4/\left(\sqrt{\bar{L}}+\sqrt{\mu}\right)^2$ and $\bar{\beta}=\left(\max\left\{\lvert1-\sqrt{\bar{\alpha}\bar{L}}\rvert,\lvert1-\sqrt{\bar{\alpha}\mu}\rvert\right\}\right)^2$, it can be seen that
    \begin{equation*}
        \bar{\beta}=\max\left\{\left(1-\frac{2\sqrt{\bar{L}}}{\sqrt{\bar{L}}+\sqrt{\mu}}\right)^2,\left(1-\frac{2\sqrt{\mu}}{\sqrt{\bar{L}}+\sqrt{\mu}}\right)^2\right\}=\left(\frac{\sqrt{\bar{L}}-\sqrt{\mu}}{\sqrt{\bar{L}}+\sqrt{\mu}}\right)^2=\left(\frac{\sqrt{\kappa}-1}{\sqrt{\kappa}+1}\right)^2,
    \end{equation*}
    and therefore,
    \begin{equation*}
        \rho(T)\leq\sqrt{\bar{\beta}}=\frac{\sqrt{\kappa}-1}{\sqrt{\kappa}+1}
    \end{equation*}
\end{proof}

\subsection{Nesterov's method for smooth convex functions}

\begin{theorem}[Modified from \citep{Bubeck_2015,Nesterov_2018}]\label{theorem:Nesterov_Convex}
    For a $\bar{L}$-smooth convex function $f$, the iterates $\{\theta_k\}_{k=0}^{\infty}$ generated by \eqref{e:Nesterov_Two_Convex_TV_beta} with $\theta_0=\nu_0$, $\bar{\alpha}=1/\bar{L}$, and $\bar{\beta}_k$ chosen as
    \begin{align}\label{e:beta_k}
        \begin{split}
            \iota_{-1} &= 0,\quad \iota_{k+1}=\frac{1+\sqrt{1+4\iota_k^2}}{2},\\
            \bar{\beta}_k&=\frac{\iota_k-1}{\iota_{k+1}},
        \end{split}
    \end{align}
    satisfy
    \begin{equation}\label{e:Nesterov_convergence}
        f(\theta_k)-f(\theta^*)\leq\frac{2\bar{L}\lVert\theta_0-\theta^*\rVert^2}{(k-1)^2},
    \end{equation}
    and therefore if
    \begin{equation}
        k\geq\left\lceil\sqrt{\frac{2\bar{L}\lVert\theta_0-\theta^*\rVert^2}{\epsilon}}+1\right\rceil,
    \end{equation}
    then $f(\theta_k)-f(\theta^*)\leq\epsilon$.
\end{theorem}
\begin{proof}[Modified from \citep{Bubeck_2015}]
    We begin by upper bounding the difference in objective values between iterates:
    \begin{align}\label{e:Nesterov_Convex_f_kp1_m_k}
        \begin{split}
            f(\theta_{k+1})-f(\theta_k)&=f(\theta_{k+1})-f(\nu_k)+f(\nu_k)-f(\theta_k)\\
            &\overset{\eqref{e:smooth_convex}}{\leq}\nabla f(\nu_k)^T(\theta_{k+1}-\nu_k)+\frac{\bar{L}}{2}\lVert\theta_{k+1}-\nu_k\rVert^2+f(\nu_k)-f(\theta_k)\\
            &\overset{\eqref{e:convex}}{\leq}\nabla f(\nu_k)^T(\theta_{k+1}-\nu_k)+\frac{\bar{L}}{2}\lVert\theta_{k+1}-\nu_k\rVert^2+\nabla f(\nu_k)^T(\nu_k-\theta_k)\\
            &\overset{\eqref{e:Nesterov_Two_Convex}}{=}\nabla f(\nu_k)^T(\theta_{k+1}-\theta_k)+\frac{1}{2\bar{L}}\lVert\nabla f(\nu_k)\rVert^2\\
            &\overset{\eqref{e:Nesterov_Two_Convex}}{=}\nabla f(\nu_k)^T(\nu_k-\theta_k)-\frac{1}{2\bar{L}}\lVert\nabla f(\nu_k)\rVert^2\\
            &\overset{\eqref{e:Nesterov_Two_Convex}}{=}\bar{L}(\nu_k-\theta_{k+1})^T(\nu_k-\theta_k)-\frac{\bar{L}}{2}\lVert\nu_k-\theta_{k+1}\rVert^2.
        \end{split}
    \end{align}
    With the same procedure the difference in objective value between the next iterate and optimum may be bounded as
    \begin{equation}\label{e:Nesterov_Convex_f_kp1_m_star}
        f(\theta_{k+1})-f(\theta^*)\leq\bar{L}(\nu_k-\theta_{k+1})^T(\nu_k-\theta^*)-\frac{\bar{L}}{2}\lVert\nu_k-\theta_{k+1}\rVert^2.
    \end{equation}
    Using \eqref{e:Nesterov_Convex_f_kp1_m_k} and \eqref{e:Nesterov_Convex_f_kp1_m_star}:
    \begin{align*}
        &(\iota_k-1)(f(\theta_{k+1})-f(\theta_k))+(f(\theta_{k+1})-f(\theta^*))=\\
        &\quad \iota_k(f(\theta_{k+1})-f(\theta^*))-(\iota_k-1)(f(\theta_k)-f(\theta^*))\leq\\
        &\quad\bar{L}(\nu_k-\theta_{k+1})^T(\iota_k\nu_k-(\iota_k-1)\theta_k-\theta^*)-\frac{\bar{L}\iota_k}{2}\lVert\nu_k-\theta_{k+1}\rVert^2
    \end{align*}
    From \eqref{e:beta_k} it can be noted that $\iota_k^2-\iota_k=\iota_{k-1}^2$. Furthermore, for all $a,b\in\mathbb{R}^N$: $2a^Tb-\lVert a\rVert^2=\lVert b\rVert^2-\lVert b-a\rVert^2$. Therefore,
    \begin{align}\label{e:Nesterov_Convex_intermediate1}
        \begin{split}
            \iota_k^2(f(\theta_{k+1})&-f(\theta^*))-\iota_{k-1}^2(f(\theta_k)-f(\theta^*))\leq\\
            &\iota_k\bar{L}(\nu_k-\theta_{k+1})^T(\iota_k\nu_k-(\iota_k-1)\theta_k-\theta^*)-\frac{\bar{L}\iota_k^2}{2}\lVert\nu_k-\theta_{k+1}\rVert^2\\
            &=\frac{\bar{L}}{2}\left(2\iota_k(\nu_k-\theta_{k+1})^T(\iota_k\nu_k-(\iota_k-1)\theta_k-\theta^*)-\lVert\iota_k(\nu_k-\theta_{k+1})\rVert^2\right)\\
            &=\frac{\bar{L}}{2}\left(\lVert\iota_k\nu_k-(\iota_k-1)\theta_k-\theta^*\rVert^2-\lVert\iota_k\theta_{k+1}-(\iota_k-1)\theta_k-\theta^*\rVert^2\right).
        \end{split}
    \end{align}
    Multiplying \eqref{e:Nesterov_Two_Convex} by $\iota_{k+1}$ and using \eqref{e:beta_k},
    \begin{equation*}
        \iota_{k+1}\nu_{k+1}=\left(\iota_{k+1}+\iota_k-1\right)\theta_{k+1} - (\iota_k-1)\theta_{k},
    \end{equation*}
    and thus,
    \begin{equation}\label{e:Nesterov_Convex_intermediate2}
        \iota_{k+1}\nu_{k+1}-(\iota_{k+1}-1)\theta_{k+1}=\iota_k\theta_{k+1}-(\iota_k-1)\theta_k.
    \end{equation}
    Using \eqref{e:Nesterov_Convex_intermediate1} and \eqref{e:Nesterov_Convex_intermediate2},
    \begin{align*}
        &\iota_k^2(f(\theta_{k+1})-f(\theta^*))-\iota_{k-1}^2(f(\theta_k)-f(\theta^*))\\
        &\quad\leq\frac{\bar{L}}{2}\left(\lVert\left\{\iota_k\nu_k-(\iota_k-1)\theta_k\right\}-\theta^*\rVert^2-\lVert\left\{\iota_{k+1}\nu_{k+1}-(\iota_{k+1}-1)\theta_{k+1}\right\}-\theta^*\rVert^2\right).
    \end{align*}
    Collecting terms with the initial conditions,
    \begin{equation}\label{e:Nesterov_Convex_intermediate3}
        f(\theta_k)-f(\theta^*)\leq \frac{\bar{L}}{2\iota_{k-1}^2}\lVert\theta_0-\theta^*\rVert^2.
    \end{equation}
    Next, it will be shown by induction that the  parameters $\iota_k$ in \eqref{e:beta_k} satisfy $\iota_k\geq k/2$ for all $k\geq0$. Base case: Note from the initial condition that the inequality is satisfied. Using \eqref{e:beta_k} with $\iota_k\geq k/2$ it can be seen that $\iota_{k+1}\geq(1+\sqrt{1+4(k/2)^2})/2\geq(k+1)/2$, completing the proof by induction. Therefore \eqref{e:Nesterov_Convex_intermediate3} may be bounded as
    \begin{equation*}
        f(\theta_k)-f(\theta^*)\leq\frac{2\bar{L}\lVert\theta_0-\theta^*\rVert^2}{(k-1)^2}.
    \end{equation*}
\end{proof}

\subsection{Nesterov's method for smooth strongly convex functions}

\begin{manualtheorem}{\ref{theorem:Nesterov_SC_paper} from Main Text}[Modified from \citep{Bubeck_2015,Nesterov_2018}]
    For a $\bar{L}$-smooth and $\mu$-strongly convex function $f$, the iterates $\{\theta_k\}_{k=0}^{\infty}$ generated by \eqref{e:Nesterov_Two_Convex} with $\theta_0=\nu_0$, $\bar{\alpha}=1/\bar{L}$, $\kappa=\bar{L}/\mu$, and $\bar{\beta}=(\sqrt{\kappa}-1)/(\sqrt{\kappa}+1)$ satisfy
    \begin{equation} \label{e:Nesterov_ConstantBetaBar_convergence}
        f(\theta_k)-f(\theta^*)\leq\frac{\bar{L}+\mu}{2}\lVert\theta_0-\theta^*\rVert^2\exp\left(-\frac{k}{\sqrt{\kappa}}\right),
    \end{equation}
    and therefore if
    \begin{equation}
        k\geq\left\lceil\sqrt{\kappa}\log\left(\frac{(\bar{L}+\mu)\lVert\theta_0-\theta^*\rVert^2}{2\epsilon}\right)\right\rceil,
    \end{equation}
    then $f(\theta_k)-f(\theta^*)\leq\epsilon$.
\end{manualtheorem}
\begin{proof}[Modified from \citep{Bubeck_2015}]
    Consider the following sequence of $\mu$-strongly convex quadratic functions defined as
    \begin{align}\label{e:Nesterov_Phi_Definition}
        \begin{split}
            \Phi_0(\theta)&=f(\nu_0)+\frac{\mu}{2}\lVert\theta-\nu_0\rVert^2\\
            \Phi_{k+1}(\theta)&=\left(1-\frac{1}{\sqrt{\kappa}}\right)\Phi_k(\theta)+\frac{1}{\sqrt{\kappa}}\left(f(\nu_k)+\nabla f(\nu_k)^T(\theta-\nu_k)+\frac{\mu}{2}\lVert\theta-\nu_k\rVert^2\right).
        \end{split}
    \end{align}
    Given that $f$ is $\mu$-strongly convex,
    \begin{equation*}
        f(\theta)\geq f(\nu_k)+\nabla f(\nu_k)^T(\theta-\nu_k)+\frac{\mu}{2}\lVert\theta-\nu_k\rVert^2,
    \end{equation*}
    and thus $\Phi_{k+1}(\theta)$ may be bounded as
    \begin{equation*}
        \Phi_{k+1}(\theta)\leq\left(1-\frac{1}{\sqrt{\kappa}}\right)\Phi_k(\theta)+\frac{1}{\sqrt{\kappa}}f(\theta).
    \end{equation*}
    Collecting terms, the function $\Phi_k$ can be seen to provide a lower bound approximation of $f$ as
    \begin{equation}\label{e:Nesterov_Phi_kp1}
        \Phi_k(\theta)\leq\left(1-\frac{1}{\sqrt{\kappa}}\right)^k\left(\Phi_0(\theta)-f(\theta)\right)+f(\theta).
    \end{equation}
    Assume for the moment the following inequality relating the function value at step $k$, $f(\theta_k)$, to the minimum value of the function $\Phi_k$:
    \begin{equation}\label{e:Nesterov_f_leq_Phi_k}
        f(\theta_k)\leq\min_{\theta}\Phi_k(\theta)=\Phi^*_k.
    \end{equation}
    Assuming the inequality in \eqref{e:Nesterov_f_leq_Phi_k}, the proof of the theorem follows as
    \begin{align}
        \begin{split}
            f(\theta_k)-f(\theta^*)&\overset{\eqref{e:Nesterov_f_leq_Phi_k}}{\leq} \Phi_k(\theta^*)-f(\theta^*)\overset{\eqref{e:Nesterov_Phi_kp1}}{\leq}\left(1-\frac{1}{\sqrt{\kappa}}\right)^k\left(\Phi_0(\theta^*)-f(\theta^*)\right)\\
            &\overset{\eqref{e:Nesterov_Phi_Definition}}{\leq}\left(1-\frac{1}{\sqrt{\kappa}}\right)^k\left(\frac{\mu}{2}\lVert\nu_0-\theta^*\rVert^2+f(\nu_0)-f(\theta^*)\right)\\
            &\overset{\nu_0=\theta_0}{\leq}\left(1-\frac{1}{\sqrt{\kappa}}\right)^k\left(\frac{\mu}{2}\lVert\theta_0-\theta^*\rVert^2+f(\theta_0)-f(\theta^*)\right)\\
            &\overset{\eqref{e:smooth_convex}}{\leq}\frac{\bar{L}+\mu}{2}\lVert\theta_0-\theta^*\rVert^2\left(1-\frac{1}{\sqrt{\kappa}}\right)^k\\
            &\leq \frac{\bar{L}+\mu}{2}\lVert\theta_0-\theta^*\rVert^2\exp\left(-\frac{k}{\sqrt{\kappa}}\right)
        \end{split}
    \end{align}
    
    Thus it remains to prove the inequality in \eqref{e:Nesterov_f_leq_Phi_k}.
    
    The proof of the inequality in \eqref{e:Nesterov_f_leq_Phi_k} follows by induction. Base case: From \eqref{e:Nesterov_Phi_Definition}$: \Phi_0(\theta_0)=f(\nu_0)+\frac{\mu}{2}\lVert\theta_0-\nu_0\rVert^2=f(\theta_0)$ as $\theta_0=\nu_0$. Bounding the function value at the next iterate:
    \begin{align}\label{e:Nesterov_f_theta_kp1}
        \begin{split}
            &f(\theta_{k+1})\overset{\eqref{e:smooth_convex}}{\leq}f(\nu_k)+\nabla f(\nu_k)^T(\theta_{k+1}-\nu_k)+\frac{\bar{L}}{2}\lVert \theta_{k+1}-\nu_k\rVert^2\overset{\eqref{e:Nesterov_Two_Convex}}{\leq}f(\nu_k)-\frac{1}{2\bar{L}}\lVert\nabla f(\nu_k)\rVert^2\\
            &=\left(1-\frac{1}{\sqrt{\kappa}}\right)f(\theta_k)+\left(1-\frac{1}{\sqrt{\kappa}}\right)\left(f(\nu_k)-f(\theta_k)\right)+\frac{1}{\sqrt{\kappa}}f(\nu_k)-\frac{1}{2\bar{L}}\lVert\nabla f(\nu_k)\rVert^2\\
            &\overset{\eqref{e:convex}}{\leq}\left(1-\frac{1}{\sqrt{\kappa}}\right)f(\theta_k)+\left(1-\frac{1}{\sqrt{\kappa}}\right)\nabla f(\nu_k)^T(\nu_k-\theta_k)+\frac{1}{\sqrt{\kappa}}f(\nu_k)-\frac{1}{2\bar{L}}\lVert\nabla f(\nu_k)\rVert^2\\
            &\overset{\eqref{e:Nesterov_f_leq_Phi_k}}{\leq}\left(1-\frac{1}{\sqrt{\kappa}}\right)\Phi^*_k+\left(1-\frac{1}{\sqrt{\kappa}}\right)\nabla f(\nu_k)^T(\nu_k-\theta_k)+\frac{1}{\sqrt{\kappa}}f(\nu_k)-\frac{1}{2\bar{L}}\lVert\nabla f(\nu_k)\rVert^2
        \end{split}
    \end{align}
    From \eqref{e:Nesterov_Phi_Definition} it can be noted that $\nabla^2\Phi_k(\theta)=\mu I$ and therefore $\Phi_k$ may be of the form
    \begin{equation}\label{e:Nesterov_Phi_k_quadratic}
        \Phi_k(\theta)=\Phi_k^*+(\mu/2)\lVert\theta-v_k\rVert^2,
    \end{equation}
    for a $v_k\in\mathbb{R}^N$. Using this form of $\Phi_k$ alongside \eqref{e:Nesterov_Phi_Definition}, the gradient may be expressed as
    \begin{equation*}
        \mu(\theta-v_{k+1})\overset{\eqref{e:Nesterov_Phi_k_quadratic}}{=}\nabla\Phi_{k+1}(\theta)\overset{\eqref{e:Nesterov_Phi_Definition}}{=}\mu\left(1-\frac{1}{\sqrt{\kappa}}\right)(\theta-v_k)+\frac{1}{\sqrt{\kappa}}\left(\nabla f(\nu_k)+\mu(\theta-\nu_k)\right).
    \end{equation*}
    Solving for $v_{k+1}$:
    \begin{equation}\label{e:Nesterov_v_kp1}
        v_{k+1}=\left(1-\frac{1}{\sqrt{\kappa}}\right)v_k-\frac{1}{\mu\sqrt{\kappa}}\nabla f(\nu_k)+\frac{1}{\sqrt{\kappa}}\nu_k.
    \end{equation}
    Thus using \eqref{e:Nesterov_Phi_Definition} and \eqref{e:Nesterov_v_kp1},
    \begin{align}\label{e:Nesterov_Phi_star_kp1}
        \begin{split}
            \Phi^*_{k+1}+\frac{\mu}{2}\lVert\nu_k-v_{k+1}\rVert^2&=\Phi_{k+1}(\nu_k)=\left(1-\frac{1}{\sqrt{\kappa}}\right)\Phi_k(\nu_k)+\frac{1}{\sqrt{\kappa}}f(\nu_k)\\
            &=\left(1-\frac{1}{\sqrt{\kappa}}\right)\Phi^*_k+\left(1-\frac{1}{\sqrt{\kappa}}\right)\frac{\mu}{2}\lVert\nu_k-v_k\rVert^2+\frac{1}{\sqrt{\kappa}}f(\nu_k)
        \end{split}
    \end{align}
    Using \eqref{e:Nesterov_v_kp1}:
    \begin{align}\label{e:Nesterov_nu_k_m_v_kp1}
        \begin{split}
        &\lVert\nu_k-v_{k+1}\rVert^2\\
        &=\left(1-\frac{1}{\sqrt{\kappa}}\right)^2\lVert\nu_k-v_k\rVert^2+\frac{1}{\mu^2\kappa}\lVert\nabla f(\nu_k)\rVert^2-\frac{2}{\mu\sqrt{\kappa}}\left(1-\frac{1}{\sqrt{\kappa}}\right)\nabla f(\nu_k)^T(v_k-\nu_k).
        \end{split}
    \end{align}
    Combining \eqref{e:Nesterov_Phi_star_kp1} and \eqref{e:Nesterov_nu_k_m_v_kp1}:
    \begin{align}\label{e:Nesterov_Phi_star_kp1_second}
        \begin{split}
            &\Phi^*_{k+1}=\left(1-\frac{1}{\sqrt{\kappa}}\right)\Phi^*_k+\left(1-\frac{1}{\sqrt{\kappa}}\right)\frac{\mu}{2}\lVert\nu_k-v_k\rVert^2+\frac{1}{\sqrt{\kappa}}f(\nu_k)\\
            &-\frac{\mu}{2}\left(1-\frac{1}{\sqrt{\kappa}}\right)^2\lVert\nu_k-v_k\rVert^2-\frac{1}{2\mu\kappa}\lVert\nabla f(\nu_k)\rVert^2+\frac{1}{\sqrt{\kappa}}\left(1-\frac{1}{\sqrt{\kappa}}\right)\nabla f(\nu_k)^T(v_k-\nu_k)\\
            &=\left(1-\frac{1}{\sqrt{\kappa}}\right)\Phi^*_k+\left(1-\frac{1}{\sqrt{\kappa}}\right)\frac{\mu}{2\sqrt{\kappa}}\lVert\nu_k-v_k\rVert^2+\frac{1}{\sqrt{\kappa}}f(\nu_k)\\
            &-\frac{1}{2\bar{L}}\lVert\nabla f(\nu_k)\rVert^2+\frac{1}{\sqrt{\kappa}}\left(1-\frac{1}{\sqrt{\kappa}}\right)\nabla f(\nu_k)^T(v_k-\nu_k).
        \end{split}
    \end{align}
    Next it will be shown that $v_k-\nu_k=\sqrt{\kappa}(\nu_k-\theta_k)$ via a proof by induction. Base case: Note from the initial condition $\theta_0=\nu_0$, \eqref{e:Nesterov_Phi_Definition}, and \eqref{e:Nesterov_Phi_k_quadratic} it can be seen that $v_0=\nu_0=\theta_0$. Consider the following equality statements at the next iteration:
    \begin{align}\label{e:Nesterov_Last_Induction}
        \begin{split}
            v_{k+1}-\nu_{k+1}&\overset{\eqref{e:Nesterov_v_kp1}}{=}\left(1-\frac{1}{\sqrt{\kappa}}\right)v_k-\frac{1}{\mu\sqrt{\kappa}}\nabla f(\nu_k)+\frac{1}{\sqrt{\kappa}}\nu_k-\nu_{k+1}\\
            &\overset{\text{I.H.}}{=}\sqrt{\kappa}\nu_k-(\sqrt{\kappa}-1)\theta_k-\frac{\sqrt{\kappa}}{\bar{L}}\nabla f(\nu_k)-\nu_{k+1}\\
            &\overset{\eqref{e:Nesterov_Two_Convex}}{=}\sqrt{\kappa}\theta_{k+1}-(\sqrt{\kappa}-1)\theta_k-\nu_{k+1}\\
            &\overset{\eqref{e:Nesterov_Two_Convex}}{=}\sqrt{\kappa}(\nu_{k+1}-\theta_{k+1})
        \end{split}
    \end{align}
    Using \eqref{e:Nesterov_f_theta_kp1}, \eqref{e:Nesterov_Phi_star_kp1_second}, and \eqref{e:Nesterov_Last_Induction}, it can be seen that
    \begin{equation}
        f(\theta_{k+1})\leq\Phi^*_{k+1},
    \end{equation}
    thus completing the proof of the inequality of \eqref{e:Nesterov_f_leq_Phi_k} by induction.
\end{proof}

\subsection{Regularization technique for non-strongly convex functions}
\label{ss:Regularization_New}

\begin{manuallemma}{\ref{l:L_N_rate} from Main Text (with proof)}
    The iterates $\{\theta_k\}_{k=0}^{\infty}$ generated by \eqref{e:Nesterov_Two_Convex} for the function in \eqref{e:Strongly_Convex_Objective} with $\theta_0=\nu_0$, $\Psi\geq\max\{1,\lVert\theta_0-\theta^*\rVert^2\}$, $\mu=\epsilon/\Psi$, $\bar{L}=1+\mu$, $\bar{\alpha}=1/\bar{L}$, $\kappa=\bar{L}/\mu$, $\bar{\beta}=(\sqrt{\kappa}-1)/(\sqrt{\kappa}+1)$, if
    \begin{equation}
        k\geq\left\lceil\sqrt{1+\frac{\Psi}{\epsilon}}\log\left(2+\frac{\Psi}{\epsilon}\right)\right\rceil,\text{ then $\frac{L(\theta_k)-L(\theta^*)}{\N}\leq\epsilon$}.
    \end{equation}
\end{manuallemma}
\begin{proof}
    The normalized loss function gap may be bounded from above using $f$ in \eqref{e:Strongly_Convex_Objective}, $\Psi\geq\max\{1,\lVert\theta_0-\theta^*\rVert^2\}$, and $\mu=\epsilon/\Psi$ as
    \begin{align*}
        \frac{L(\theta_k)-L(\theta^*)}{\N}&=f(\theta_k)-f(\theta^*)+\frac{\mu}{2}\left(\lVert\theta^*-\theta_0\rVert^2-\lVert\theta_k-\theta_0\rVert^2\right)\\
        &\leq f(\theta_k)-f(\theta^*_{\epsilon})+\frac{\epsilon}{2},
    \end{align*}
    where $f(\theta^*_{\epsilon})$ is the optimal value of $f$, that is, $\nabla f(\theta^*_{\epsilon})= \frac{\nabla L(\theta^*_{\epsilon})}{\N}+\mu(\theta^*_{\epsilon}-\theta_0)=0$. Applying the result of Theorem \ref{theorem:Nesterov_SC_paper} to $f$ in \eqref{e:Strongly_Convex_Objective} with $\Psi\geq\max\{1,\lVert\theta_0-\theta^*\rVert^2\}$, $\mu=\epsilon/\Psi$, $\bar{L}=1+\mu$, $\bar{\alpha}=1/\bar{L}$, $\kappa=\bar{L}/\mu$, $\bar{\beta}_k=(\sqrt{\kappa}-1)/(\sqrt{\kappa}+1)$, the normalized loss function gap may be bounded as
    \begin{align*}
        \begin{split}
            \frac{L(\theta_k)-L(\theta^*)}{\N}&\leq \frac{\bar{L}+\mu}{2}\lVert\theta_0-\theta^*_{\epsilon}\rVert^2\exp\left(-\frac{k}{\sqrt{\kappa}}\right)+\frac{\epsilon}{2}\\
            &\leq\frac{\bar{L}+\mu}{2}\lVert\theta_0-\theta^*\rVert^2\exp\left(-\frac{k}{\sqrt{\kappa}}\right)+\frac{\epsilon}{2}\\
            &=\frac{\Psi+2\epsilon}{2\Psi}\lVert\theta_0-\theta^*\rVert^2\exp\left(-\frac{k}{\sqrt{\kappa}}\right)+\frac{\epsilon}{2}\\
            &\leq\frac{\Psi+2\epsilon}{2}\exp\left(-\frac{k}{\sqrt{\kappa}}\right)+\frac{\epsilon}{2}\\
            &=\frac{\Psi+2\epsilon}{2}\exp\left(-\frac{k}{\sqrt{1+\frac{\Psi}{\epsilon}}}\right)+\frac{\epsilon}{2}
        \end{split}
    \end{align*}
    Thus $\frac{L(\theta_k)-L(\theta^*)}{\N}\leq\epsilon$ if,
    \begin{equation*}
        k\geq\left\lceil\sqrt{1+\frac{\Psi}{\epsilon}}\log\left(2+\frac{\Psi}{\epsilon}\right)\right\rceil.
    \end{equation*}
\end{proof}
\begin{manuallemma}{\ref{l:L_rate} from Main Text (with proof)}
    The iterates $\{\theta_k\}_{k=0}^{\infty}$ generated by \eqref{e:Nesterov_Two_Convex} for the function in \eqref{e:Strongly_Convex_Objective} with $\theta_0=\nu_0$, $\Psi\geq\max\{1,\N\lVert\theta_0-\theta^*\rVert^2\}$, $\mu=\epsilon/\Psi$, $\bar{L}=1+\mu$, $\bar{\alpha}=1/\bar{L}$, $\kappa=\bar{L}/\mu$, $\bar{\beta}=(\sqrt{\kappa}-1)/(\sqrt{\kappa}+1)$, if
    \begin{equation}
        k\geq\left\lceil\sqrt{1+\frac{\Psi}{\epsilon}}\log\left(2+\frac{\Psi}{\epsilon}\right)\right\rceil,\text{ then $L(\theta_k)-L(\theta^*)\leq\epsilon$}.
    \end{equation}
\end{manuallemma}
\begin{proof}
    The loss function gap may be bounded from above using $f$ in \eqref{e:Strongly_Convex_Objective}, $\Psi\geq\max\{1,\N\lVert\theta_0-\theta^*\rVert^2\}$, and $\mu=\epsilon/\Psi$ as
    \begin{align*}
        L(\theta_k)-L(\theta^*)&=\N(f(\theta_k)-f(\theta^*))+\frac{\mu}{2}\N\left(\lVert\theta^*-\theta_0\rVert^2-\lVert\theta_k-\theta_0\rVert^2\right)\\
        &\leq \N(f(\theta_k)-f(\theta^*_{\epsilon}))+\frac{\epsilon}{2},
    \end{align*}
    where $f(\theta^*_{\epsilon})$ is the optimal value of $f$, that is, $\nabla f(\theta^*_{\epsilon})= \frac{\nabla L(\theta^*_{\epsilon})}{\N}+\mu(\theta^*_{\epsilon}-\theta_0)=0$. Applying the result of Theorem \ref{theorem:Nesterov_SC_paper} to $f$ in \eqref{e:Strongly_Convex_Objective} with $\Psi\geq\max\{1,\N\lVert\theta_0-\theta^*\rVert^2\}$, $\mu=\epsilon/\Psi$, $\bar{L}=1+\mu$, $\bar{\alpha}=1/\bar{L}$, $\kappa=\bar{L}/\mu$, $\bar{\beta}_k=(\sqrt{\kappa}-1)/(\sqrt{\kappa}+1)$, the loss function gap may be bounded as
    \begin{align*}
        \begin{split}
            L(\theta_k)-L(\theta^*)&\leq \frac{\bar{L}+\mu}{2}\N\lVert\theta_0-\theta^*_{\epsilon}\rVert^2\exp\left(-\frac{k}{\sqrt{\kappa}}\right)+\frac{\epsilon}{2}\\
            &\leq\frac{\bar{L}+\mu}{2}\N\lVert\theta_0-\theta^*\rVert^2\exp\left(-\frac{k}{\sqrt{\kappa}}\right)+\frac{\epsilon}{2}\\
            &=\frac{\Psi+2\epsilon}{2\Psi}\N\lVert\theta_0-\theta^*\rVert^2\exp\left(-\frac{k}{\sqrt{\kappa}}\right)+\frac{\epsilon}{2}\\
            &\leq\frac{\Psi+2\epsilon}{2}\exp\left(-\frac{k}{\sqrt{\kappa}}\right)+\frac{\epsilon}{2}\\
            &=\frac{\Psi+2\epsilon}{2}\exp\left(-\frac{k}{\sqrt{1+\frac{\Psi}{\epsilon}}}\right)+\frac{\epsilon}{2}
        \end{split}
    \end{align*}
    Thus $L(\theta_k)-L(\theta^*)\leq\epsilon$ if,
    \begin{equation*}
        k\geq\left\lceil\sqrt{1+\frac{\Psi}{\epsilon}}\log\left(2+\frac{\Psi}{\epsilon}\right)\right\rceil.
    \end{equation*}
\end{proof}

\clearpage
\subsection{Comparison of non-asymptotic convergence rates}
\label{ss:Comparison_with_Constants}

\begin{table}[h]
    \caption{Comparison of gradient-based methods for the linear regression squared loss function in \eqref{e:Squared_Loss_N} with constants included (constants chosen optimally according to each proof), for the case of a constant regressor $\phi$. Here $\bar{L}=\lVert\phi\rVert^2$.}
    \label{t:Comparison_Gradient_Methods_Constants}
    \centering
    \begin{tabular}{lc}
        \toprule
        Algorithm & \# Iterations for $L(\theta_k)-L(\theta^*)\leq\epsilon$ \\
        \midrule
        Gradient Descent Fixed (Theorem \ref{theorem:GD_k}) & $\left\lceil\frac{2\bar{L}\lVert\theta_0-\theta^*\rVert^2}{\epsilon}-4\right\rceil$ \\
        Nesterov Acceleration (Theorem \ref{theorem:Nesterov_Convex}) & $\left\lceil\sqrt{\frac{2\bar{L}\lVert\theta_0-\theta^*\rVert^2}{\epsilon}}+1\right\rceil$ \\
        \rowcolor{lightgray!40}This Paper (Lemma \ref{l:L_rate}) & $\left\lceil\sqrt{1+\frac{(1+\bar{L})\lVert\theta_0-\theta^*\rVert^2}{\epsilon}}\log\left(2+\frac{(1+\bar{L})\lVert\theta_0-\theta^*\rVert^2}{\epsilon}\right)\right\rceil$ \\
        \bottomrule
    \end{tabular}
\end{table}

\begin{figure}[ht]
    \centerline{
	    \includegraphics[trim={0.5cm 1cm 2cm 1.5cm},clip,width=0.5\textwidth]{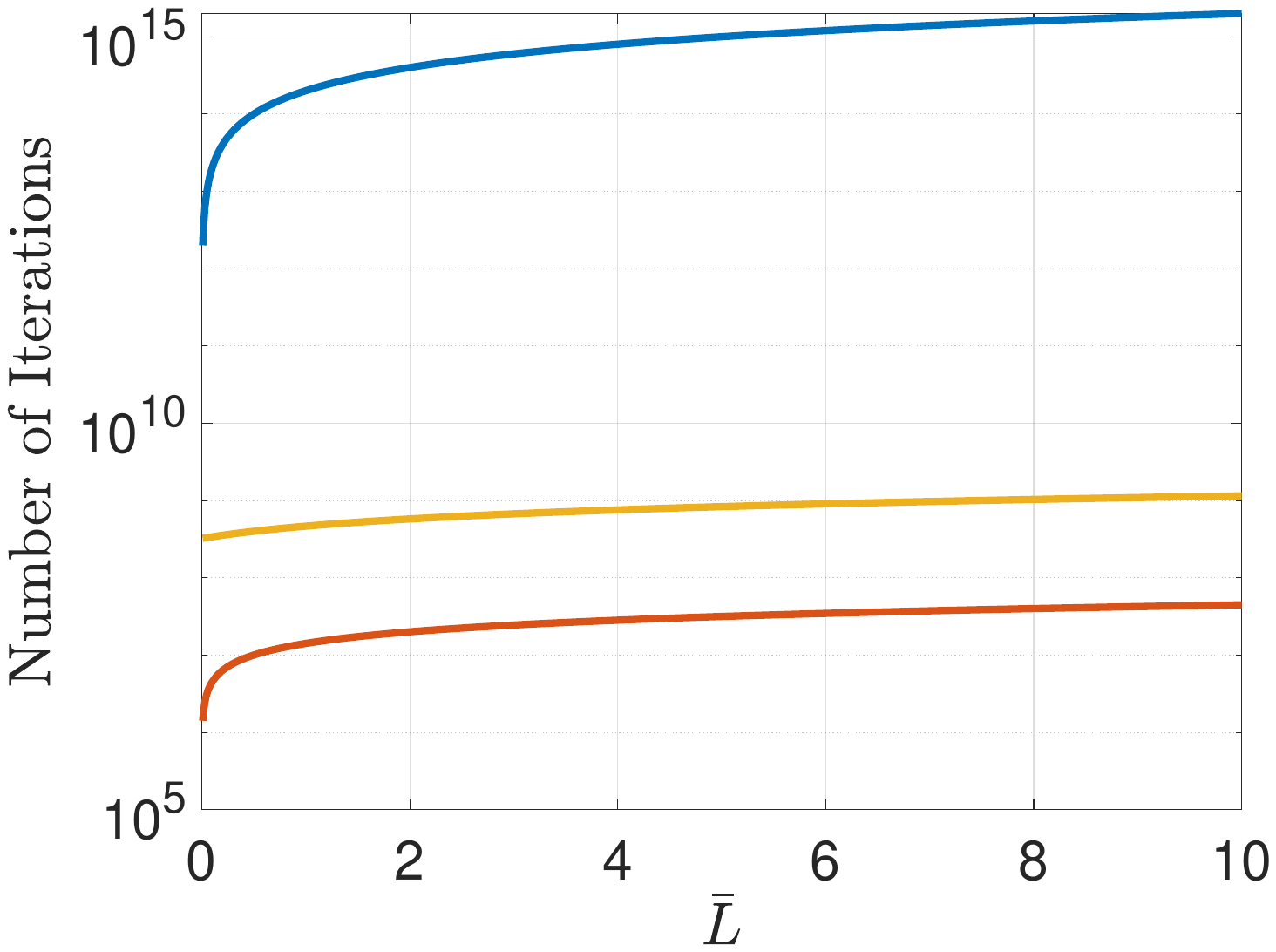}
	    \includegraphics[trim={0.5cm 1cm 2cm 1.5cm},clip,width=0.5\textwidth]{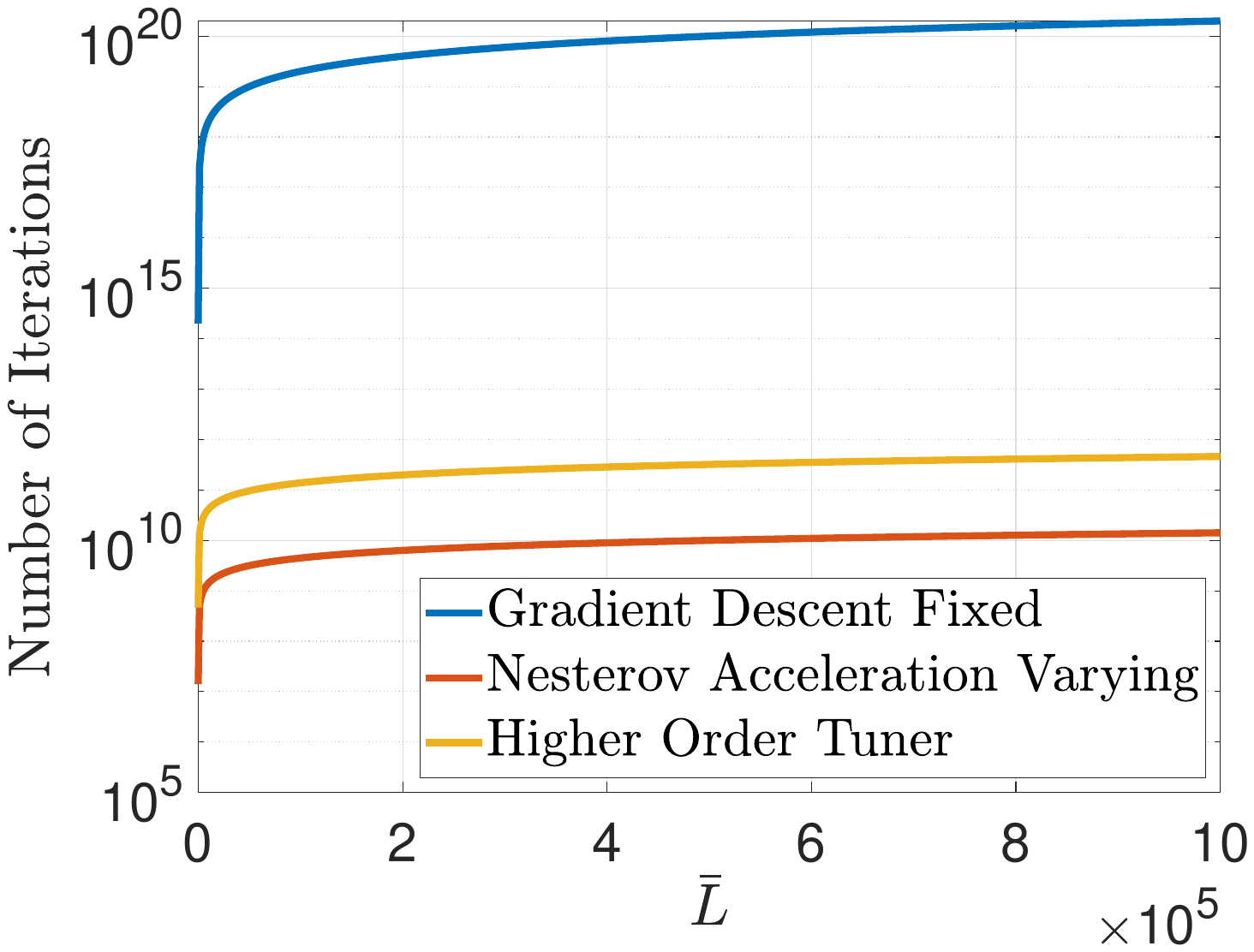}
	    }
	    \caption{(to be viewed in color) Number of iterations to reach an $\epsilon$ sub-optimal point as in Table \ref{t:Comparison_Gradient_Methods_Constants}, for $\epsilon=10^{-14}$, and $\lVert\theta_0-\theta^*\rVert=1$.}
    \label{f:Comparison_Gradient_Methods_Constants}
\end{figure}

\clearpage
\section{Details of experiments in the main paper and additional experiments}
\label{s:Experiment_details}
This section provides additional information regarding the experiments reported in the main paper in Sections \ref{sec:SHP} and \ref{sec:IDP}, which are covered in \ref{ProblemStatement} and in \ref{IDPProblemStatement}, respectively. Here, the details of the algorithms compared as well as the selection of the hyperparameters are discussed. Additional experiments are covered in \ref{SHPadd} and \ref{IDPadd}, where different types of loss functions are discussed where the underlying regressors have new variations.

All simulations were implemented in Python. Online python notebooks can be found at \underline{\href{https://colab.research.google.com/drive/14V1X4kyEU3G6NMZuCaswSHSLaywUbqEe}{link1}} and at \underline{\href{https://colab.research.google.com/drive/1cQRk44jaMNXCtOvEZ-QWIeqzKZBujQQN}{link2}} for the corresponding problems. Videos of the image deblurring problems can be found also at \underline{\href{https://drive.google.com/drive/folders/1oQFFWuUox-ChbFeQFveGlLR2nim-8y0w?usp=sharing}{Google Drive folder}}. 

The experiments were run on a MacBook Pro 2018 with a 2.2 GHz Intel Core i7. We report the average iterations/second for each problem:
\begin{itemize}
    \item Nesterov's smooth convex function: For $n=401$, the average iterations/second was: $64.63$.
    \item Image Deblurring Problem: The average iterations/second was: $130.6$.
\end{itemize}

\subsection{Nesterov's smooth convex function}
\label{SmoothConvexMinProblem}

\subsubsection{Details of experiments in the main paper}\label{ProblemStatement}
The function used in this experiment, described in Section \ref{sec:SHP}, is modified from \citep[p.~69]{Nesterov_2018} and consists of minimizing a function of the form
\begin{equation}\label{e:minSHP}
    \underset{\theta\in  \mathbb{R}^{n}}{\text{min}}L_{k}(\theta),
\end{equation}
where
\begin{equation} \label{e:SHP}
    L_{k}(\theta)=\frac{a_k}{4}\left\{ \frac{1}{2}\left[b_k(\theta^{(1)})^{2}+c_k\sum_{i=1}^{n-1}(\theta^{(i)}-\theta^{(i+1)})^{2}+b_k(\theta^{(n)})^{2}\right]-d_k\theta^{(1)}\right\}.
\end{equation}
In \eqref{e:SHP}, $a_k, b_k, c_k$ and $d_k$ are positive scalars which depend on the iteration $k$, and the superscript $(i)$ indicates the $i$-th element of the vector $\theta$.

The gradient of this function can be written in the following form:
\begin{equation}
    \begin{split}
    \nabla L_{k}(\theta) &= \frac{a_k}{4}\left( A_k\theta - D_k \right) \\
    &=\frac{a_k}{4}\left(\left[\begin{array}{cccccc}
        b_k+c_k & -c_k & 0 & \cdots & 0 & 0\\
        -c_k & 2c_k & -c_k & \cdots & 0 & 0\\
        0 & -c_k & 2c_k & \cdots & 0 & 0\\
        \vdots & \vdots & \vdots & \ddots & \vdots & \vdots\\
        0 & 0 & 0 & \cdots & 2c_k & -c_k\\
        0 & 0 & 0 & \cdots & -c_k & b_k+c_k
        \end{array}\right]\left[\begin{array}{c}
        \theta^{(1)}\\
        \theta^{(2)}\\
        \theta^{(3)}\\
        \vdots\\
        \theta^{(n-1)}\\
        \theta^{(n)}
        \end{array}\right]-\left[\begin{array}{c}
        d_k\\
        0\\
        0\\
        \vdots\\
        0\\
        0
        \end{array}\right]\right)
    \end{split}{}
\end{equation}{}

The problem in \eqref{e:minSHP} can therefore be solved using the following system of equations:
\begin{equation}
    A_k\theta = D_k.
\end{equation}
Because of the structure of $D_k$, we only need the first column of $A_k^{-1}$, to obtain the optimal solution $\theta^*$, which can be determined as
\begin{equation}\label{e:optimaltheta}
    {\theta_k^{*}}^{(i)} = d_k \left[ \frac{(n-i)b_k+c_k}{(n-1)b_k^2+2b_kc_k} \right]\text{ for } 1 \leq i\leq n, 
\end{equation}

which shows that $\theta^*$ does not depend on $a_k$. However, the minimum value of $L_k(\theta^*)$ is given by
\begin{equation}\label{e:OptimalValueF_m}
    L_k^*=L_k(\theta^*)=\frac{a_k}{8}(-1+\frac{1}{n+1}),
\end{equation}{}
where it has been assumed that $b_k=c_k=d_k=1$. Also, the problem can be written in the form of $L_k(\theta)=||\phi_k^T 
\theta||^2+B_k^T\theta$, where  $\phi_k^T$ and $B_k^T$ are:
\begin{equation} \label{e:phiconstant}
    \phi_k^{T}=\frac{\sqrt{2a_k}}{4}\left[\begin{array}{ccccccc}
\sqrt{b_k} & 0 & 0 & \cdots & 0 & 0 & 0\\
0 & 0 & 0 & \cdots & 0 & 0 & \sqrt{b_k}\\
\sqrt{c_k} & -\sqrt{c_k} & 0 & \cdots & 0 & 0 & 0\\
0 & \sqrt{c_k} & -\sqrt{c_k} & \cdots & 0 & 0 & 0\\
\vdots & \vdots & \vdots & \ddots & \vdots & \vdots & \vdots\\
0 & 0 & 0 & \cdots & \sqrt{c_k} & -\sqrt{c_k} & 0\\
0 & 0 & 0 & \cdots & 0 & \sqrt{c_k} & -\sqrt{c_k}
\end{array}\right]
\end{equation}{}

\begin{equation}
    B_k^{T}=\frac{-a_k}{4}\left[\begin{array}{cccc}
d_k & 0 & \cdots & 0\end{array}\right]
\end{equation}{}

In this form, the gradient of the function can also be written as:
\begin{equation}
    \nabla L_{k}(\theta) = 2\phi_k\phi_k^\intercal\theta +B_k 
\end{equation}{}

Also, note that the hessian of $L_k(\theta)$ is $\nabla^2 L_k(\theta)=2\phi_k\phi_k^\intercal=\frac{a_k}{4}A_k \leq \bar{L}I $, and therefore we can obtain the $\bar{L}$-smoothness parameter of the function by upper-bounding its maximum eigenvalue. For instance, for $b_k=c_k=1$, the $\bar{L}$-smoothness parameter is $a_k$.

The results reported in Figure \ref{fig:SHPconstant} and Figure \ref{fig:SHPstep2} corresponded to the case where $\theta \in \mathbb{R}^{401}$ and to a fixed choice of $b_k$, $c_k$, $d_k$ as 
\begin{equation} \label{e:abcd_sims}
    b_k=c_k=d_k=1,
\end{equation}
while $a_k$ was switched from 2 to 8000 in Figure \ref{fig:SHPconstant} and from 2 to 8 in Figure \ref{fig:SHPstep2}, both at $k=500$. As mentioned in Section \ref{sec:SHP}, this corresponded to a step change in the $\Bar{L}$-smoothness parameter of the cost function $L_k$.

\paragraph{(A) Algorithms:} \label{SHPAlg}
The algorithms that were compared in Figure \ref{fig:SHPconstant} and Figure \ref{fig:SHPstep2} in Section \ref{sec:SHP} were the following, and are listed in the same order as in the legend in Figure \ref{fig:SHPsimulations}.
\begin{itemize}
    \item Gradient Descent \eqref{e:Gradient_Method}:
    \begin{equation*}
        \theta_{k+1}=\theta_k-\bar{\alpha}\nabla L_k(\theta_k)
    \end{equation*}
    \item Normalized Gradient Descent:
    \begin{equation}\label{e:NormalizedGD_App}
        \theta_{k+1}=\theta_k-\frac{\bar{\alpha}}{\N_k}\nabla L_k(\theta_k)
    \end{equation}
    \item Nesterov Acceleration with time-varying $\bar{\beta}_k$ (\eqref{e:Nesterov_Two_Convex_TV_beta} with $\bar{\beta}_k$ chosen as in \eqref{e:beta_k}):
    \begin{equation} \label{e:NesterovApp}
        \begin{split}
            \theta_{k}&= \nu_k - \bar{\alpha} \nabla L_k(\theta_k) \\
            \nu_{k+1}&=(1+\bar{\beta}_k)\theta_{k+1}-\Bar{\beta}_k\theta_k
        \end{split}{}
    \end{equation}{}
    \item Nesterov Acceleration with constant $\bar{\beta}$ \eqref{e:Nesterov_Two_Convex}:
    \begin{equation*}
        \begin{split}
        \theta_{k}&= \nu_k - \bar{\alpha} \nabla L_k(\theta_k) \\
        \nu_{k+1}&=\theta_k+\Bar{\beta}(\theta_k-\theta_{k-1})
        \end{split}{}
    \end{equation*}{}    
    \item Higher Order Tuner (Algorithm \ref{alg:HOT_R}):
    \begin{equation*}
        \begin{split}
        \bar{\theta}_{k}&= \theta_k - \gamma\beta \left(\frac{\nabla L_k(\theta_k)}{\N_k}+\mu(\theta_k-\theta_0)\right) \\
        \theta_{k+1} &= \bar{\theta}_{k} -\beta (\bar{\theta}_{k}-\vartheta_k) \\
        \vartheta_{k+1}&=\vartheta_k-\gamma\left(\frac{\nabla L_k(\theta_{k+1})}{\N_k}+\mu(\theta_{k+1}-\theta_0)\right)
        \end{split}{}
    \end{equation*}{}   
\end{itemize}

\paragraph{(B) Hyperparameter Selection:} \label{SHPHyperparameter}
As mentioned in Section \ref{sec:SHP}, this experiment was carried out using hyperparameters (a) chosen according to Theorem \ref{th:HOT_paper_Full_Alg}, and (b) chosen as in Proposition \ref{prop:equivalence}, with the results shown in  Figure \ref{fig:SHPconstant} and Figure \ref{fig:SHPstep2}, respectively.
\begin{itemize}
    \item \textbf{Choice (a)}: This corresponds to Figure \ref{fig:SHPconstant}. For the Higher Order Tuner we choose
    \begin{equation}\label{e:SHP_betagamma}
        \begin{split}
            \mu &= 10^{-5},\\
            \beta &= 0.1,\\
            \gamma &= \frac{\beta(2-\beta)}{16+\beta^2+\mu\left(\frac{57\beta+1}{16\beta}\right)} = 0.01186.
        \end{split}
    \end{equation}
    For the four other methods shown in Figure \ref{fig:SHPsimulations}, described by Eqs. \eqref{e:Gradient_Method}, \eqref{e:NormalizedGD_App}, \eqref{e:NesterovApp}, and \eqref{e:Nesterov_Two_Convex}, the constant parameters are chosen as $\Bar{\alpha}=\gamma\beta$ and $\Bar{\beta}=1-\beta$, as per Proposition \ref{prop:equivalence}, with the hyperparameters $\beta$ and $\gamma$ chosen as in \eqref{e:SHP_betagamma}.
    \item \textbf{Choice (b)}:  As mentioned before, the cost function here was assumed to change as described in Figure \ref{fig:SHPstep2} on page \pageref{fig:SHPsimulations}, which was enabled by allowing $a_k$ in \eqref{e:phiconstant} to change from 2 to 8.  In this case, we chose the hyperparameters for the Higher Order Tuner as
    \begin{equation} \label{e:1bmubetagamma}
        \begin{split}
            \mu &= 3.7174\cdot{} 10^{-6},\\
            \beta &= 0.0027,\\
            \gamma &= 183.62,
        \end{split}
    \end{equation}
    and the hyperparameters for the other four methods as     
    \begin{equation}\label{e:1balphabarbetabar}
        \begin{split}
            \Bar{\alpha} &= 0.4999,\\
            \Bar{\beta} &= 0.9972.
        \end{split}
    \end{equation}
    These choices were arrived at using the flow chart indicated in Figure \ref{fig:Hyperparameter_selection}. That is, using \eqref{e:phiconstant}, and the value of $a_0$, $b_0$ and $c_0$ as in \eqref{e:abcd_sims}, $\N_0$ was determined. $\theta^*$ is determined through \eqref{e:optimaltheta}, and $\theta_0=0$. With a choice of $\epsilon=10^{-3}$, the remaining parameters were calculated as in Figure \ref{fig:Hyperparameter_selection}, which led us to the hyperparameters as in \eqref{e:1bmubetagamma} and \eqref{e:1balphabarbetabar}. The selection as in Figure \ref{fig:Hyperparameter_selection} is reasonable, as it follows the guidelines in Lemma \ref{l:L_rate} which guarantees fast convergence of the Nesterov algorithm as indicated in \eqref{e:Desired_Rate}.
\end{itemize}

\begin{figure}
    \centering
         \begin{tikzpicture}[
                roundnode1/.style={circle, fill=Black!5, very thick, minimum size=7mm},
                roundnode2/.style={circle, draw=Black!60, fill=Black!5, very thick, minimum size=7mm},
                squarednode1/.style={rectangle, fill=Black!5, very thick, minimum size=5mm},
                squarednode2/.style={rectangle, draw=Black!60, fill=Black!5, very thick, minimum size=5mm},
                ]
                %Nodes
                \node[squarednode2]        (N0)        {$\N_0=1+\lVert \phi_0 \rVert_2^2$};
                \node[squarednode2]        (Lbar)       [right=of N0] {$\bar{L}=\bar{L}$-smoothness of $L_0(\theta)$};
                \node[squarednode1]        (Psi)       [below=of N0] {$\Psi=\max\{1,\N_0\lVert\theta_0-\theta^*\rVert^2\}$};
                \node[squarednode1]        (mubar)       [below =of Psi] {$\bar{\mu}=\frac{\epsilon}{\Psi}$};
                \node[squarednode2]        (epsilon)       [left=of mubar] {$\epsilon$};
                \node[squarednode1]        (Lbarupdate)       [right =of mubar] {$\bar{L}=\bar{L}+\bar{\mu}$};
                \node[squarednode1]        (kappa)       [below=of mubar] {$\kappa = \frac{\bar{L}}{\bar{\mu}}$};
                \node[squarednode1]        (alphabar)       [below=of Lbarupdate] {$\bar{\alpha}=\frac{1}{\bar{L}}$};
                \node[squarednode1]        (betabar)       [below=of kappa] {$\bar{\beta}=\frac{\sqrt{\kappa}-1}{\sqrt{\kappa}+1}$};
                \node[squarednode1]        (beta)       [below=of betabar] {$\beta=1-\bar{\beta}$};
                \node[squarednode1]        (gamma)       [right=of beta] {$\gamma=\frac{\bar{\alpha}}{\beta}$};
                
                %Lines
                %\draw[->] (N0.south) -- (Psi.north);
                \draw [->] (N0.south) -- node[midway,right=1em ]  {} (Psi.north);
                \draw [->] (epsilon.east) -- node[midway,right=1em ]{} (mubar.west);
                \draw [->] (Psi.south) -- node[midway,right=1em ]{} (mubar.north);
                \draw [->] (Lbar.south) -- node[midway,right=1em ]{} (Lbarupdate.north);
                \draw [->] (mubar.east) -- node[midway,right=1em ]{} (Lbarupdate.west);
                \draw [->] (Lbarupdate.south) -- node[midway,right=1em ]{} (kappa.east);
                \draw [->] (mubar.south) -- node[midway,right=1em ]{} (kappa.north);
                \draw [->] (Lbarupdate.south) -- node[midway,right=1em ]{} (alphabar.north);
                \draw [->] (kappa.south) -- node[midway,right=1em ]{} (betabar.north);
                \draw [->] (betabar.south) -- node[midway,right=1em ]{} (beta.north);
                \draw [->] (beta.east) -- node[midway,right=1em ]{} (gamma.west);
                \draw [->] (alphabar.south) -- node[midway,right=1em ]{} (gamma.north);
            \end{tikzpicture}
\caption{Hyperparameter selection for Choice (b)}
\label{fig:Hyperparameter_selection}
\end{figure}
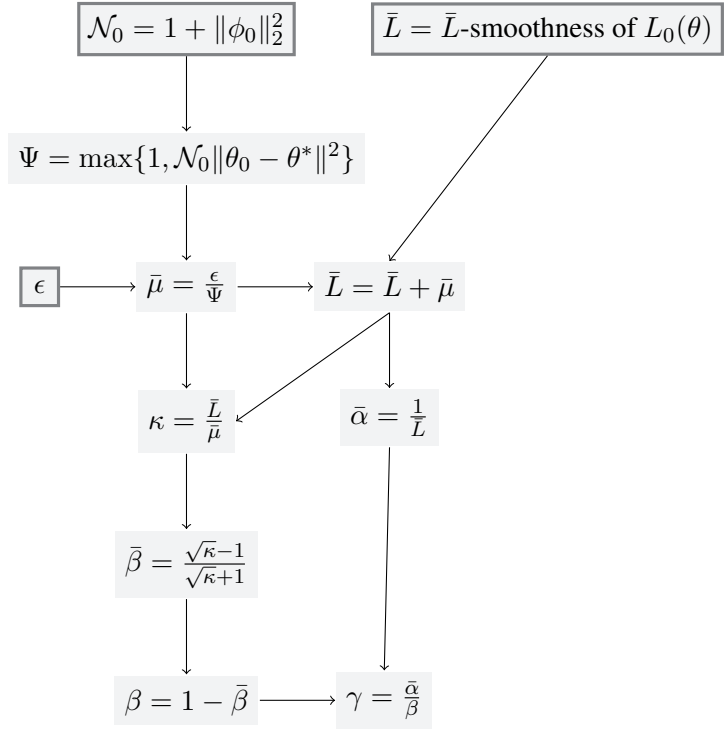{}

\subsubsection{Additional experiments}\label{SHPadd}

\paragraph{(A) Simulations:}

In the following additional experiments we show the performance of the same algorithms considered in the previous experiments, with the hyperparameters chosen as in Choice (a) and Choice (b). For each choice, we consider three scenarios:
\begin{itemize}
    \item Constant regressors simulation, with $a_k$, $b_k$, $c_k$ and $d_k$ as
    \begin{equation}\label{e:SHP_stable_constantApp}
        \begin{split}
            a_k &= 2,\\
            b_k &= c_k = d_k = 1.\\
        \end{split}
    \end{equation}
    \item Two types of step changes in regressors introduced by changing both the $\Bar{L}$-smoothness parameter of the function and $\theta^*$. For the first step change, St-I, this is accomplished by choosing:
        \begin{equation}\label{e:SHP_stable_stepApp}
        \begin{split}
            a_k &= 2,\\
            b_{k}&=\left\{ \begin{array}{cc}
                1 & \text{if }k<500\\
                2000 & \text{if }k\geq500
                \end{array}\right.,\\
            c_{k}&=\left\{ \begin{array}{cc}
                1 & \text{if }k<500\\
                1500 & \text{if }k\geq500
                \end{array}\right.,\\
            d_k &= 1.\\
        \end{split}
    \end{equation}
    Step change, St-II, is defined by the selections
    \begin{equation}\label{e:SHP_aggressive_stepApp}
        \begin{split}
            a_k &= 2,\\
            b_{k}&=\left\{ \begin{array}{cc}
                1 & \text{if }k<500\\
                2 & \text{if }k\geq500
                \end{array}\right.,\\
            c_{k}&=\left\{ \begin{array}{cc}
                1 & \text{if }k<500\\
                1.5 & \text{if }k\geq500
                \end{array}\right.,\\
            d_k &= 1.\\
        \end{split}
    \end{equation}
    \item Sinusoidal change in the regressors modifying the $\Bar{L}$-smoothness parameter of the function while keeping $\theta^*$ constant. Here too, there are two types of changes. Type Sn-I is defined by:
    \begin{equation}\label{e:SHP_stable_sinusoidalApp}
        \begin{split}
            a_k &= 8000 - 7998\sin(0.01k+\pi/2),\\
            b_k &= c_k = d_k = 1,\\
        \end{split}
    \end{equation}
        and type Sn-II is defined by
    \begin{equation}\label{e:SHP_aggressive_sinusoidalApp}
        \begin{split}
            a_k &= 5 - 3\cos(0.01k),\\
            b_k &= c_k = d_k = 1.\\
        \end{split}
    \end{equation}
\end{itemize}

Figure \ref{fig:SHP_stable_simulationsApp} shows the results for these three scenarios with the hyperparameters chosen as in Choice (a). Each sub-figure shows the error between $L_k(\theta_k)$ and $L_k(\theta^*)$ divided by $\bar{L}$. For the constant case, in Figure \ref{fig:SHP_stable_constantApp}, the Higher Order Tuner converges slowly compared to both Nesterov methods, but faster than the simple gradient descent methods. It is important to note that the Normalized Gradient Descent and the Higher Order Tuner are slower than the corresponding Gradient Descent and Nesterov Acceleration with constant $\Bar{\beta}$ which is due to the introduction of the normalization in the learning rate. In Figure \ref{fig:SHP_stable_stepApp} $a_k$, $b_k$, $c_k$ and $d_k$ vary as in \eqref{e:SHP_stable_stepApp}. This change, not only modifies the $\Bar{L}$-smoothness parameter of the function, but also changes the optimal solution $\theta^*$ of the function, which explains the jump in the graph. Similar to the case considered in Figure \ref{fig:SHPsimulations}, these variations in regressors cause the unnormalized methods to become unstable. Lastly, Figure \ref{fig:SHP_stable_sinusoidalApp} correspond to the regressors as in \eqref{e:SHP_stable_sinusoidalApp}. We can see how both Nesterov methods become unstable in this case even before the Gradient Descent, and that the normalized methods remain stable. 

\begin{figure}[ht]
\centering
\begin{subfigure}{.5\textwidth}
  \centering
  \includegraphics[width=1\linewidth]{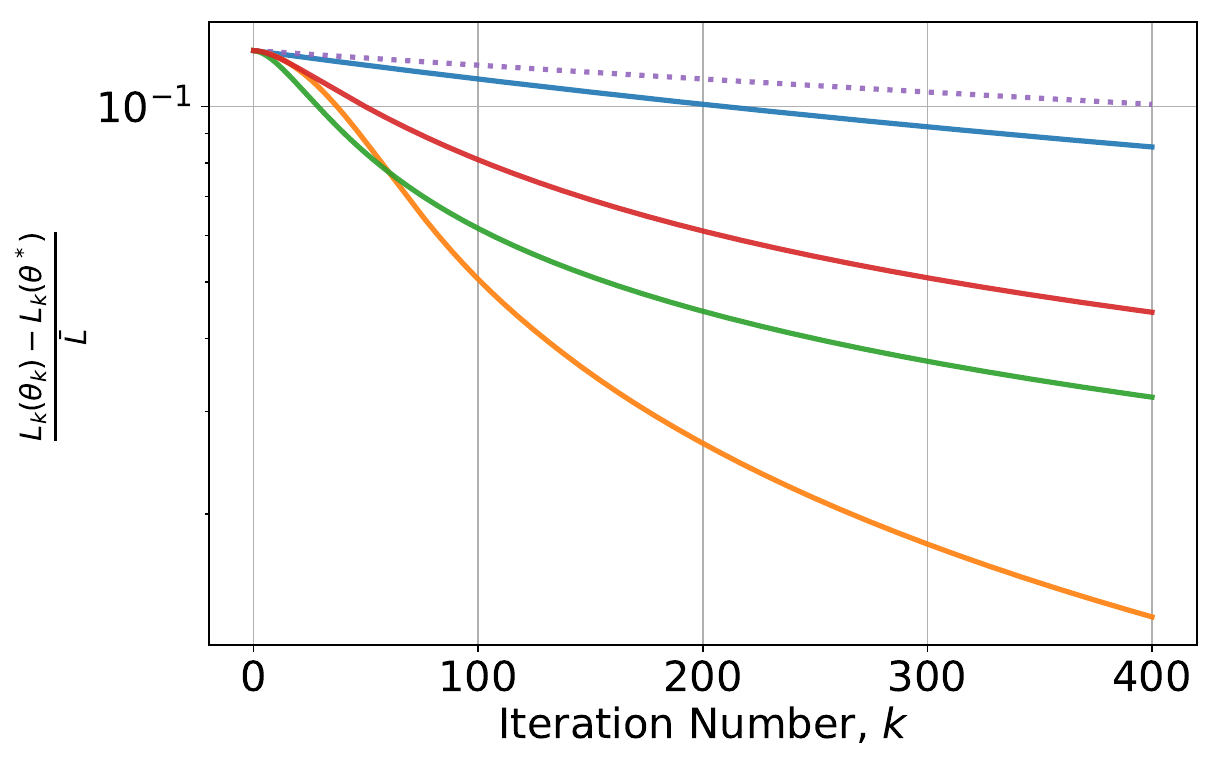}
  \caption{}
  \label{fig:SHP_stable_constantApp}
\end{subfigure}%
\begin{subfigure}{.5\textwidth}
  \centering
  \includegraphics[width=1\linewidth]{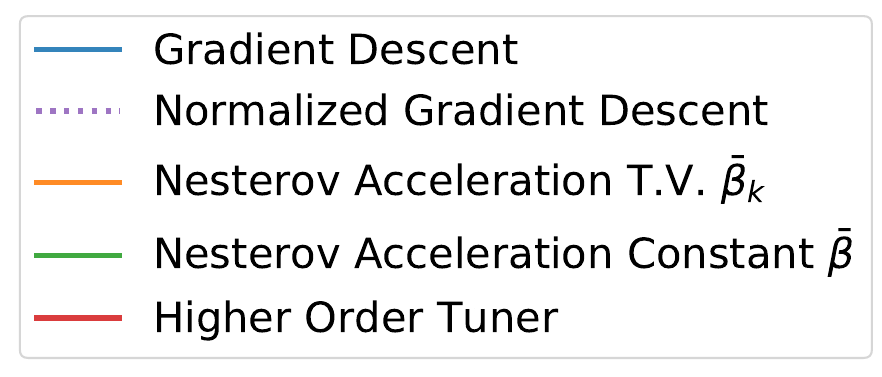}
\end{subfigure}
\begin{subfigure}{.5\textwidth}
  \centering
  \includegraphics[width=1\linewidth]{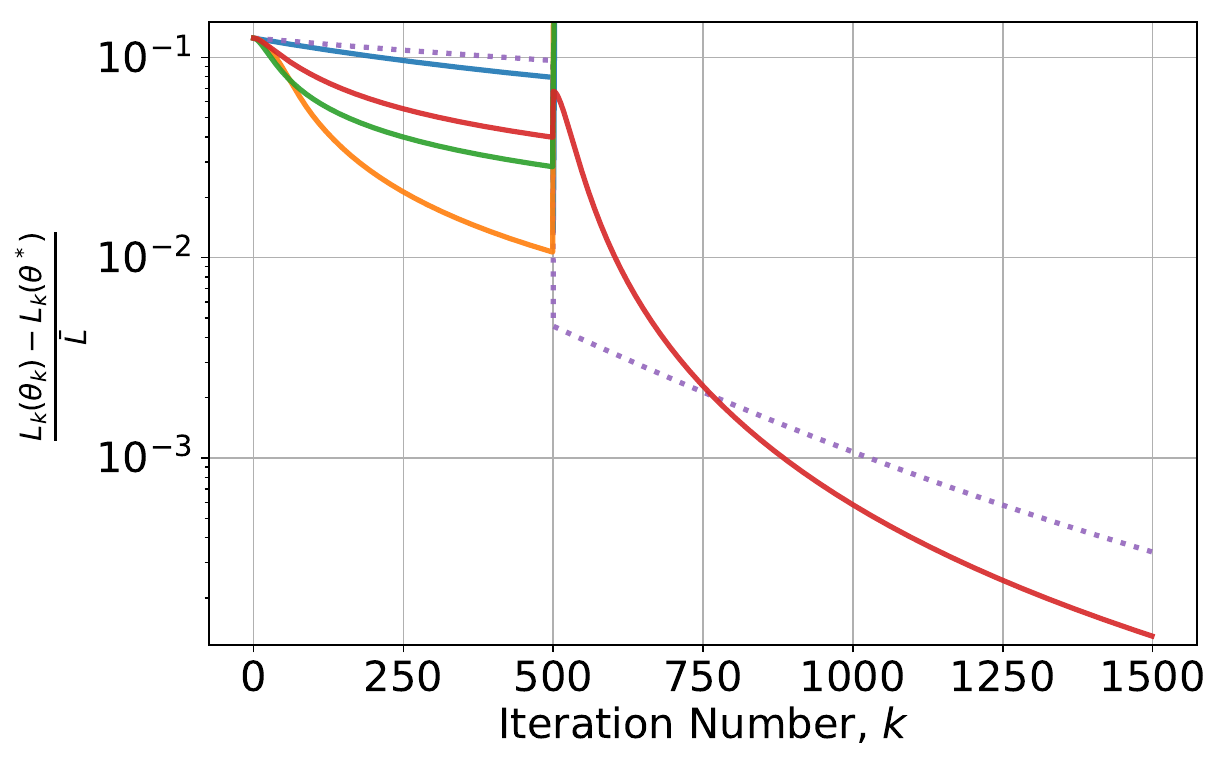}
  \caption{}
  \label{fig:SHP_stable_stepApp}
\end{subfigure}%
\begin{subfigure}{.5\textwidth}
  \centering
  \includegraphics[width=1\linewidth]{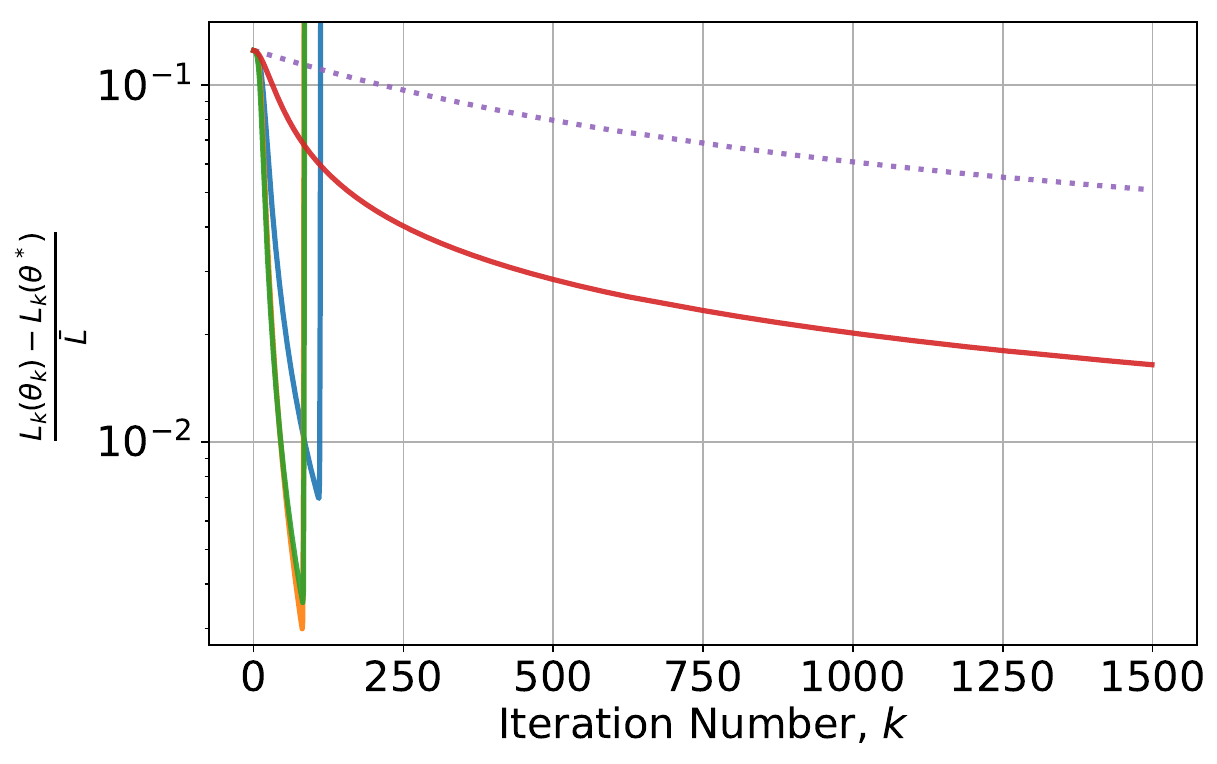}
  \caption{}
  \label{fig:SHP_stable_sinusoidalApp}
\end{subfigure}
\caption{Smooth convex minimization hyperparamters chosen as in Choice (a). (a) Constant regressors as in \eqref{e:SHP_stable_constantApp}. (b) Step change in regressors as in \eqref{e:SHP_stable_stepApp}.  (c) Sinusoidal time-varying regressors as in \eqref{e:SHP_stable_sinusoidalApp}}.
\label{fig:SHP_stable_simulationsApp}
\end{figure}

Figure \ref{fig:SHPsimulationsApp} shows the results for the same three scenarios with the hyperparameters chosen as in Choice (b). It is important to note that the modification of regressors in these experiments is smaller than the modifications from the previous case. Figure \ref{fig:SHPconstantApp} corresponds to constant regressors given by \eqref{e:SHP_stable_constantApp}, where the performance of each algorithm as well as the lower and upper bounds for the corresponding methods are shown. Here we can see that the gradient descent and normalized gradient descent algorithms have slower convergence rates than the other accelerated learning algorithms. Also, Nesterov Acceleration with time-varying $\bar{\beta}_k$ algorithm achieves the fastest performance, followed by the Nesterov with constant $\bar{\beta}$ and then the Higher Order Tuner. In Figure \ref{fig:SHPstepApp} the regressors change as in \eqref{e:SHP_aggressive_stepApp}. This produces instability in the unnormalized methods except for the simple Gradient Descent method.  Lastly, in Figure \ref{fig:SHP_stable_sinusoidalApp} the regressors change in a sinusoidal way as in \eqref{e:SHP_aggressive_sinusoidalApp}. We can see how both Nesterov methods become unstable before the Gradient Descent, and that the normalized methods remain stable. 

\begin{figure}[ht]
\centering
\begin{subfigure}{.5\textwidth}
  \centering
  \includegraphics[width=1\linewidth]{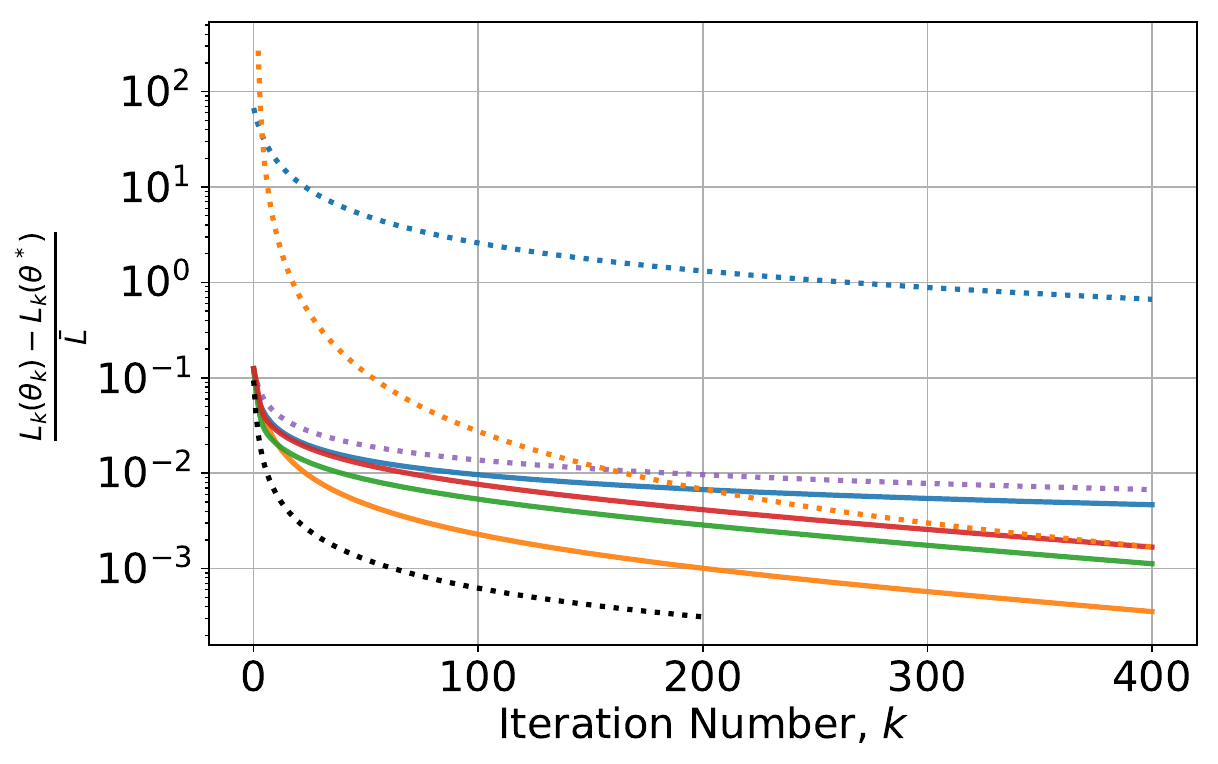}
  \caption{}
  \label{fig:SHPconstantApp}
\end{subfigure}%
\begin{subfigure}{.5\textwidth}
  \centering
  \includegraphics[width=1\linewidth]{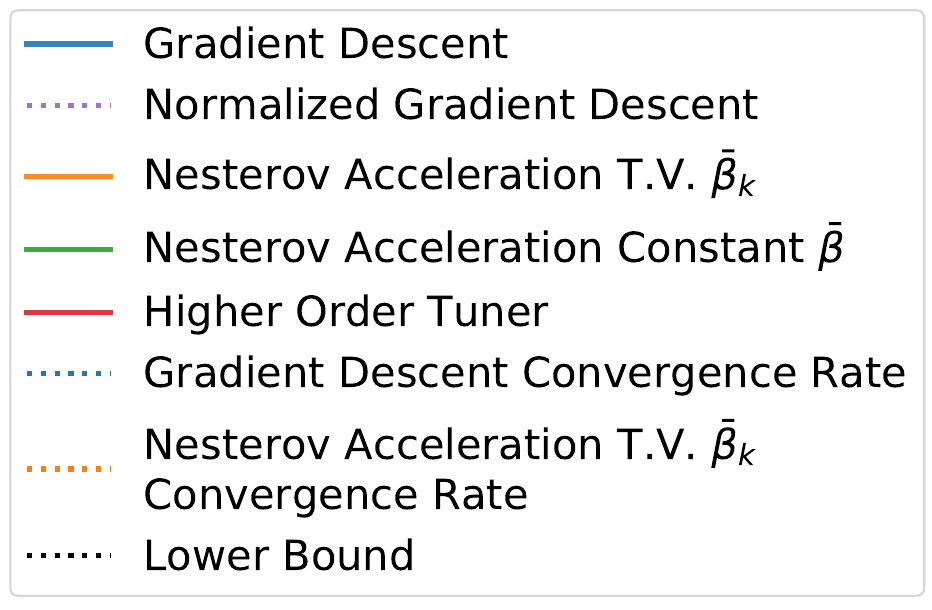}
\end{subfigure}
\begin{subfigure}{.5\textwidth}
  \centering
  \includegraphics[width=1\linewidth]{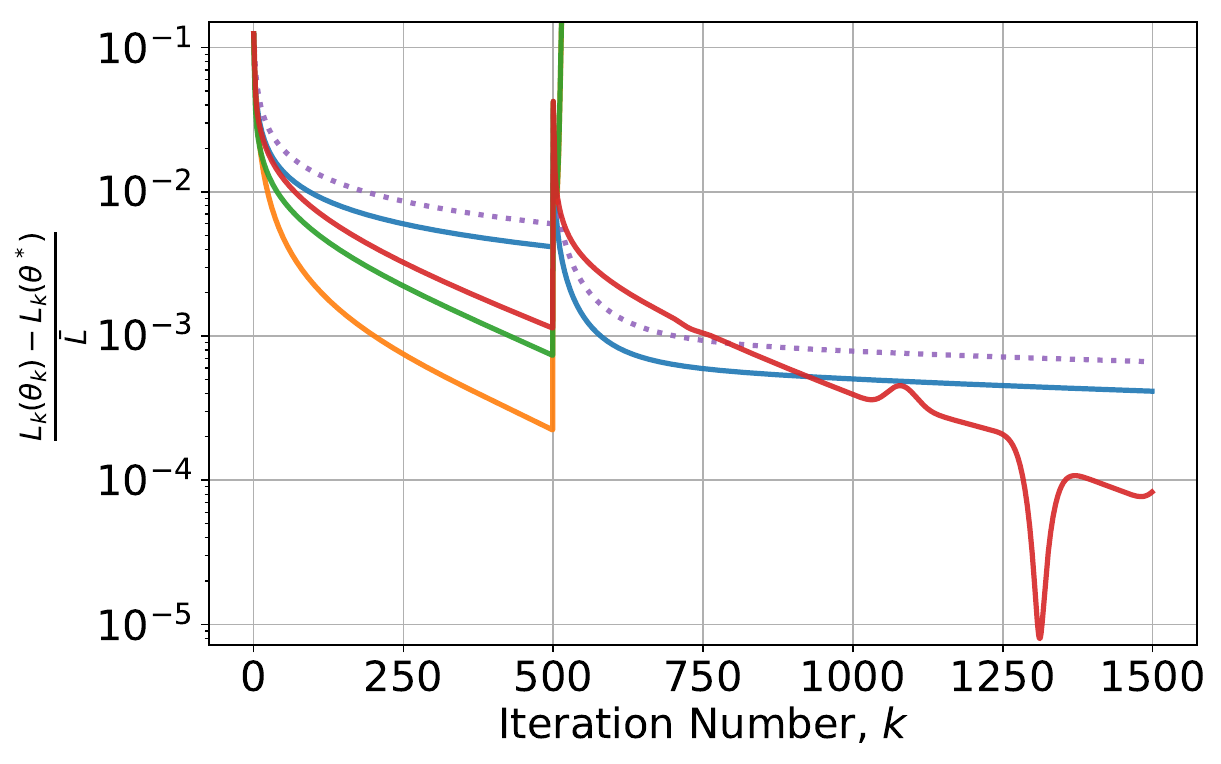}
  \caption{}
  \label{fig:SHPstepApp}
\end{subfigure}%
\begin{subfigure}{.5\textwidth}
  \centering
  \includegraphics[width=1\linewidth]{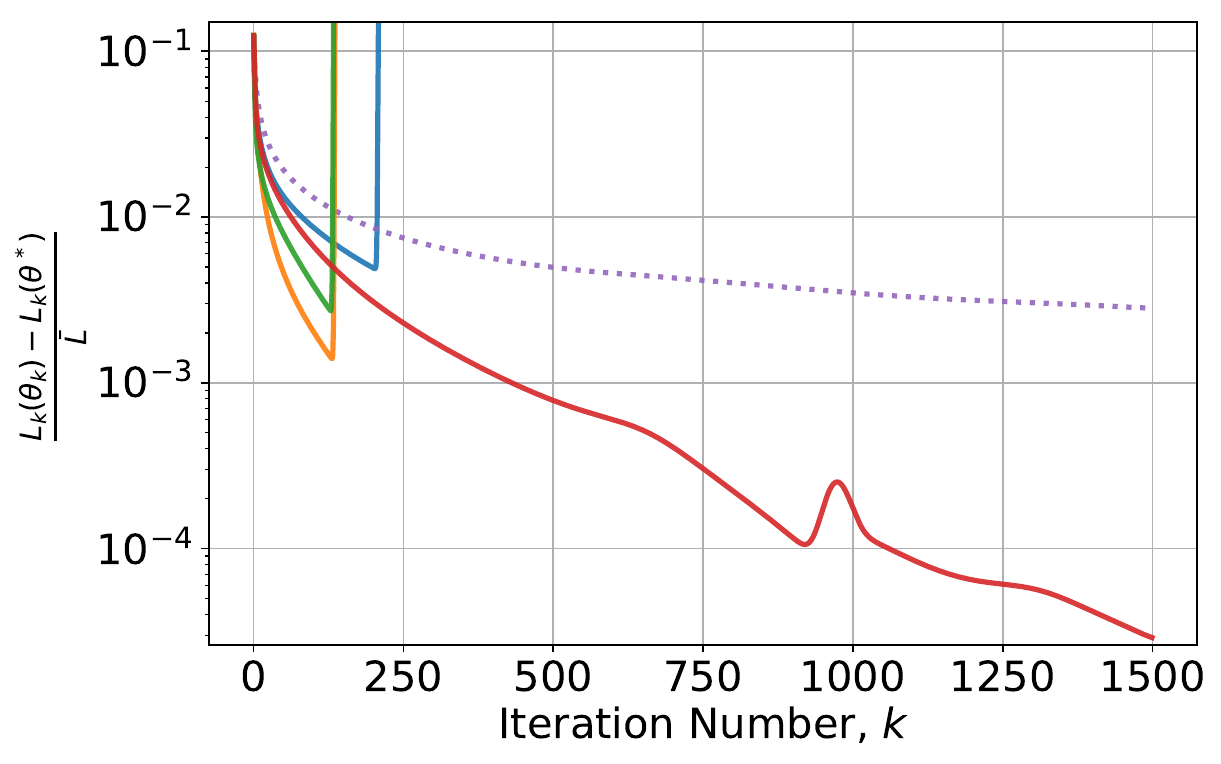}
  \caption{}
  \label{fig:SHPsinusoidalApp}
\end{subfigure}
\caption{Smooth convex minimization with hyperparameters chosen as in Choice (b). (a) Constant regressors as in  \eqref{e:SHP_stable_constantApp}.  (b) Step change in regressors as in \eqref{e:SHP_aggressive_stepApp}. (c) Sinusoidal time-varying regressors as in \eqref{e:SHP_aggressive_sinusoidalApp}. }
\label{fig:SHPsimulationsApp}
\end{figure}

\paragraph{(B) Convergence rates and lower bound:} \label{paragraph:LowerBound}
For the simulations done with the hyperparameters as in Choice (b), and for the constant regressors case (Figure \ref{fig:SHPconstantApp}), we plot the lower complexity bound and the convergence rates for the following algorithms:

We plot the convergence rates obtained from the corresponding equations of each theorem:
\begin{itemize}
    \item Gradient Descent Upper Bound: Theorem \ref{theorem:GD_k}, Equation \ref{e:GD_SmoothConvex_convergence}.
    \item Nesterov Acceleration with time-varying $\bar{\beta}_k$ Upper Bound: Theorem \ref{theorem:Nesterov_Convex}, equation \ref{e:Nesterov_convergence}.
\end{itemize}{}

For the lower bound, we make use of a theorem from \citep[p.~71]{Nesterov_2018}. This is done by pointing out a hard function to minimize for iterative schemes satisfying the following assumption:
\begin{assumption}\label{ass:iterativemethod}
    (Modified from \citep[p.~69]{Nesterov_2018}) An iterative method $\mathscr{M}$ generates a sequence of test points ${\theta_k}$ such that
    \begin{equation}
        \theta_k \in \theta_0 + \textit{Lin} \left\{ \nabla L_k(\theta_0),...,\nabla L_k(\theta_{k+1}) \right \}, \quad k \geq 1.
    \end{equation}{}
\end{assumption}{}

\begin{theorem}\label{theorem:NesterovLowerRate}
(Modified from \citep[p.~71]{Nesterov_2018}) For any $k, 1\leq k \leq \frac{1}{2}(n-1)$, and any $\theta_0 \in \Re^n$ there exists a function $f \in \mathscr{F}_L^{\infty,1}(\mathbb{R}^n)$ such that for any first-order method $\mathscr{M}$ satisfying Assumption \ref{ass:iterativemethod} we have
\[f(\theta_k)-f^* \geq \frac{3\bar{L}\left\Vert \theta_0-\theta^{*}\right\Vert^2}{32(k+1)^2} ,\]
\[\left\Vert \theta_k-\theta^{*}\right\Vert^2 \geq \frac{1}{8} \left\Vert \theta_0-\theta^{*}\right\Vert^2, \]
where $\theta^*$ is the minimum of the function $f$ and $f^* =f(\theta^*)$.
\end{theorem}{}
\begin{proof} Refer to \citep[p.~72]{Nesterov_2018}.
\end{proof}

\subsection{Image deblurring problem}\label{ImageDeblurringStatement}
The image deblurring problem is a Linear Inverse Problem which has been studied previously in \citep{Hansen_2006,Beck_2009,Beck_2009_TIP}. We simulate this problem to show the advantages of the proposed method through improved deblurring in the images.

In this problem, we consider gray-scale images, with an image represented by a matrix of size $n$ by $m$. This matrix can be reshaped as a vector concatenating each one of the columns producing a column vector of size $nm$. If the image is in a matrix form, it will be denoted with an upper case letter; if it is reshaped as a vector, it will be denoted with a lower case letter.

We consider the non-blind deblurring problem framework \citep{WangTao}, which consists on obtaining the best estimate of the original image $X$, given its blurry version $B$, and the known blur operator $A$. This is expressed with the following linear model:
\begin{equation}\label{e:IDP_constant}
Ax=b
\end{equation}
where $x \in \mathbb{R}^{nm}$ is the reshaped actual image $X$, $b \in \mathbb{R}^{nm}$ is the reshaped blurry image $B$, and $A \in \mathbb{R}^{nm\text{x}nm}$ is the blur operator.

Apart from the blurred image $b$, which is measurable, the matrix $A$ is assumed to be known from apriori external information. Blurry images can be caused by physical or mechanical processes, such as: out-of-focus lens, moving while taking the photo, or due to air variations which affect the light going through the lens \citep{Hansen_2006}, in addition to malicious cyber attacks. 

We are interested on the time-varying blur scenario, where we are given a set of blurry images and we want to obtain the actual image. Such time-variations can occur in a fixed security camera with a constant actual image, due to fog or smoke, or an out-of-focus arrangement of the camera lens. This causes \eqref{e:IDP_constant} to be modified as
\begin{equation}
A_kx=b_k
\end{equation}

In order to build the blur operator $A$, we only require the Point Spread Function (PSF) of the blur and the boundary conditions \citep{Hansen_2006}:
\begin{itemize}
    \item The PSF determines how a single bright pixel is blurred and how its neighbors are affected. The blurry image is obtained by performing the 2-dimensional convolution of the PSF over the actual image. Typically, if the blur is considered a local phenomenon and space invariant, the PSF can be expressed through a smaller array $P$. Assuming that all the light is captured by the camera, the values in the PSF should sum to 1 \citep{Hansen_2006}.
    \item The boundary conditions detail how the image would look like out of the frame. The boundary conditions could be zero, periodic or reflexive. For this specific experiment we will consider periodic boundary conditions.
\end{itemize}

For the experiments, we will consider Gaussian blur. The PSF array P in such case is defined through the following formula from \citep[p~26]{Hansen_2006}:
\begin{equation}\label{e:psfGauss}
    p_{ij}=\exp\left(-\frac{(i-k)^{2}}{2\sigma^{2}}-\frac{(j-l)^{2}}{2\sigma^{2}}\right),
\end{equation}
where $(k,l)$ is the center of the kernel, and $\sigma$ is the standard deviation of the Gaussian distribution. All the elements in $P$ must be scaled such that the elements sum to 1. In these experiments, we will maintain the kernel size constant, and we will vary $\sigma_k$. Therefore, the PSF array $P_k$ is defined by $P_k=\text{psfGauss}(\sigma_k)$, where $\text{psfGauss}$ denotes the corresponding steps to compute \eqref{e:psfGauss} for a kernel size of $9$, and scaling $P$ appropriately.

In order to obtain the best estimate of $x$, one can minimize the following loss function:
\begin{equation} \label{e:idp_notreg}
\underset{x}{\mathrm{argmin}}\left[\frac{1}{2}||Ax-b||_2^2 \right] 
\end{equation}

Typically, a regularizing term is added to Equation \ref{e:idp_notreg}. This is because $A$ may be ill-conditioned \citep{Hansen_2006}, and for certain situations where noise might be present. Thus, adding the regularization term helps to stabilize the solution. Different types of regularization can be applied \citep{Hansen_2006,Beck_2009,Beck_2009_TIP}. Much of this paper is focused on the noise-free case except for one short experiment discussed at the end of this document with a few preliminary results.

A more computationally efficient way to solve this problem (with periodic boundary conditions) is based on the Fourier transform \citep{Hansen_2006}. In this domain, the blur operation is a simple element-wise matrix multiplication of the blur operator $\phi_k$ and the transformed image $\theta^*$, i.e.:  $\phi_k \odot \theta^*$. The blur operator, $\phi_k$, is obtained directly from the PSF array $P$, by applying the function from \citep[p.~43]{Hansen_2006}, which obtains the spectrum of the matrix $A$ without building that matrix, nor computing its eigenvalues, and using only the information on the PSF array P. Note that the PSF array $P$ needs to be reshaped to the image size by padding with $0$ values. In this paper, we will denote this function as
\begin{equation}
    \phi_k = \text{blur\_operator} (P_k) = \text{blur\_operator}(\text{psfGauss}(\sigma_k)).
\end{equation}
Using the frequency domain translates to the following relationships: $X \leftrightarrow \theta^*$ and $B_k \leftrightarrow y_k$, where $\leftrightarrow$ represents the direct and inverse 2D Fourier transformation. 

We represent the Image Deblurring problem in the Fourier domain as:
\begin{equation} 
\theta=\underset{\theta}{\mathrm{argmin}}\left[\frac{1}{2}\lVert \phi_k \odot \theta-y_k\rVert_{fro}^2 \right].
\end{equation}

Computationally the operations will be performed in the element-wise form. In spite of this, and in order to be consistent with the math notation used in this paper, we will consider $\theta$ and $y_k$ as column vectors of size $nm$, and $\phi_k$ will be reshaped as a diagonal matrix of size $nm$ by $nm$. Then, the problem is rewritten as:
\begin{equation} \label{e:L1_reg}
\theta=\underset{\theta}{\mathrm{argmin}}\left[\frac{1}{2}\lVert \phi_k  \theta-y_k\rVert_2^2 \right] .
\end{equation}

An important consequence from the condition that all PSFs should satisfy (all elements sum to 1 for any blur intensity) is that
\begin{equation}
    \lVert \phi_k \rVert_2^2=1, \quad \forall k.
\end{equation}
Because of this, the case where the blur is time-varying might not affect the stability of unnormalized algorithms. However, we consider an alternative scenario where the information on $y_k$ and $\phi_k$ might be corrupted due to some external issue or disturbance on the communication system. In these circumstances, the relation $y_k = \phi\theta^*$ would be satisfied, but $\lVert \phi_k \rVert_2^2$ could vary its value. We represent this disturbance on the information with the scalar $\delta_k$. Thus, the blur operator $\phi_k$ is redefined as
\begin{equation}
    \phi_k = \delta_k \odot \text{blur\_operator} (P_k) = \delta_k \odot \text{blur\_operator}(\text{psfGauss}(\sigma_k)).
\end{equation}

\subsubsection{Details of experiments in the main paper}\label{IDPProblemStatement}

As indicated in Section \ref{sec:IDP}, Figure \ref{fig:IDPsimulations} represents the results of three different algorithms with an specific choice of hyperparameters on each subfigure. The scenario is the same in all cases where the blurry image $y_k$ is generated through a constant Gaussian PSF array P of size 11 as in \eqref{e:psfGauss} with $\sigma_k=7$, and a time-varying $\delta_k$ as
\begin{equation}
    \delta_{k} = \left\{ \begin{array}{cc}
                1 & \text{if }k<500\\
                \frac{199}{200}k-496.5 & \text{if }500 \leq k < 700 \\
                200 & \text{if }k\geq700
                \end{array}\right..\\
\end{equation}

We now proceed to a discussion of the algorithms and hyperparameters.

\paragraph{(A) Algorithms}
We consider the following methods listed in the same order as in the legend in Figure \ref{fig:IDPsimulations}.
\begin{itemize}
    \item Gradient Descent Method \eqref{e:Gradient_Method}. With $\theta_0$ chosen as the blurred image in the first iteration, i.e. $\theta_0=y_0$; since this is the best estimate of $\theta^*$ at the beginning.
    \begin{equation*}\label{e:GD_IDP}
    \begin{split}
        \theta_{k+1}&= \theta_k - \Bar{\alpha} \phi_k(\phi_k^T\theta_k-y_k)  \\
    \end{split}{}
    \end{equation*}{}
    \item Nesterov Acceleration with time-varying $\bar{\beta}_k$ (\eqref{e:Nesterov_Two_Convex_TV_beta} with $\bar{\beta}_k$ chosen as in \eqref{e:beta_k}). With $\theta_0$ chosen consistently with the previous method, and $\nu_0=\theta_0$.
    \begin{equation*}\label{e:NEST_IDP}
    \begin{split}
        \theta_{k}&= \nu_k - \Bar{\alpha} \phi_k(\phi_k^T\nu_k-y_k) \\
        \nu_{k+1}&=(1+\bar{\beta}_k)\theta_{k+1}-\Bar{\beta}_k\theta_k
    \end{split}{}
    \end{equation*}{}
    \item Higher Order Tuner (Algorithm \ref{alg:HOT_R}). Again, $\theta_0=\vartheta_0=y_0$ to be consistent with the previous methods.
    \begin{equation*}
        \begin{split}
        \bar{\theta}_{k}&= \theta_k - \gamma\beta \left(\frac{\phi_k(\phi_k^T\theta_k-y_k)}{\N_k}+\mu(\theta_k-\theta_0)\right) \\
        \theta_{k+1} &= \bar{\theta}_{k} -\beta (\bar{\theta}_{k}-\vartheta_k) \\
        \vartheta_{k+1}&=\vartheta_k-\gamma\left(\frac{\phi_k(\phi_k^T\theta_{k+1}-y_k)}{\N_k}+\mu(\theta_{k+1}-\theta_0)\right)
        \end{split}{}
    \end{equation*}{}  
\end{itemize}

\paragraph{(B) Hyperparameter selection:}
For this experiment,as mentioned in Section \ref{sec:IDP}, we have four choices for the hyperparameters as opposed to two in \ref{SmoothConvexMinProblem}.

\begin{itemize}
    \item \textbf{Choice (1)}: This corresponds to Figure \ref{fig:IDP_stable}. Here, we choose the hyperparameters of the Higher Order Tuner satisfying Theorem \ref{th:HOT_paper_Full_Alg}, with
    \begin{equation}\label{e:IDP_betagamma}
        \begin{split}
            \mu &= 10^{-20},\\
            \beta &= 0.1,\\
            \gamma &= \frac{\beta(2-\beta)}{16+\beta^2+\mu\left(\frac{57\beta+1}{16\beta}\right)} = 0.01186.
        \end{split}
    \end{equation}
    For the two other methods (Gradient Descent and Nesterov Acceleration with time-varying $\Bar{\beta}_k$) the constant parameter is chosen as $\Bar{\alpha}=\frac{\gamma\beta}{\N_0}=0.00059$ where $\N_0$ corresponds to the first $\phi_k$.
    \item \textbf{Choice (2)}: This corresponds to Figure \ref{fig:IDP_stable_SUP}. The hyperparameters for the Higher Order Tuner are chosen as in \eqref{e:IDP_betagamma} and $\Bar{\alpha}=\frac{\gamma\beta}{\max \N_k}=2.966\cdot10^{-8}$, where $\max \N_k$ corresponds to the maximum value of all future $\phi_k$.
    \item \textbf{Choice (3)}: This corresponds to Figure \ref{fig:IDP_aggresive}. We choose the learning rate $\Bar{\alpha}$ as the maximum value for which convergence is guaranteed for constant regressors \citep{Beck_2009}. In this case, we only know the first $\phi_k$, so this value corresponds to $\Bar{\alpha}=1/\lVert \phi_0 \rVert^2_2=1$. Then, for the Higher Order Tuner, we define comparable hyperparameters by choosing $\gamma\beta=1$ with
    \begin{equation} \label{e:IDP_gammabeta34}
        \begin{split}
            \mu &= 10\cdot10^{-20},\\
            \beta &= 0.1,\\
            \gamma &= 10. \\
        \end{split}
    \end{equation}
    \item \textbf{Choice (4)}: This corresponds to Figure \ref{fig:IDP_aggresive_SUP}. In this case, we know all $\phi_k$, so $\Bar{\alpha}=\frac{1}{\lVert \max \phi_k \rVert^2_2}=2.5\cdot10^{-5}$. Then, for the Higher Order Tuner $\mu$, $\beta$ and $\gamma$ are chosen as in \eqref{e:IDP_gammabeta34}.
\end{itemize}

\subsubsection{Experiments in the presence of time-varying blur and noise}\label{IDPadd}
\paragraph{(A) Simulations:}
In Figure \ref{fig:SinusoidalBlur} we show the performance of the Higher Order Tuner for constant $\delta_k$ and time-varying blur. We modify the blur intensity by modifying $\sigma$ in a sinusoidal form. The hyperparameters are chosen as in Choice (3). Note that these modifications on $\sigma_k$ do not modify the $\lVert \phi_k\rVert^2_2$, and the other methods remain stable in this scenario.

\begin{figure}[b]
        \begin{subfigure}[b]{0.25\textwidth}
            \centering
            \includegraphics[width=\textwidth]{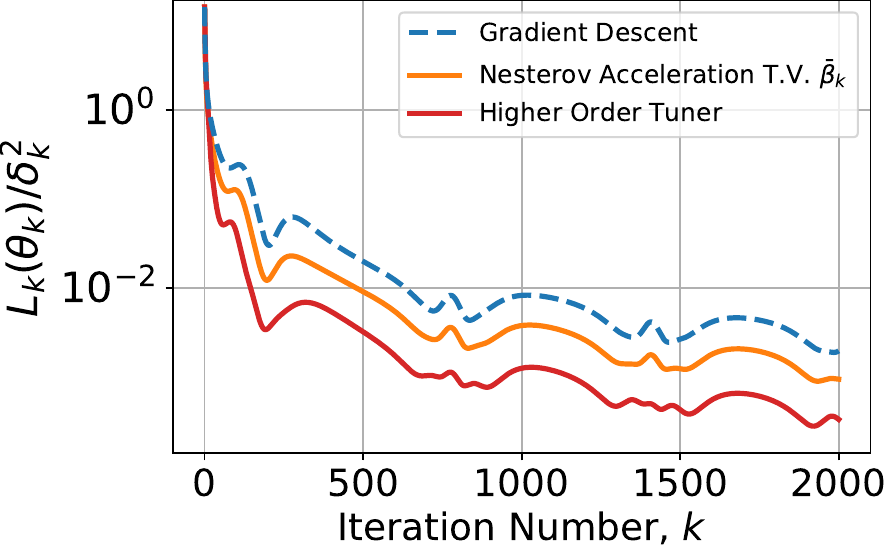}
            \caption{} 
        \end{subfigure}
        \begin{subfigure}[b]{0.74\textwidth}
            \centering
            \includegraphics[width=\textwidth]{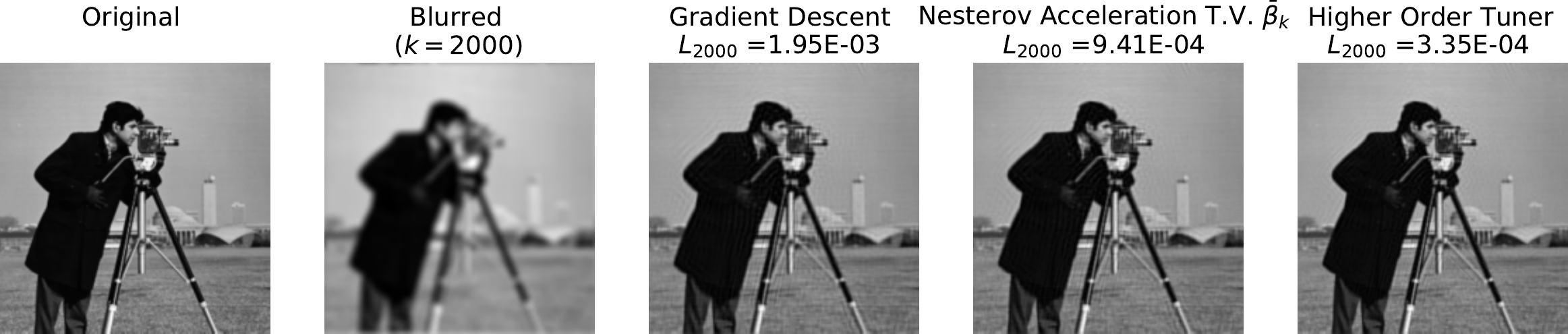}
            \caption{} 
        \end{subfigure}
        \caption{Image deblurring problem for time-varying blur, constant $\delta_k=1$, and hyperparameters as in Choice (3). Blur is defined with a Gaussian blur PSF with a constant kernel size of 11 and $\sigma_k=7-4.1\sin(0.01k)$. This produces $\phi_k=\text{blur\_operator}(\text{psfGauss}(\sigma_k))$. (a) Loss values. (b) Original, blurry and reconstructed images for each different method. }
        \label{fig:SinusoidalBlur}
\end{figure}

Lastly, we show the performance of the Higher Order Tuner when there is noise present in the blurry image. A random noise is added to the blurry image as $N(0,0.02)$. The Higher Order Tuner was implemented for the case when $\mu=0$ and $\mu=0.1$, and the hyperparameters $\gamma$ and $\beta$ as in Choice (3).  It can be seen from Figure \ref{fig:Noisy} that the tuner results in better results with $\mu=0.1$ rather than $\mu=0$. In addition, an $l_1$ regularization term with a weighting parameter $\lambda$ as in \citep{Beck_2009} can be added as well to the loss function. The deblurring problem in this setting can be rewritten as
\begin{equation} \label{e:L1_reg_theta}
    \theta=\underset{\theta}{\mathrm{argmin}}\left[\frac{1}{2}\lVert \phi_k  \theta-y_k\rVert_2^2 + \lambda \lVert \theta \rVert_1 \right].
\end{equation}
It was observed that with the addition of such a $\lambda$, the performance of the high-order tuner improved even further compared to Figure x. As a theoretical analysis of this case remains to be analyzed, we do not present these figures in this paper.

\begin{figure}
\centering
        \begin{subfigure}[b]{0.74\textwidth}
            \centering
            \includegraphics[width=\textwidth]{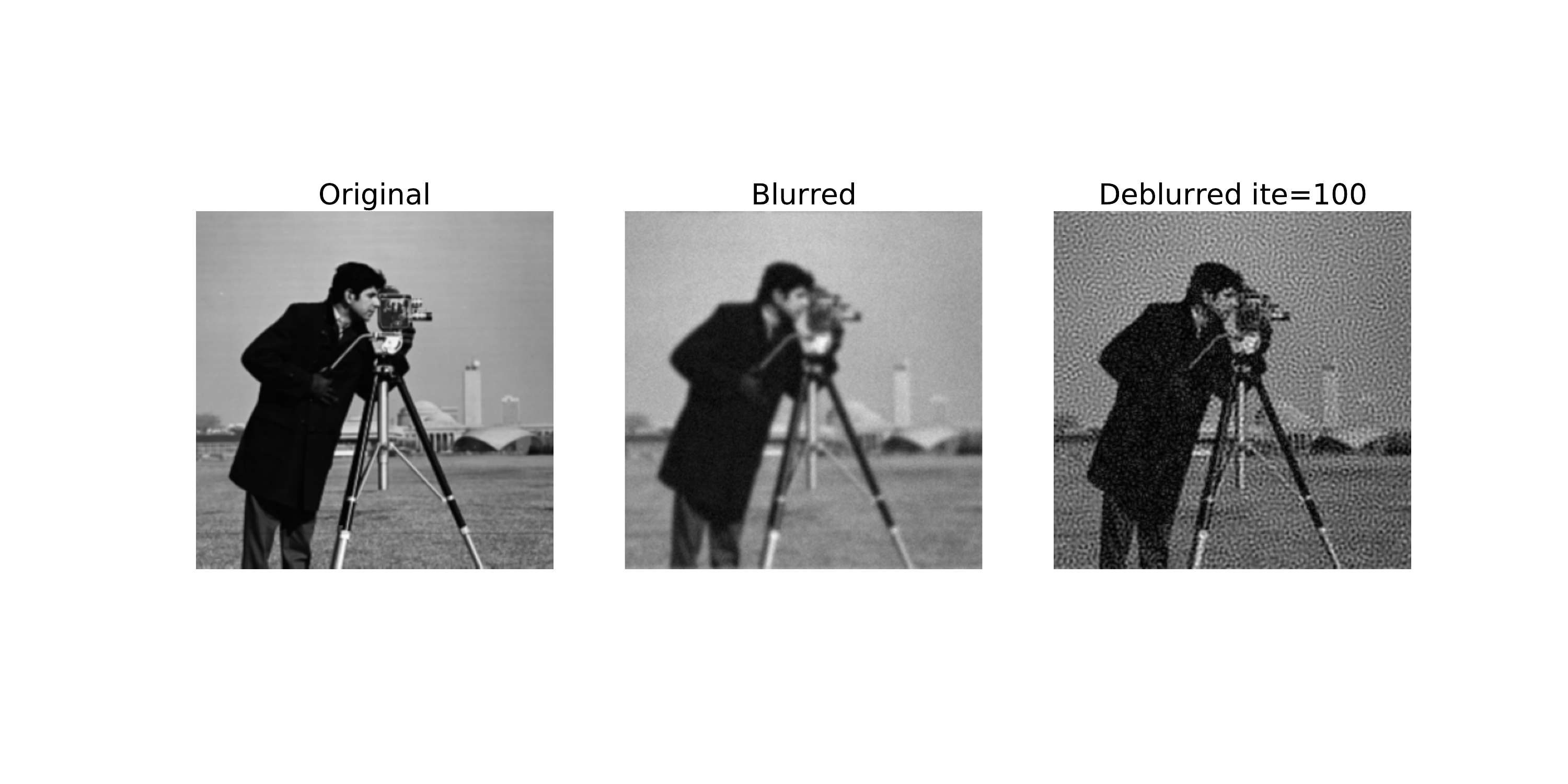}
            \caption{}
        \end{subfigure}
        \begin{subfigure}[b]{0.74\textwidth}
            \centering
            \includegraphics[width=\textwidth]{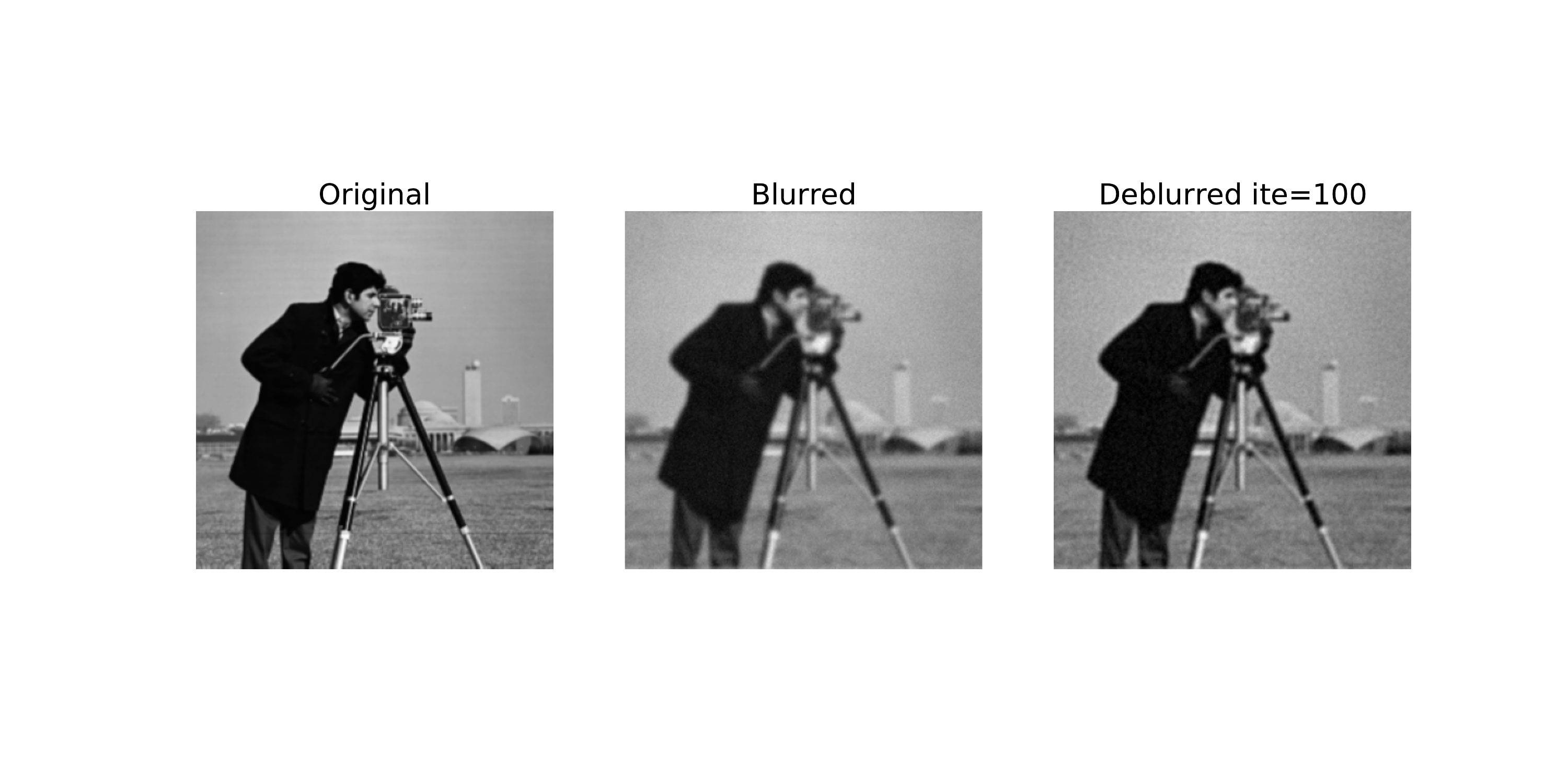}
            \caption{}
        \end{subfigure}
        \caption{Noisy image deblurring problem for constant blur and constant $\delta_k=1$, $\gamma$ and $\beta$ as in Choice (3). Random noise is $N(0,0.02)$. a) $\mu=0$ b) $\mu=0.1$}
        \label{fig:Noisy}
\end{figure}

\end{document}